\documentclass[letterpaper,11pt]{article}
\usepackage[margin=1in]{geometry}
\usepackage[bookmarks, colorlinks=true, plainpages = false, colorlinks=true,
            citecolor=blue,
            linkcolor=blue,
            anchorcolor=red,
            urlcolor=blue]{hyperref}
\usepackage{url}
\usepackage{amsmath, amsfonts,amsthm,amssymb,bm, verbatim,dsfont,mathtools}
\usepackage{color,graphicx,appendix}
\usepackage{etoolbox}

\usepackage{array}
\usepackage{multirow}

\usepackage{float}

\usepackage{tikz}
\usetikzlibrary{arrows.meta}
\usetikzlibrary{matrix,arrows,calc,shapes,backgrounds}
\usetikzlibrary{shapes.callouts,decorations.text}

\usetikzlibrary{shapes.misc}

\tikzset{cross/.style={cross out, draw=black, minimum size=2*(#1-\pgflinewidth), inner sep=0pt, outer sep=0pt},
cross/.default={1pt}}

\tikzstyle{int}=[draw, fill=blue!20, minimum size=2em]
\tikzstyle{dot}=[circle, draw, fill=blue!20, minimum size=2em]
\tikzstyle{dotred}=[circle, draw, fill=red!20, minimum size=2em]
\tikzstyle{init} = [pin edge={to-,thin,black}]
\tikzstyle{initred} = [pin edge={to-,thin,red}]
\tikzstyle{plan}=[draw, fill=blue!20, minimum size=2em, text width=5em, rounded corners,align=center]
\tikzstyle{planwide}=[draw, fill=blue!20, minimum size=2em, text width=8em, rounded corners,align=center]
\tikzstyle{nodedot}=[circle, draw, fill=white, minimum size=0.3cm,inner sep=0pt]
\tikzstyle{nodedot}=[circle, draw, fill=white, minimum size=3,inner sep=0pt]
\tikzstyle{Medge}=[green!60!black, thick]
\tikzstyle{Bedge}=[red, thick]
\tikzstyle{Cedge}=[blue, thick]
\tikzstyle{Sedge}=[black, thick]
\tikzstyle{Mgiantedge}=[green!60!black, line width=3.0pt]
\tikzstyle{Bgiantedge}=[red,line width=3.0pt]
\tikzstyle{Cgiantedge}=[blue,line width=3.0pt]
\tikzstyle{Sgiantedge}=[black,line width=3.0pt]
\tikzstyle{shadedgiantnode}=[circle, draw, fill=black!10, minimum size=1cm, inner sep=0pt]
\tikzstyle{unshadedgiantnode}=[circle, draw, fill=white, minimum size=1cm, inner sep=0pt]
\makeatletter
\tikzset{my loop/.style =  {to path={
  \pgfextra{}
  [looseness=5,min distance=10mm]
  \tikz@to@curve@path},font=\sffamily\small
  }}  
\makeatletter 
\newcolumntype{C}[1]{>{\centering\arraybackslash}p{#1}}

\tikzstyle{vertexdot}=[circle, draw, fill=white, minimum size=3,inner sep=0pt]
\tikzstyle{root}=[circle, draw, fill=black, minimum size=3,inner sep=0pt]

\tikzstyle{vertexdotsolid}=[circle, draw, fill=black, minimum size=3,inner sep=0pt]

\usepackage{xr,xspace}
\usepackage{todonotes}
\usepackage{paralist}
\usepackage{enumitem}

\usepackage{caption,subcaption,soul}
\usepackage{algorithm}%
\usepackage{algpseudocode}
\makeatletter
\usepackage{romanbar}

\numberwithin{equation}{section}
\newtheorem{theorem}{Theorem}[section]
\newtheorem{corollary}[theorem]{Corollary}
\newtheorem{lemma}[theorem]{Lemma}
\newtheorem{remark}[theorem]{Remark}
\newtheorem{proposition}[theorem]{Proposition}

\newtheorem{problem}[theorem]{Problem}

\newtheorem{definition}[theorem]{Definition}

\newtheorem*{customthm*}{\thistheoremname}
\newcommand{\thistheoremname}{} %

\newcommand \E[1]{\mathbb{E}[#1]}

\newcommand{\op}{\mathrm{op}}

\usepackage{xspace,prettyref}

\newcommand{\eps}{\varepsilon}

\newcommand{\iiddistr}{{\stackrel{\text{\iid}}{\sim}}}

\newcommand{\reals}{{\mathbb{R}}}

\newcommand{\Prob}{\mathbb{P}}

\def\Var{\mathrm{Var}}
\def\E{\mathbb{E}}
\newcommand{\Bern}{{\rm Bern}}
\newcommand{\Binom}{{\rm Binom}}

\newcommand{\eg}{e.g.\xspace}

\newcommand{\iid}{i.i.d.\xspace}
\newcommand{\GOE}{{\rm GOE}}
\newcommand{\Wishart}{{\rm Wishart}}

\newrefformat{eq}{(\ref{#1})}
\newrefformat{chap}{Chapter~\ref{#1}}
\newrefformat{sec}{Section~\ref{#1}}
\newrefformat{alg}{Algorithm~\ref{#1}}
\newrefformat{assump}{Assumption~\ref{#1}}
\newrefformat{fig}{Figure~\ref{#1}}
\newrefformat{tab}{Table~\ref{#1}}
\newrefformat{rmk}{Remark~\ref{#1}}
\newrefformat{clm}{Claim~\ref{#1}}
\newrefformat{claim}{Claim~\ref{#1}}
\newrefformat{def}{Definition~\ref{#1}}
\newrefformat{cor}{Corollary~\ref{#1}}
\newrefformat{lmm}{Lemma~\ref{#1}}
\newrefformat{prop}{Proposition~\ref{#1}}
\newrefformat{prob}{Problem~\ref{#1}}
\newrefformat{app}{Appendix~\ref{#1}}
\newrefformat{hyp}{Hypothesis~\ref{#1}}
\newrefformat{thm}{Theorem~\ref{#1}}
\newrefformat{fac}{Fact~\ref{#1}}
\newrefformat{ex}{Example~\ref{#1}}

\newcommand{\norm}[1]{\left\|{#1} \right\|}

\newcommand{\iprod}[2]{\left \langle #1, #2 \right\rangle}

\newcommand{\dTV}{d_{\rm TV}}
\newcommand{\dKL}{d_{\rm KL}}

\newcommand{\bbP}{{\mathbb{P}}}

\newcommand{\calC}{{\mathcal{C}}}

\newcommand{\calE}{{\mathcal{E}}}
\newcommand{\calF}{{\mathcal{F}}}
\newcommand{\calG}{{\mathcal{G}}}

\newcommand{\calL}{{\mathcal{L}}}

\newcommand{\calN}{{\mathcal{N}}}

\newcommand{\calP}{{\mathcal{P}}}
\newcommand{\calQ}{{\mathcal{Q}}}

\newcommand{\calU}{{\mathcal{U}}}

\newcommand{\ER}{Erd\H{o}s--R\'enyi\xspace}

\newcommand{\Tr}{\mathrm{Tr}}
\renewcommand{\tilde}{\widetilde}

\newcommand{\tri}{\mathrm{tri}}
\newcommand{\scan}{\mathrm{scan}}

\newif\ifhideproofs
\ifhideproofs
\usepackage{environ}
\NewEnviron{hide}{}

\fi

\newif\ifdraft
\drafttrue

\begin{document}
\title{\hspace{-0.5cm} \parbox{\textwidth}{\centering Detection of local geometry in random graphs: \\information-theoretic and 
computational limits}}
\author{
    Jinho\ Bok\thanks{J. Bok is with the Department of Statistics and Data Science, The Wharton School, University of Pennsylvania. Email: \href{mailto:jinhobok@wharton.upenn.edu}{\color{black}{\texttt{jinhobok@wharton.upenn.edu}}}.},~  %
    Shuangping\ Li\thanks{S. Li is with the Department of Statistics and Data Science, Yale University. Email: \href{mailto:shuangping.li@yale.edu}{\color{black}{\texttt{shuangping.li@yale.edu}}}.},~and %
    Sophie\ H.\ Yu\thanks{S. H. Yu is with the Operations, Information and Decisions Department, The Wharton School, University of Pennsylvania. Email: \href{mailto:hysophie@wharton.upenn.edu}{\color{black}{\texttt{hysophie@wharton.upenn.edu}}}.}
} 
\date{March 25, 2026}
\maketitle
\thispagestyle{empty}
\begin{abstract}

We study the problem of detecting local geometry in random graphs. We introduce a model $\mathcal{G}(n,p,d,k)$, where a hidden community of average size $k$ has edges drawn as a random geometric graph on $\mathbb{S}^{d-1}$, while all remaining edges follow the \ER model $\calG(n,p)$. The random geometric graph is generated by thresholding inner products of latent vectors on $\mathbb{S}^{d-1}$, with each edge having marginal probability equal to $p$. This implies that $\calG(n,p,d, k)$ and $\calG(n, p)$ are indistinguishable at the level of the marginals, and the signal lies entirely in the edge dependencies induced by the local geometry.

We investigate both the information-theoretic and computational limits of detection. On the information-theoretic side, our upper bounds follow from three tests based on signed triangle counts: a global test, a scan test, and a constrained scan test; our lower bounds follow from two complementary methods: truncated second moment via Wishart--GOE comparison, and tensorization of KL divergence.
These results together settle the detection threshold at 
$d = \widetilde{\Theta}(k^2 \vee k^6/n^3)$ for fixed $p$, and extend the state-of-the-art bounds from the full model (i.e., $k = n$) for vanishing $p$. 
On the computational side, we identify a computational--statistical gap and provide evidence via the low-degree polynomial framework, as well as the suboptimality of signed cycle counts of length $\ell \geq 4$.

\end{abstract}

\newpage 
\tableofcontents
\thispagestyle{empty}
\clearpage

\setcounter{page}{1}

\section{Introduction}

Networks across 
multiple domains often contain inherent structures \cite{Newman2010Networks,Barabasi2016NetworkScience}: communities in social networks \cite{holland1983sbm, girvan2002community, Fortunato2010CommunityDetection}, functional modules in biological systems \cite{HartwellEtAl1999ModularCellBiology,SpirinMirny2003ProteinModules,BarabasiOltvai2004NetworkBiology}, and anomalous subgraphs in communication networks \cite{PriebeEtAl2005ScanEnron,AkogluTongKoutra2015AnomalySurvey}.
Detecting such structure from noisy observations is a fundamental statistical problem, which has also driven significant advances in probability theory, combinatorics, and theory of algorithms. Prominent models for this task include the stochastic block model \cite{holland1983sbm,DecelleEtAl2011AsymptoticSBM,Abbe2018SBMSurvey}, the planted clique \cite{jerrum1992plantedclique,kucera1995plantedclique}, the planted dense subgraph \cite{pdsdense, hajek2015computationalcommunity, pdssparse}, and the planted matching \cite{moharrami2021plantedmatching, ding2023plantedmatching}, %
each serving as a benchmark for understanding statistical and computational phase transitions in structured random graphs.

The hidden subgraphs in these models are often assumed to be ``simple'', having a distinctive combinatorial shape or an elevated edge density relative to the background. While analytically convenient, such assumptions can be misaligned with real-world networks, where the defining signature of a subgraph may lie not in its density or shape but in how its vertices relate through their latent features. %
In particular, edges are influenced by the similarity between those features (e.g., personal profiles, textual representations, biological summaries), which is often modeled through latent space \cite{hoff2002latent,penrosebook,handcock2007latentcluster,newman2016annotated}. 
Under this perspective, the structure to be detected is better described by its interaction patterns that are consistent with the underlying geometry.

This distinction is especially important in settings where each vertex resembles or imitates the others.
For instance, in social networks, the subgraph of interest may be a small group of genuine users among bots \cite{ferrara2016socialbots}, or a set of accounts under coordination for influence \cite{pacheco2021uncovering}; similarly, in economic networks the subgraph may consist of firms under collusion in a marketplace \cite{morselli2018collusion, wachs2019collusion}. 
In such cases, each vertex may appear to be statistically similar despite the interactions within the subgraph at the level of the latent space. Furthermore, those interactions are often inherently intricate, characterized by contextual or longitudinal features in high dimensions.

Motivated by the geometry-based signals as described, we introduce a random graph model in which a small, hidden community exhibits \emph{local geometry}.  %
We focus on the fundamental task of \emph{detection}: given an observed graph $G$ on vertex set $[n]$, decide whether it was generated from a null model with no signal or from our proposed alternative model.

Formally, we consider the hypothesis testing problem
\[
\mathcal P := \mathcal G(n,p,d,k)
\quad\text{vs.}\quad
\mathcal Q := \mathcal G(n,p)\,,
\]
where $\mathcal Q$ is the \ER model \cite{erdosrenyi} and $\mathcal P$ is a planted version of the high-dimensional random geometric graph model \cite{devroye2011rgg}. Under $\mathcal P$, a hidden set $S$ of expected size $k$ (which we refer to as the community)
carries latent feature vectors, and edges within %
$S$ are formed according to geometric proximity in a $d$-dimensional latent space; all remaining edges behave as in $\mathcal G(n,p)$. We now give the formal definition of $\mathcal P$.

\begin{definition}[Random graphs with local high-dimensional geometry]\label{def:planteddist}
A sample $G \sim \calP$ is drawn as follows:
\begin{enumerate}
    \item Each vertex $i \in [n]$ joins the community $S$ independently with probability $k/n$.
    \item Each community vertex $i \in S$ receives a latent feature vector
    $U_i \overset{\mathrm{i.i.d.}}{\sim} \mathcal{U}(\mathbb{S}^{d-1})$.
    \item For any $i,j \in [n]$ with $i\neq j$, if $i,j \in S$,
    edge $ij$ is present iff $\langle U_i, U_j\rangle \geq \tau(p,d)$;
    otherwise edge $ij$ is present independently with probability $p$.
    The threshold $\tau(p,d)$ is chosen so that
    $\mathbb{P}(\langle U_i,U_j\rangle \geq \tau(p,d)) = p$.
\end{enumerate}
\end{definition}

\begin{figure}[!t]
    \centering
    \begin{subfigure}[t]{0.4\linewidth}
        \centering
        \includegraphics[width=\linewidth]{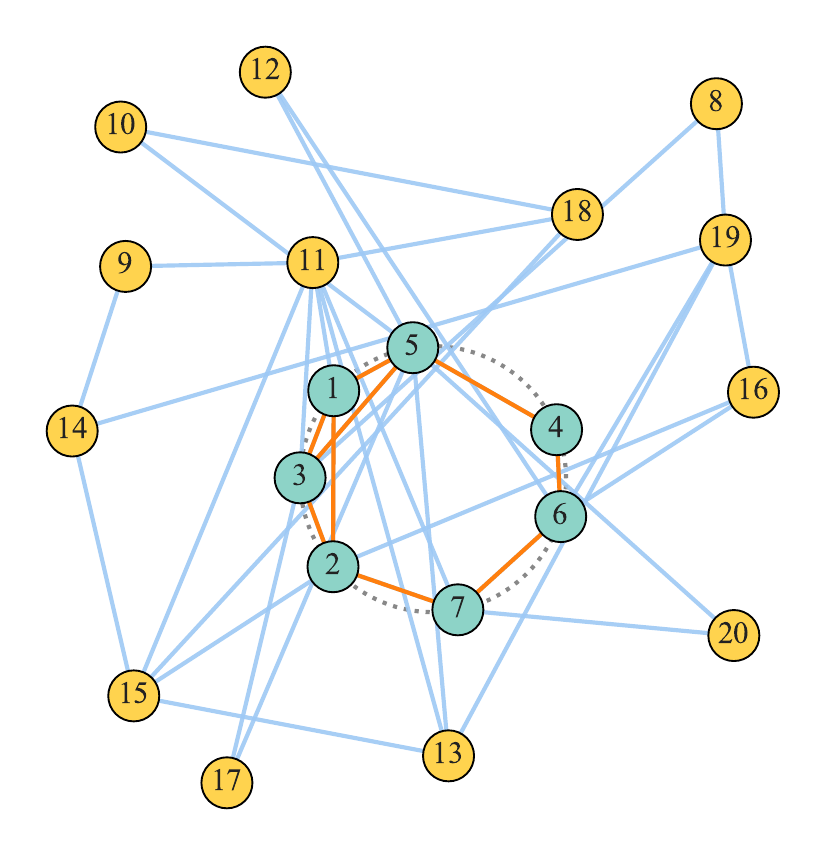}
        \caption{}\label{fig:illustration-a}
    \end{subfigure}
    \hspace{2mm}
    \begin{subfigure}[t]{0.4\linewidth}
        \centering
        \includegraphics[width=\linewidth]{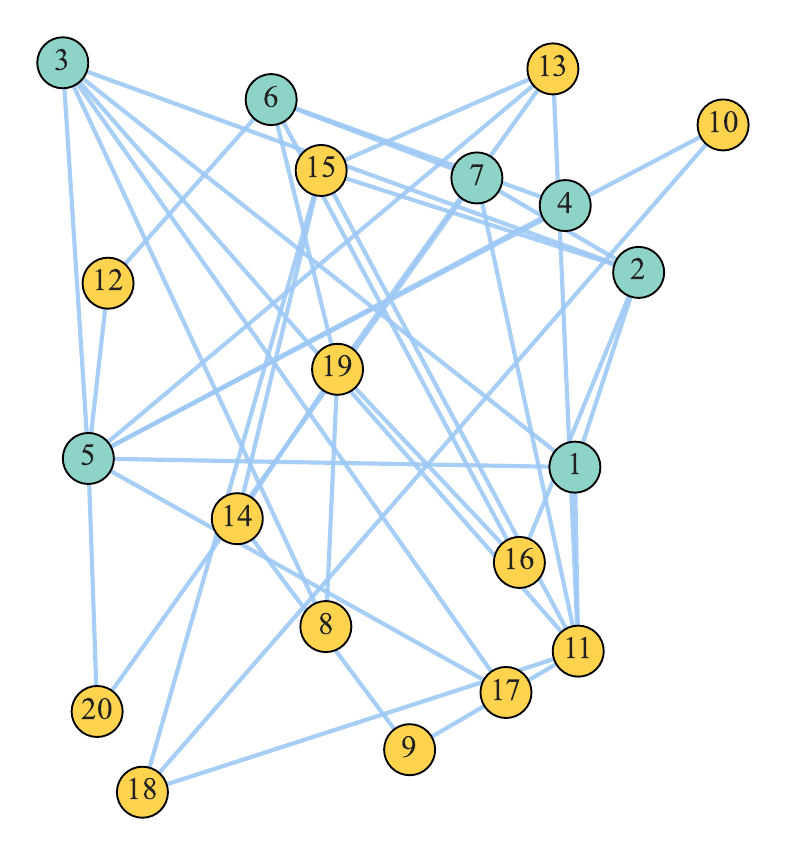}
        \caption{}\label{fig:illustration-b}
    \end{subfigure}

    \caption{
    Two drawings of the same graph sampled from $\calG(n, p, d, k)$ with $n=20$, $p=0.18$, $k=7$; we set $d=2$ for visualization purposes. 
    (a) Vertices are positioned to reflect the latent geometry: the $k=7$ community 
    vertices (circled, teal) have latent vectors drawn from $\mathcal{U}(\mathbb{S}^1)$, 
    so they lie on a circle in the latent space. The orange edges---based on geometric proximity---reveal the resulting cycle-rich structure 
    induced by the local geometry. 
    (b) Vertices are positioned randomly, and the planted community becomes 
    visually indistinguishable from the \ER background. } 
    \label{fig:illustration}
\end{figure}

We only observe the final graph $G$; neither the community $S$ nor the latent vectors $\{U_i\}$ are observed. Moreover, by construction, every edge marginally appears with probability $p$, and vertices have the same marginal neighborhood distribution regardless of whether they belong to $S$. Thus, the signal is not visible at the level of the first moment and is instead carried by the dependence structure induced by the local geometry. An illustration of a sample from $\calG(n, p, d, k)$ is provided in Figure~\ref{fig:illustration}. When $k = n$, we write the distribution $\calG(n, p, d, k)$ as $\calG(n, p, d)$; this is the traditional high-dimensional random geometric graph \cite{devroye2011rgg} studied in the literature, which we refer to as the full model.

We study when detection is possible as $n\to\infty$, allowing the parameters $p,d,k$ to depend on $n$ (and hence on each other). We use the following standard notions.

\begin{definition}
For the detection problem of $\calP$ vs.\ $\calQ$, a test statistic $f(G)$ with threshold $\gamma$ achieves
\begin{itemize}
    \item[(a)] \emph{strong detection} if
    $\Prob_{G \sim \calQ}(f(G) > \gamma) + \Prob_{G \sim \calP}(f(G) \leq \gamma) = o(1)$;
    \item[(b)] \emph{weak detection} if
    $\Prob_{G \sim \calQ}(f(G) > \gamma) + \Prob_{G \sim \calP}(f(G) \leq \gamma) = 1 - \Omega(1)$.
\end{itemize}
\end{definition}

It is well-known (by the Neyman--Pearson lemma) that the infimum of the sum of type~I and type~II errors equals $1-\dTV(\calP,\calQ)$. In particular, no test can achieve weak detection if $\dTV(\calP,\calQ)=o(1)$. Intuitively, detection becomes harder as the dimension $d$ grows, since the geometric constraints induce weaker dependencies among edges in higher dimensions. Our main goal is to quantitatively characterize how large $d$ can be (as a function of $n,p,k$) while detection remains possible.
\subsection{Our contributions}

We characterize when the detection between $\calP$ and $\calQ$ is possible, tracing out both the information-theoretic and computational limits as functions of $n, p, d, k$.
The following theorem summarizes our main results in the log-density setting \cite{bhaskara2010logdensity}; see Figure~\ref{fig:phasediagram} for the resulting phase diagrams.

\begin{theorem}[Informal]\label{thm:main}
Let
\begin{align*}
    p = \Theta(n^{-\alpha})\,, \quad
    d = \Theta(n^{\beta})\,, \quad
    k = \Theta(n^{\gamma})\,,
\end{align*}
where $0 \leq \alpha < 1$, $\beta > 0$, and $0 < \gamma \leq 1$.
\begin{enumerate}
    \item[(i)] If $\beta < 6\gamma - 3\alpha - 3$, strong detection is possible with a test statistic that is efficiently computable.
    \item[(ii)] If $\beta < 2\gamma - 3\alpha$, strong detection is possible with a test statistic that is inefficiently computable.
    \item[(iii)] If any of the following holds, weak detection is impossible:
    \begin{itemize}
        \item[$\bullet$] $\beta > 2\gamma \vee (6\gamma - 3)$;
        \item[$\bullet$] $\beta > (2\gamma - 2\alpha) \vee (4\gamma - 2\alpha - 1)$ and $\gamma > \alpha$.
    \end{itemize}
    \item[(iv)] If $\beta > 6\gamma - 3\alpha - 3$, weak detection is impossible for low-degree polynomial algorithms.
\end{enumerate}
\end{theorem}

\begin{figure}[!t]
    \centering
    \includegraphics[width=0.9\textwidth]{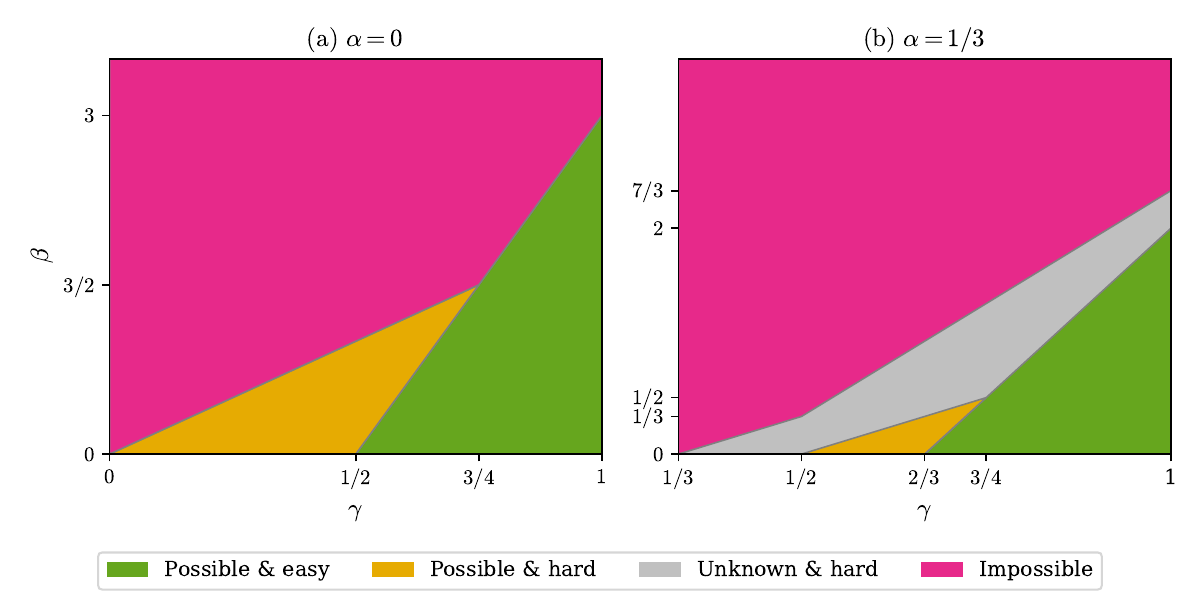}
    \caption{Phase diagram for detection, for (a) $\alpha = 0$ and  (b) $\alpha = 1/3$. Note that the two plots are in different scales. In the possible \& easy phase (green), strong detection can be done by an efficient test statistic. In the possible \& hard phase (yellow), strong detection can be done by an inefficient test statistic, and weak detection is impossible for low-degree polynomial algorithms. In the unknown \& hard phase (grey), it is open whether strong detection is possible, but weak detection is impossible for low-degree polynomial algorithms. Finally, in the impossible phase (magenta), weak detection is impossible. The impossible phases extend to all values of $\beta > 0$ beyond those presented in the plots. } 
    \label{fig:phasediagram}
\end{figure}

\paragraph{Information-theoretic limits.}
Parts (i)--(iii) of Theorem~\ref{thm:main} together characterize the information-theoretic threshold for detection. On the upper bound side, we propose three tests based on \emph{signed triangle counts}: a global test (counting over the entire graph), a scan test (taking the maximum signed triangle count over all subsets of size $\approx k$), and a constrained scan test (further restricting to subsets with controlled wedge sums). The scan test covers a complementary regime to the global test---which together suffice for fixed $p$---and the constrained scan test strictly extends the parameter regime of the scan test when $p = \widetilde{o}(1)$. For details, see Theorems~\ref{thm:ubsignedtrianglecount}, \ref{thm:ubscandense} and \ref{thm:ubscangeneral}.
On the lower bound side, we develop two complementary approaches: truncated second moment, which captures the dependence on the average community size $k$; and tensorization of KL divergence, which captures the dependence on the edge density $p$. For details, see Theorems~\ref{thm:lbdense} and \ref{thm:lbgeneral}.

In the dense case ($\alpha = 0$), combining those upper and lower bounds settles the detection threshold sharply at $d = \widetilde{\Theta}(k^2 \vee k^6/n^3)$; see Figure~\ref{fig:phasediagram}(a). 
For $0 < \alpha < 1$, our results generalize the state-of-the-art bounds for the full model $\mathcal{G}(n,p,d)$ (the special case $k = n$,  i.e., $\gamma = 1$), recovering the upper bound $d = \widetilde{o}(n^3p^3)$ and lower bound $d = \widetilde{\Omega}(n^3p^2)$ of~\cite{liu2022rgg} and extending them to our planted setting for all $k = \Theta(n^\gamma)$ with $\gamma \in (\alpha,1)$; see Figure~\ref{fig:phasediagram}(b). 
Our results leave a gap (the unknown \& hard region in Figure~\ref{fig:phasediagram}(b)) between upper and lower bounds when $0 < \alpha < 1$; we conjecture that the upper bounds are tight and that this region is in fact in the impossible phase.

\paragraph{Computational limits.}
Part~(iv) of Theorem~\ref{thm:main} provides evidence that the regime beyond part (i) is computationally hard, based on the \emph{low-degree 
polynomial framework}~\cite{hopkins2018statistical, kunisky2019notes, 
wein2025computational} (see Sections~\ref{sec:priorwork} and \ref{subsec:maincomp} for background); note that among our proposed tests, only the global signed triangle count runs in polynomial time.
Specifically, we show that no polynomial of degree at most 
$\lfloor (\log n / \log (\log n))^2 \rfloor $ achieves weak separation whenever $\beta > 6\gamma - 
3\alpha - 3$, matching the threshold of the global test; see Theorem~\ref{thm:lowdeglb}.  We further show 
that no signed cycle count of length $\ell \geq 4$ improves upon the triangle count, 
providing additional evidence that the global signed triangle count may be the asymptotically 
optimal efficient test among all signed cycle counts; see Proposition~\ref{prop:signedcycle}.

\paragraph{Computational--statistical gap.} Parts~(i) and~(ii), together with (iv) identify a regime of $6\gamma - 
3\alpha - 3 < \beta < 2\gamma - 3\alpha$ (the yellow regions in 
Figure~\ref{fig:phasediagram}), where detection is 
information-theoretically possible yet no efficient algorithm is known.
Interestingly, for the full model $\mathcal{G}(n, 
p, d)$ (i.e., $k = n$ without locality), a computational--statistical gap is not 
known and conjecturally does not exist~\cite{liu2022rgg}. 
The gap here thus appears 
to be a consequence of the local nature of the planted geometric structure: the hidden community may be placed at exponentially many possible locations, introducing a combinatorial search barrier for efficient algorithms.

\subsection{Related literature}\label{sec:priorwork}
We review several lines of research that are closely related to our work.

\paragraph*{High-dimensional random geometric graphs.} A classical line of work studied random geometric graphs for fixed dimension $d$; see \cite{penrosebook} for an overview. The study of random graphs with high-dimensional (i.e., $d \to \infty$ as $n \to \infty$) geometry was initiated by \cite{devroye2011rgg}, where the authors studied the clique number of $\calG(n, p, d)$ and showed that for $d \geq \exp(\widetilde{\Omega}(n^2))$, the graph is indistinguishable from $\calG(n, p)$. 
Ever since, there has been a growing line of work on random geometric graphs in high dimensions, investigating a variety of algorithmic and statistical phenomena.

For the detection between $\calG(n, p, d)$ and $\calG(n, p)$, the breakthrough work of \cite{bubeck2016rgg} settled the threshold of $d = \Theta(n^3)$ for fixed $p$. That paper also introduced a test that counts the number of signed triangles (see Section~\ref{sec:overview} for details), which attains the best known upper bound of $d = \widetilde{o}(n^3p^3)$ for general $p$ \cite{liu2022rgg}. Several works since then have improved lower bounds for the regime of $p = o(1)$. Namely, \cite{brennan2020rgg} showed a lower bound of $d = \widetilde{\Omega}(n^3p \vee n^{7/2}p^2 \vee n)$ for $p = \widetilde{\Omega}(1/n^2)$. This was later improved in \cite{liu2022rgg}, which is the state-of-the-art result: $d = \widetilde{\Omega}(n^3p^2)$ for $p = \Omega(1/n)$, 
and $d = \Omega((\log n)^{36})$ for $p = \Theta(1/n)$. Notably, it is an open problem to reduce the gap of polynomial factor $p$ between the upper and lower bounds.

Besides detection, $\calG(n, p, d)$ have received increasing attention in recent years, in terms of its spectral property \cite{liu2023expansion, abdalla2024synch, cao2025spectra}, low-degree moments \cite{bangachev2024fourier}, coupling with $\calG(n, p)$ \cite{liu2022rgg, bangachev2025sandwiching}, latent estimation \cite{mao2024latent}, and rare events \cite{deka2025rare}, to name a few. Furthermore, there has been active research on models that are different from but closely related to $\calG(n, p, d)$, where the difference lies in various factors such as edge connection rule \cite{liu2023noisyrgg, liuracz, mao2026rggsmooth}, metric of the latent geometry \cite{bangachev2024geometry, baguley2025toroidal}, isotropy \cite{eldan2020anisotropy, brennan2024anisotropy}, homogeneity \cite{bringmann2019geometricinhomogeneous}, cluster structure \cite{li2024spectralclusteringgaussianmixture}, and combinations thereof.
For further discussion on high-dimensional random geometric graphs, we refer readers to a recent survey \cite{rggsurvey}.

\paragraph*{Planted subgraphs.} Our model can be viewed as a particular case within a class of random graphs known as planted subgraphs. 
In general, these can be generated by first drawing a background \ER graph $G \sim \calG(n,p)$ and independently a subgraph $H$ over the complete graph from another distribution. For vertex set $S$ of $H$, the final graph is then obtained by either replacing the induced subgraph $G[S]$ with $H$, or taking the union $G \cup H$.
Planted subgraphs can be considered as graph (binary) versions of spiked random matrices \cite{johnstone2001spiked, baik2005spiked}, which are fundamental objects in probability theory and statistics. Within the vast landscape of random graph models with community structures, many planted subgraphs can be characterized with the existence of a single community. The literature on models with multiple communities (e.g., the stochastic block model \cite{holland1983sbm}) is extensive and merits a separate discussion; here we focus on the single community case.

An iconic example of planted subgraphs is the planted clique \cite{jerrum1992plantedclique, kucera1995plantedclique}, where a clique (i.e., a complete graph) of small size is hidden within a graph from $\calG(n, 1/2)$. This simple model is well-known for exhibiting a computational--statistical gap, with vast connections across theoretical computer science \cite{juels2000plantedcliquecrypto, hazan2011plantedcliquegametheory} and high-dimensional statistics \cite{berthet2013plantedcliquereduction, brennan2020plantedcliquereduction}. The literature has since been expanding, with various choices for the planted subgraph such as a dense subgraph \cite{pdsdense, hajek2015computationalcommunity, pdssparse}, a tree \cite{massoulie2019plantedtrees}, a cycle \cite{bagaria2020hamiltoniancycle, gaudio2025plantedcycles}, a matching \cite{moharrami2021plantedmatching, ding2023plantedmatching, wee2025cluster, addario2026statistical}, to name a few. Besides subgraph-specific results, a recent line of work \cite{huleihel2022hidden, elimelech2025arbitrary,  lee2025plantedsubgraph, mossel2025inference, yu2025countingstars} aims to provide a unified theory for general subgraphs.

\paragraph*{Signals beyond mean.} A common feature in planted subgraphs is that the signal exhibits at the level of mean. For example, in planted dense subgraph \cite{pdsdense} the edge density is higher on average within the community than the rest of the graph. In other words, the signal already exists at the lowest possible level (the first moment), and often there is no need to consider any interaction (e.g., higher-order moments) between the inputs. As a result, it is often the case that if properly done, thresholding the mean suffices.
This is not the case for our model, as each vertex marginally has the same distribution.

A few papers have studied settings where the hidden signals are not observed at the mean level. For example, \cite{arias-castro2012correlation, arias-castro2015correlation, arias-castro2018markov} studied detection problems where among samples of mean 0, only a small unknown subset has dependence within, such as positive correlation or Markovian structure. Another notable recent paper \cite{kunisky2025directed} extensively analyzed a new directed random graph model, where a small unknown subset of vertices have latent ranking. Between the ranked vertices in that model, a directed edge from a vertex of higher rank to a vertex of lower rank is more likely to be added, compared to the other direction. As the overall edge density is the same over the whole graph, there is no signal at the mean level; in particular, the model is equivalent to $\calG(n, p)$ if the direction is ignored. As a result, the detection is done by considering the unusual consistency of \emph{pairwise} orderings, rather than the edge density.\footnote{We mention that despite the apparent differences between the models, interestingly, our analysis shares certain key technical results with theirs (e.g., Lemma~\ref{lem:hypergeomcubemgf}).}

\paragraph*{Random graphs with geometry-based communities.}

In recent years, various random graph models with both latent geometry and community structures have been studied. Often the community structures in those models are in the style of the stochastic block model, consisting of multiple communities: examples include the geometric block model \cite{galhotra2018geometricblock, galhotra2023geometricblock}, the geometric SBM \cite{abbe2021community, gaudio2024gsbm, gaudio2025exact}, the geometric hidden community model \cite{gaudio2024exact, gaudio2026exactrecoverygeometrichidden}, and different variants thereof \cite{peche2020community, avrachenkov2021geometric, avrachenkov2024community}. Our model differs from those in that it has a single hidden community rather than multiple. One notable exception is the planted dense cycle \cite{mao2023planteddensecycle, mao2025planteddensecycle}, a latent-based model where signed triangle count is also used for detection; however, in terms of the modeling components this model is also fundamentally distinct from ours.

We highlight two models in the literature that share certain common features with our model. In the model introduced in \cite{bet2020botnet}, a small community made of an \ER graph is hidden in a larger random geometric graph. %
Hence, at a conceptual level our model can be viewed as an inverted version of theirs, where the roles of community and non-community are flipped. In another model introduced in \cite{bet2025localized}, edges within a hidden community are affected by its latent geometry, whereas edges outside are formed independently.
Thus, in principle we combine the community structure and the latent geometry in the same way. Despite those similarities, the mathematical details of those models are quite different from ours and hence our results cannot be directly compared to theirs. Moreover, we quantitatively characterize how high dimensionality in the geometry affects the graph, whereas the results in both of those works are either independent of $d$ or for fixed $d$.

\paragraph*{Low-degree polynomial framework.}

Many high-dimensional statistical models exhibit a phenomenon known as computational--statistical gap, where in certain parameter regimes a task is information-theoretically feasible but appears to lack any polynomial-time algorithm. A prominent approach for analyzing this phenomenon is the low-degree polynomial framework~\cite{hopkins2018statistical}, which studies algorithms expressible as low-degree polynomials of the input. For input dimension $n$, one considers polynomials of degree at most $D$; the guiding heuristic is that $D = O(\log n)$ often matches the power of polynomial-time algorithms for many average-case problems. This heuristic is supported by many examples, as polynomials of degree $O(\log n)$ can implement or approximate various efficient algorithms including spectral methods, subgraph-counting procedures, and approximate message passing (see~\cite[Section~6.2]{wein2025computational}). Accordingly, hardness at degree $D = \omega(\log n)$ is widely regarded as evidence of computational hardness beyond polynomial time.

For hypothesis testing, the low-degree polynomial framework examines whether low-degree polynomials can achieve \emph{separation} between the null and the alternative distributions (see Definition~\ref{def:strongsep}). In practice, this is often analyzed via the low-degree likelihood ratio, i.e., the norm of the projection of the likelihood ratio onto the space of polynomials of degree at most $D$; see, e.g., \cite{kunisky2019notes} for details. Understanding the rigorous implications of this criterion has recently attracted significant attention and is an active area of research \cite{holmgren2021lowdeg, buhai2025quasi, hsieh2026rigorous, jia2026lowdegree}; we refer to a recent survey~\cite{wein2025computational} for a further discussion.

\subsection{Notations}

We denote $[n] := \{1, \dots, n\}$, and $\binom{[n]}{k}$ to be the set of all size-$k$ subsets of $[n]$. For graph $H$, $V(H)$ and $E(H)$ respectively denote the set of its vertices and edges; $v(H)$ and $e(H)$ denote their respective cardinalities. $K_n$ denotes the complete graph on $[n]$. We use $\mathbb{S}^{d-1}$ to denote the unit sphere in $\reals^d$, and $\calU(\mathbb{S}^{d-1})$ to denote the uniform distribution over $\mathbb{S}^{d-1}$ (i.e., the Haar measure); throughout, $U_i \overset{\mathrm{i.i.d.}}{\sim} \calU(\mathbb{S}^{d-1})$ for $i \in [n]$. For symmetric matrix $A$, $\norm{A}_{\op}$ denotes its operator norm. All logarithms are with base $\exp(1)$, and $\log^a b$ denotes $(\log b)^a$.

For asymptotics, we always assume $n \to \infty$ with $k = k(n) \to \infty$, and $d$ to be sufficiently large. We use standard big-$O$ notation, where for any $a_n, b_n$, $a_n = O(b_n)$ and $a_n \lesssim b_n$ denote $a_n \leq Cb_n$ for some absolute constant $C > 0$; $a_n = \Omega(b_n)$ and $a_n \gtrsim b_n$ denote $b_n = O(a_n)$; $a_n = \Theta(b_n)$ denotes $a_n = O(b_n)$ and $a_n = \Omega(b_n)$. Also, $a_n = o(b_n)$ and $a_n \ll b_n$ denote $\lim_{n \to \infty} (a_n / b_n) = 0$; $a_n = \omega(b_n)$ and $a_n \gg b_n$ denote $b_n = o(a_n)$. For each of those we use $\widetilde{O}, \widetilde{\Omega}, \widetilde{\Theta}, \widetilde{o}, \widetilde{\omega}$ to hide $\mathrm{polylog}(n)$ factors.

\section{Main results}
In this section, we present the information-theoretic upper and lower bounds, and the computational lower bound for the detection problem. Throughout this section,  we assume $p \leq 1/2$; for any fixed $p \in (1/2, 1)$, it can be readily deduced (\eg, following \cite[Lemmas 3 \& 4]{bubeck2016rgg}) that our detection threshold for fixed $p \in (0, 1/2]$ extends.

\subsection{Information-theoretic upper bound}\label{subsec:mainub}
We present three different tests for the detection between $\calP$ and $\calQ$.
First, we consider the \emph{global test} for counting signed triangles, whose test statistic is defined as
\begin{equation}\label{eq:signedtristatistic}
    f_{\mathrm{tri}}(G) := \sum_{i < j < \ell \in [n]}(G_{ij}-p)(G_{j\ell}-p)(G_{i\ell}-p)\,.
\end{equation}
This is a natural candidate in that it achieves the best known performance for the detection between $\calG(n, p)$ and the full model $\calG(n, p, d)$ \cite{bubeck2016rgg, liu2022rgg}. %
\begin{theorem}[Detection via global test]\label{thm:ubsignedtrianglecount} 
There exists a constant $C_{\ref{thm:ubsignedtrianglecount}} > 0$ such that if %
\[\frac{1}{k} \leq  p \leq \frac{1}{2} \quad \text{and} \quad C_{\ref{thm:ubsignedtrianglecount}} \vee (5\log(1/p))^4 \leq  d \ll \frac{k^6p^3}{n^3}\log^3(1/p)\,,
\]
the testing error satisfies
\begin{align*}
    \bbP_{G \sim \calQ}\left(f_{\mathrm{tri}}(G) > \gamma_{\mathrm{tri}} \right) + \bbP_{G \sim \calP}\left(f_{\mathrm{tri}}(G) \leq \gamma_{\mathrm{tri}} \right) = o(1)\,, 
\end{align*}
where the threshold is chosen as\footnote{The constant factor of $1/2$ in front of $\gamma_{\tri}$ is arbitrary and can be replaced with any fixed constant in $(0, 1)$. This also applies to the thresholds in the scan test and the constrained scan test. The equality between expressions in $\E_{G \sim \calP}$ and $\E_{G \sim \calG(n, p, d)}$ follows from Lemma~\ref{lem:fouriersimple}.} 
\[\gamma_{\mathrm{tri}} :=  \frac{1}{2}\E_{G \sim \calP}[f_{\mathrm{tri}}(G)] = \frac{1}{2}\binom{n}{3}\left(\frac{k}{n}\right)^3\E_{G \sim \calG(n, p, d)}[(G_{12}-p)(G_{23}-p)(G_{13}-p)]\,.
\]
\end{theorem}

In the global test, we essentially compare the signed triangle count within the community with the fluctuation of the signed triangle count over the entire graph. This fluctuation can be quite large, in particular when $n$ is much larger than $k$. It is thus natural to consider the \emph{scan test}, where we instead consider the individual fluctuations of subgraphs of similar size. In particular, let
\begin{equation}\label{eq:scanstatistic}
    f_{\mathrm{scan}}(G) := \max_{\substack{A \subseteq [n],  |A| = k^-}}  \sum_{i < j < \ell \in A} (G_{ij}-p)(G_{j\ell}-p)(G_{i\ell}-p) \,,
\end{equation}
where $k^- :=  \lfloor 0.9k \rfloor $. We show that this succeeds for $d = \widetilde{O}(k^2p^6)$, a threshold with only logarithmic dependence on $n$. 
\begin{theorem}[Detection via scan test]\label{thm:ubscandense} There exists a constant $C_{\ref{thm:ubscandense}} > 0$ such that if 
\[
p \leq \frac{1}{2}, \quad k \geq C_{\ref{thm:ubscandense}}\log^2n \quad \text{and} \quad C_{\ref{thm:ubscandense}} \vee (5\log(1/p))^4 \leq d \leq \frac{k^2p^6\log^6(1/p)}{C_{\ref{thm:ubscandense}}\log n}\,,
\]
the testing error satisfies
\begin{align}
    \bbP_{G \sim \calQ}\left(f_{\mathrm{scan}}(G) > \gamma_{\mathrm{scan}} \right) + \bbP_{G \sim \calP}\left(f_{\mathrm{scan}}(G) \leq \gamma_{\mathrm{scan}} \right) = o(1)\,,  \label{eq:scan_testing}
\end{align}
where the threshold is chosen as  
\begin{align}
    \gamma_{\mathrm{scan}} := \frac{1}{2}\binom{k^-}{3}\E_{G \sim \calG(n, p, d)}[(G_{12}-p)(G_{23}-p)(G_{13}-p)]\,.  \label{eq:scan_gamma}
\end{align} 
\end{theorem}
As we will see later, the global test and the scan test are sufficient for any fixed $p$ in that there exists a matching lower bound (up to a logarithmic factor). However, the detection threshold provided by the scan test quickly degrades as $p \to 0$.
To improve upon this, we add certain constraints on top of the scan test, which we call the \emph{constrained scan test}. To be specific, among the subgraphs of size $\approx k$, we only consider those that satisfy additional conditions on (signed) wedge counts. Formally, let
\begin{equation}\label{eq:constscanstatistic}
    \widetilde{f}_{\mathrm{scan}} (G) : =\max_{ A \in \calC(G), |A| = k^-}\sum_{i <j < \ell \in A} (G_{ij}-p)(G_{j\ell}-p)(G_{i\ell}-p) \,,
\end{equation}
where %
\begin{align}
    \calC(G) := \Bigg\{A \subseteq [n]: & \sum_{i < j \in A} \left(\sum_{\ell < i, \ell \in A} (G_{\ell i}-p)(G_{\ell j}-p)\right)^2 \leq \sigma^2 \,, \label{eq:constrainedscantestvar} \\
    &  \max_{i<j\in A} \left|\sum_{\ell < i, \ell \in A} (G_{\ell i}-p)(G_{\ell j}-p)\right| \leq B  \label{eq:constrainedscantestrange} \Bigg\}\,,
\end{align}
with 
\begin{align*}
    \sigma^2 := k^3p^2 + C_{\ref{prop:signedcyclecount}}^4\frac{k^4p^4\log^2 (1/p)}{d} \quad  \text{and} \quad
    B := (2048kp^2 + 8)\lceil \log k \rceil \,. 
\end{align*}
\begin{theorem}[Detection via constrained scan test]\label{thm:ubscangeneral} Assume that there exists a constant $\delta > 0$ such that $d \geq n^{\delta}, n^{-1+\delta} \leq p \leq 1/2$. Then there exists a constant $C_{\ref{thm:ubscangeneral}} = C_{\ref{thm:ubscangeneral}}(\delta) > 0$ such that if
\[d \leq \frac{1}{C_{\ref{thm:ubscangeneral}}}\frac{k^2p^3\log^3(1/p)}{(\log^2 n) (\log^2k)}\,,
\]
\prettyref{eq:scan_testing} holds with $\widetilde{f}_{\mathrm{scan}} (G) $ in place of $ f_{\mathrm{scan}} (G) $, that is,
\begin{align*}
    \bbP_{G \sim \calQ}\left(\widetilde{f}_{\mathrm{scan}} (G) > \gamma_{\mathrm{scan}} \right) + \bbP_{G \sim \calP}\left(\widetilde{f}_{\mathrm{scan}}(G) \leq \gamma_{\mathrm{scan}} \right) = o(1)\,. 
\end{align*}
\end{theorem}

\begin{remark}[Comparison between tests]\label{remark:testcomparison} As noted earlier, the global test (Theorem \ref{thm:ubsignedtrianglecount}) and the scan test (Theorem \ref{thm:ubscandense}) together suffice for characterizing the detection threshold for any fixed $p$, up to a logarithmic factor. In terms of the performance guarantees from the theorems, the global test is better if $k \gg (n^3 / \log n)^{1/4}$, whereas the scan test is better if $k \ll (n^3/\log n)^{1/4}$. 

When $p = o(1)$, the constrained scan test (Theorem \ref{thm:ubscangeneral}) is better than the scan test except for a very limited regime. In fact, it directly follows from our analysis that for $p = O((\log n)^{-1} (\log^2 k)^{-1})$, the constrained scan test succeeds if $d = O(k^2p^3\log^3(1/p)/\log n)$, strictly improving the scan test; see \eqref{eq:ubgeneralfinalcond1} and the surrounding arguments there. As in the dense case, whether the global test or the constrained scan test is guaranteed for a better performance depends on how $k$ compares to $\widetilde{\Theta}(n^{3/4})$.

\end{remark}

\subsection{Information-theoretic lower bound}\label{subsec:mainlb}
We present different thresholds for the impossibility of detection, based on two different approaches.
First, we focus on capturing the dependence on $k$. For this, we consider calculating the truncated second moment between certain random matrices that generate the random graphs. 

\begin{theorem}[Lower bound via truncated second moment]\label{thm:lbdense} There exists a constant $C_{\ref{thm:lbdense}} > 0$ such that the following holds: if
\[k \leq \frac{n}{5}, \quad d \geq C_{\ref{thm:lbdense}}\log(1/p) \quad \text{and} \quad d \gg k^2 \vee  \frac{k^6}{n^3}\,,
\]
no test achieves weak detection.\footnote{If $k > n/5$, no test achieves weak detection if $d \gg k^3 = \Theta(n^3)$ even when the community location is known.}
\end{theorem}
By combining this lower bound with the upper bounds (Theorems~\ref{thm:ubsignedtrianglecount} and \ref{thm:ubscandense}), one can conclude that for any fixed $p$ the detection threshold is given as
\[d = \widetilde{\Theta}\left(k^2 \vee \frac{k^6}{n^3}\right)\,.
\]

On the other hand, the threshold in Theorem~\ref{thm:lbdense} essentially has no dependence on $p$. 
A common key feature in recent works \cite{brennan2020rgg, liu2022rgg, liuracz} that consider $p = o(1)$ is to leverage the tensorization property (i.e., chain rule) of KL divergence, which allows ``local'' comparison between the models; for further details, see the technical overview in Section~\ref{sec:overview}. Our next result refines such approaches for our setting, which in addition has a community structure.

\begin{theorem}[Lower bound via tensorization]\label{thm:lbgeneral} There exists a constant $C_{\ref{thm:lbgeneral}} > 0$ such that the following holds: if
\[\frac{C_{\ref{thm:lbgeneral}} \log n}{k} \leq  p \leq \frac{1}{2} \quad \text{and} \quad d \geq C_{\ref{thm:lbgeneral}}\left(k^2p^2 \vee \frac{k^4p^2}{n}\right)\log^2(k/p) \log^2(1/p)\log^3n\,,
\]
no test achieves weak detection.
\end{theorem}
In terms of the dependence on $p$ we obtain a polynomial factor of $p^2$ for the threshold, which matches and extends (by considering $k = n$) the state-of-the-art results of $d = \widetilde{\Omega}(n^3p^2)$ \cite{liu2022rgg, baguley2025toroidal}.

\begin{remark}[Comparison between lower bounds]\label{remark:lbcomparison}
For lower bounds, we focus on the regime of $p \geq \widetilde{O}(1/k)$; for the sparse regime $p = \Theta(1/k)$, $\dTV(\calP, \calQ) = o(1)$ already holds for $d = \Omega(\mathrm{polylog}(k))$ \cite{liu2022rgg}. We choose not to pursue the case of $p = o(1/k)$, as the average degree within the community is already $o(1)$ there. 

Theorem~\ref{thm:lbgeneral} does not strictly extend Theorem~\ref{thm:lbdense}, in terms of its dependence on $k/n$. Indeed, it can be checked that depending on the size of $k$, Theorem~\ref{thm:lbdense} covers a wider regime: specifically, when $n^{1/2} \leq k = \widetilde{O}(n^{3/4})$ and $p = \widetilde{\Omega}(\sqrt{n}/k)$, 
 or $k = \widetilde{\Omega}(n^{3/4})$ and $p = \widetilde{\Omega}(k/n)$. This mainly comes from the differences in their underlying approaches; see Section \ref{sec:overview} for a detailed discussion. In brief, the proof of Theorem~\ref{thm:lbdense} essentially proceeds by bounding TV distance with $\chi^2$ divergence, which seems to be essential for capturing the tight dependence on the community size. This cannot be directly adapted to the proof of Theorem~\ref{thm:lbgeneral}:
 that comes at the cost of losing the chain-rule structure of KL divergence, which is essential for all existing approaches that capture dependence on $p$. We believe that improving the dependence on $k/n$ for $p = o(1)$ would require substantially new ideas, which we leave as an open question.
\end{remark}

\subsection{Computational lower bound}\label{subsec:maincomp}
While the global signed triangle count can clearly be calculated in polynomial time, the scan-based tests in general seem to require superpolynomial time as brute-force algorithms. This suggests the existence of a computational--statistical gap for our detection problem; we claim that this is indeed the case. 

Our analysis is based on the low-degree polynomial framework \cite{hopkins2018statistical, kunisky2019notes, wein2025computational}, which considers the following criterion for polynomials as test statistics.

\begin{definition}\label{def:strongsep} Let $f$ be a polynomial. A test statistic $f(G)$ achieves
\begin{itemize}
    \item[(a)] \emph{strong separation} if $\E_{G \sim \calP}[f(G)] - \E_{G \sim \calQ}[f(G)] =\omega\left( \sqrt{\Var_{G \sim \calQ}[f(G)]} \vee \sqrt{\Var_{G \sim \calP}[f(G)]}\right)$;
    \item[(b)] \emph{weak separation} if $\E_{G \sim \calP}[f(G)] - \E_{G \sim \calQ}[f(G)] = \Omega\left( \sqrt{\Var_{G \sim \calQ}[f(G)]} \vee \sqrt{\Var_{G \sim \calP}[f(G)]}\right)$.
\end{itemize}
\end{definition}

In the low-degree polynomial framework, a negative result for this criterion with degree $\omega(\log n)$ serves as evidence that no polynomial-time algorithms exist (for background, see Section~\ref{sec:priorwork}).

Recall from Theorem~\ref{thm:ubsignedtrianglecount} that detection can be done efficiently for $d = \widetilde{o}(k^6p^3/n^3)$. The following result complements this, showing that no low-degree polynomial can significantly improve that threshold even by weak separation.

\begin{theorem}[Low-degree lower bound]\label{thm:lowdeglb}
    Assume that there exists a constant $\delta > 0$ such that $d \geq n^{\delta}, n^{-1+\delta} \leq p \leq 1/2$. If there exists any constant $\eps > 0$ such that
\[d \geq \frac{k^6}{n^{3-\eps}}p^3\,,
\]
no degree-$\lfloor (\log n / \log (\log n))^2\rfloor$ polynomial achieves weak separation.
\end{theorem}

A related question is whether there are efficient algorithms other than the global signed triangle count. A natural extension of the signed triangle count is the class of signed cycle counts, frequently appearing in latent geometry detection \cite{bangachev2024geometry, bangachev2025random}. %
In the following proposition, we provide a negative answer, showing that any longer cycle count is strictly less powerful than the triangle count.

\begin{proposition}[Suboptimality of longer cycle counts]\label{prop:signedcycle}
    Let $\ell \geq 3$ and $d$ be sufficiently large with $d \geq (5\log(1/p))^4$. If the global signed count of length-$\ell$ cycle achieves strong separation, then
    \[d \ll \left(\frac{k^2p\log(1/p)}{n}\right)^{\ell/(\ell-2)}\,.
    \]
\end{proposition}
In this proposition, the signed triangle count succeeds for the largest range of $d$, as the right hand side is maximized at $\ell = 3$. This suggests that the global signed triangle count may be the asymptotically optimal efficient test.

\section{Technical overview}\label{sec:overview}
\paragraph*{Information-theoretic upper bound.}
A key feature of random geometric graphs is homophily: adjacent vertices share similar latent vectors, making their common neighbors more likely to be adjacent as well. As a result, geometric graphs contain more triangles than an \ER graph with the same edge density. 
Our test statistics are based on the signed triangle count $\sum_{i<j<\ell}(G_{ij}-p)(G_{j\ell}-p)(G_{i\ell}-p)$, which further exploits the homophily by centering each edge indicator. In particular, the mean under $\calP$ accurately reflects the geometric signal, while its variance remains small as the centering cancels out redundant contribution from the \ER background~\cite{bubeck2016rgg, liu2022rgg}. On top of this, the geometry is local in our setting, only confined to a hidden community of average size $k$. 
Together, these motivate a sequence of three tests, each adapted for expanding parameter regimes beyond those of the previous.

The global statistic $f_\mathrm{tri}$ \eqref{eq:signedtristatistic} sums signed triangles over the entire graph. 
The mean under $\calP$ is strictly positive due to the extra triangles within the community (Lemmas~\ref{lem:fouriersimple} and \ref{lem:signedtrianglelb}), while the variance is dominated by \ER fluctuations over the whole graph. As a technical remark, verifying the latter requires bounding the mean of signed 4-cycle counts under $\calP$, where we invoke a general bound for cycles of any length developed in this paper; see the last paragraph of this section for details. By comparing the mean and the variance using Chebyshev inequality,
we obtain the detection threshold $d \ll k^6p^3\log^3(1/p)/n^3$ of  Theorem~\ref{thm:ubsignedtrianglecount}.

A limitation of the global statistic is that $k$ needs to be large (e.g., polynomial in $n$) for the test to succeed; when the community is smaller, the geometry-based signal is already diluted by the ambient \ER noise. To resolve this issue, 
we introduce the scan statistic $f_\mathrm{scan}$ \eqref{eq:scanstatistic} which maximizes the signed triangle count over all subsets of size $k^-\approx k$. Under $\calP$, the planted community has a large signed triangle count with high probability by locally applying the same Chebyshev argument earlier. 
The new challenge lies in analyzing the error under $\calQ$: showing that no subset under $\calQ$ has an exceedingly large signed triangle count. 
In order to dominate a factor of $\binom{n}{k^-}$ from the union bound, we use a strong concentration inequality for polynomials of subgaussian variables (Lemma~\ref{lem:ertriconc}).
Altogether, we obtain the threshold $d = \widetilde{O}(k^2p^6)$ of Theorem~\ref{thm:ubscandense}. As noted earlier, for any fixed $p$ the global test and the scan test together characterize the optimal detection threshold up to a logarithmic factor. %

However, the scan test is not enough when $p$ vanishes. This is because the \ER background
can produce large signed triangle counts from its dense fluctuations without any geometric structure. This is not merely an artefact of our analysis but indeed a fundamental barrier; see Remark~\ref{remark:ubgeneralrole}.
Our approach is to characterize those anomalous subsets (roughly behaving as cliques) as having large signed wedge sums $W_{ij} := \sum_{\ell<i}(G_{\ell i}-p)(G_{\ell j}-p)$ for many pairs.
In particular, the constraints on $\sum_{i<j}W_{ij}^2$ and $\max_{i<j}|W_{ij}|$ (see \eqref{eq:constrainedscantestvar} and~\eqref{eq:constrainedscantestrange}) for the constrained scan statistics $\widetilde{f}_{\scan}$ \eqref{eq:constscanstatistic} play distinct roles: 
the former controls the variance in a Bernstein-type concentration argument, whereas the latter filters the aforementioned dense, clique-like patches
that produce large signed wedge counts.

A key technical step in those arguments is to verify that under $\calP$, the planted community lies in the constraint set $\calC(G)$ with high probability (Lemma~\ref{lem:ubgeneralrggwhp}). This amounts to showing that the signed wedge counts within the community concentrate around their typical values. For the constraint on $\sum_{i<j}W_{ij}^2$ (see~\eqref{eq:constrainedscantestvar}), this is done through a fine-grained control over the mean and the variance of small subgraphs formed by the wedges; for the constraint on $\max_{i<j}|W_{ij}|$ (see~\eqref{eq:constrainedscantestrange}), this is established via logarithmic-order moment bounds %
for the wedges. 
The rest of the argument is conceptually similar to those for the scan test, and we obtain the threshold $d = \widetilde{O}(k^2p^3)$ of Theorem~\ref{thm:ubscangeneral} which strictly improves upon the scan test for $p = \widetilde{o}(1)$.

\paragraph*{Information-theoretic lower bound.}
Establishing lower bounds for distinguishing $\calP$ and $\calQ$ is considerably more delicate, due to the strong 
dependencies among edges in the random geometric graph. 
Even in the classical model $\calG(n,p,d)$, the optimal detection threshold remains unresolved: the best known lower bound exhibits a polynomial gap in $p$ from the conjectured optimal upper bound~\cite{liu2022rgg}. In our localized setting---which is strictly more general---we derive lower bounds that match and generalize the strongest results currently available in that simpler model. We approach proving the lower bound in two complementary ways (see Remark~\ref{remark:lbcomparison}), each capturing a different aspect of the detection threshold.

The first approach captures dependence on the average community size $k$ and is particularly effective in the dense regime. Lower bounds for models with hidden structure are often obtained via the second moment of the likelihood ratio (i.e., the Ingster--Suslina method~\cite{ingster2003nonparametric}). In our setting, this calculation becomes tractable after viewing $\calP$ and $\calQ$ as pushforwards of classical random matrix ensembles, namely the Wishart distribution and the Gaussian Orthogonal Ensemble (GOE). While comparisons between these ensembles have appeared in previous works~\cite{bubeck2016rgg, bubeck2018entropicclt, raczrichey}, these are done at the levels of TV distance or KL divergence which are weaker than $\chi^2$-divergence as done here.
Furthermore,
a direct second-moment argument fails in our model, as %
contributions from rare configurations dominate. 
Our analysis is based on a truncation argument on the spectrum that only preserves the relevant bulk behavior; the second moment is then bounded by a tractable function of the overlap between two i.i.d. communities.

A further challenge introduced by the community structure for this approach is to compare the Wishart distribution against its spherical variant~\cite{paquette2021random}.
This arises from the fact that the community edges are generated by thresholding a normalized version of a Wishart matrix. 
The comparison here is subtle: only the strictly upper triangular entries should be considered,\footnote{The diagonal entries of a Wishart matrix have a continuous density, whereas those of a spherical Wishart matrix have a Dirac measure, implying that the TV distance between those is trivially equal to $1$.} and certain na\"ive comparisons are strictly suboptimal.\footnote{For example, comparing two matrices entrywise and invoking union bound requires $d \gg k^4$, and comparing both matrices with GOE requires $d \gg k^3$ \cite{jiang2015wishart, bubeck2016rgg}.} By directly analyzing their respective densities as well as the spectral properties, we
show those two distributions are indistinguishable when $d \gg k^2$ (Proposition~\ref{prop:sphwisandwis}). Notably, this identifies a regime where they are asymptotically equivalent to each other but are distinct from GOE---which may be of independent interest. Taken together with the second moment analysis earlier, this refined matrix comparison yields the optimal detection lower bound in the dense regime (Theorem~\ref{thm:lbdense}) up to a logarithmic factor. 

The second approach captures dependence on the edge density $p$, as the first lower bound is essentially independent of $p$. Following the strategies of~\cite{liu2022rgg, liuracz} for the full model (i.e., $k = n$), 
our starting point is to view
the random graph as generated by revealing one vertex at a time
and then bound the TV distance via the tensorization property of KL divergence. In particular, we extend the earlier strategies %
to our setting by incorporating the decision of community membership into the sequential process: the latent variable for each vertex is %
augmented to jointly encode both its feature vector $U_i$ and its community membership indicator $V_i$.  The key leverage comes from the average-case structure: each vertex belongs to the community with probability $k/n$, and conditional on the community membership, only $O(k)$ previously revealed vertices are relevant. These features are captured through a careful truncation argument, yielding the $p$-dependent lower bound of Theorem~\ref{thm:lbgeneral} that generalizes the $k=n$ case. 

An important technical step is to bound the squared likelihood ratio between the neighborhood distributions under $\calP$ (conditioned on the latents) and $\calQ$. This was previously done only under $d = \widetilde{\Omega}(k^3p^3)$, which is sufficient for the $k = n$ case but falls short for our more general setting. Through a new analysis, we improve this condition to $d = \widetilde{\Omega}(k^2p^2)$.
The main difference lies in the strategy for using the martingale structure of the likelihood ratio. By leveraging the concentration property of each martingale difference, we build a pair of recursive inequalities that directly characterizes the second moment and is tighter than a bound that is only based on the concentration of the martingale itself \cite{liu2022rgg}.

\paragraph*{Computational lower bound.}
We provide evidence that no polynomial-time algorithm may succeed beyond the regime of the global test, based on the low-degree polynomial framework~\cite{hopkins2018statistical}; see Sections~\ref{sec:priorwork} and~\ref{subsec:maincomp} for background. In our setting, controlling all low-degree polynomials reduces to bounding the squared Fourier coefficients $\Phi_\calP(H)^2$ summed over all subgraphs $H$ of polylogarithmic size. 
Our analysis is based on two key ideas: for forests, the contribution is zero by Corollary~\ref{cor:treesimple}; for the remaining subgraphs, a general-purpose moment bound (from~\cite{bangachev2024fourier}; Lemma~\ref{lem:rggfourier}) gives an upper bound depending only on $v(H)$ and $e(H)$. 
Summing over all possible combinations of $(v(H), e(H))$ then establishes the result (Theorem~\ref{thm:lowdeglb}).

A related question is whether any efficient test other than the global signed triangle count can match its performance. We show that any longer signed cycle count performs strictly worse (Proposition~\ref{prop:signedcycle}). This follows from a tight estimate for the expectation of the signed cycle count of any length (Proposition~\ref{prop:signedcyclecount}); this result may be of independent interest, and earlier we used the result for the $4$-cycle in order to analyze our tests. The main idea is to expand the threshold function $\bm{1}\{\langle U_i, U_j\rangle \geq \tau(p,d)\}$ in the orthonormal basis of spherical harmonics and Gegenbauer polynomials. By orthonormality, the expectation of the signed cycle count reduces to an infinite sum $\sum_{m=1}^\infty c_m^\ell / N_m^{\ell/2-1}$, where $c_m$ are the Gegenbauer coefficients of the threshold function and $N_m$ are the multiplicities of the spherical harmonics. This sum is dominated by the $m=1$ term, resulting in the tight estimate $p^\ell \log^{\ell/2}(1/p) / d^{\ell/2-1}$ up to multiplicative factors of $(1/C)^\ell$ and $C^\ell$ for a constant $C$. %
Using this to compare the mean and the variance of signed cycle counts, we show that the $\ell=3$ case covers the widest regime, suggesting that the global signed triangle count is the asymptotically optimal test among all signed cycle counts.

\section{Proofs for information-theoretic upper bound}\label{sec:ub}
In this section, we prove the upper bounds presented in Section~\ref{subsec:mainub}. 

Before presenting proofs, we review some elementary facts of the model $\calP = \calG(n, p, d, k)$.
First, we can write $\calP$ as a mixture
\begin{equation}\label{eq:PandPS}
    \calP = \E_S [\calP_S]\,, 
\end{equation}
where for each fixed $S \subseteq [n]$, $\calP_S$ is a distribution drawn as in Step 3 in Definition~\ref{def:planteddist}. 

Instead of drawing a community of variable size $\Binom(n, k/n)$, it is often useful to consider a fixed-size variant.  
For integer $0 \leq s \leq n$, define
\begin{equation}\label{eq:fixedsizevariant}
    \calP'_s := \E_{S \big| |S| = s} [\calP_S]\,.
\end{equation}
Note that after conditioning on $|S| = s$, $S$ is chosen uniformly at random from $\binom{[n]}{s}$. The size $s$ can be considered to be around $k$ in the following sense: for 
\[k^- := \lfloor 0.9k\rfloor, \quad k^+ := \lceil 1.1k \rceil\,,
\]
any event $E$ that only depends on $G$ satisfies 
\begin{equation}\label{eq:whpcommunitysize}
\begin{aligned}
    \Prob_{G \sim \calP}(E) &\leq \Prob(|S| \notin [k^-, k^+]) + \max_{s \in [k^-, k^+]}\Prob_{G \sim \calP'_s}(E) \\
    &\leq 2\exp(-\Omega(k)) + \max_{s \in [k^-, k^+]}\Prob_{G \sim \calP'_s}(E)\,, 
\end{aligned}
\end{equation}
where the second inequality is by applying Chernoff bound.

The key quantities throughout our analysis are the expectations of signed subgraph counts. For the \ER graph $\calQ$, they are always $0$ by independence. For $\calP = \calG(n, p, d, k)$, we first note that they are directly related to the expectations under $\calG(n, p, d)$ in the following sense.

\begin{lemma}\label{lem:fouriersimple} Let $H$ be a subgraph of $K_n$ without isolated vertices. Then, 
\[\E_{G \sim \calP}\left[\prod_{ij \in E(H)}(G_{ij}-p)\right] = \left(\frac{k}{n}\right)^{v(H)}\E_{G \sim \calG(n, p, d)}\left[\prod_{ij \in E(H)}(G_{ij}-p)\right]\,.
\]\
\end{lemma}

\begin{proof}
We have
\[\E_{G \sim \calP}\left[\prod_{ij \in E(H)}(G_{ij} - p)\right] 
    = \E_S\left[\E_{G \sim \calP_S}\left[\prod_{ij \in E(H)}(G_{ij} - p)\right]\right]\,.
\]
If $V(H) \not\subseteq S$, the inner expectation is $0$ since the centered factor for any edge leaving $S$ has mean $0$ and is independent of the rest. 
If $V(H) \subseteq S$, the inner expectation equals 
$\E_{G \sim \calG(n,p,d)}[\prod_{ij\in E(H)}(G_{ij}-p)]$ since $\calP_S$ 
agrees with $\calG(n,p,d)$ on $V(H)$. Since $\Prob(V(H)\subseteq S) = (k/n)^{v(H)}$, 
the result follows. 
\end{proof}
\begin{corollary}\label{cor:treesimple} Let $H$ be a subgraph of $K_n$ that is a forest. Then $\E_{G \sim \calP}\left[\prod_{ij \in E(H)}(G_{ij}-p)\right] = 0$, or equivalently, $\E_{G \sim \calP}\left[\prod_{ij \in E(H)}G_{ij}\right] = p^{e(H)}$.    
\end{corollary}

Among different subgraphs, the triangle ($K_3$) plays a central role through our analysis. While the expectation of signed triangle is $0$ under $\calQ$, it is strictly positive\footnote{The reference in Lemma \ref{lem:signedtrianglelb} assumes $p \in (0, 0.49)$; for $p \in [0.49, 0.5]$, one can apply Proposition~\ref{prop:signedcyclecount}.} under $\calG(n, p, d)$ and thus also under $\calP$.
\begin{lemma}[{\cite[Lemma 1]{bubeck2016rgg} and \cite[Claim A.2]{liu2022rgg}}]\label{lem:signedtrianglelb} There exists a universal constant $C_{\ref{lem:signedtrianglelb}} > 0 $ such that for all sufficiently large $d$ and $0 < p \leq 1/2$,
\[\E_{G \sim \calG(n, p, d)}\left[(G_{12}-p)(G_{23}-p)(G_{13}-p)\right] \geq \frac{p^3 \log^{3/2}(1/p)}{C_{\ref{lem:signedtrianglelb}}\sqrt{d}}\,.
\]
\end{lemma}

\subsection{Global test}\label{subsec:signedtrianglecount}
In this subsection, we investigate the global test, proving Theorem~\ref{thm:ubsignedtrianglecount}.

Recall the definition of $f_{\mathrm{tri}}(G) = \sum_{i < j < \ell \in [n] }(G_{ij}-p)(G_{j\ell}-p)(G_{i\ell}-p)$. We will calculate the first and second moments of $f_{\mathrm{tri}}(G)$ under $\calP$ and $\calQ$, in order to use Chebyshev's inequality. A straightforward calculation (see, for example, \cite[Lemma A.1]{liu2022rgg}) yields
    \begin{align*}
        \E_{G \sim \calQ}[f_{\mathrm{tri}}(G)] &= 0\,, \\
        \Var_{G \sim \calQ}[f_{\mathrm{tri}}(G)] &= \binom{n}{3}p^3(1-p)^3\,.
    \end{align*}

\paragraph*{Type I error.} We have
\[\Prob_{G \sim \calQ}(f_{\tri}(G) > \gamma_{\tri}) \leq \frac{\Var_{G \sim \calQ}[f_{\tri}(G)]}{\gamma_{\tri}^2}\,,
\]
where by Lemmas~\ref{lem:fouriersimple} and \ref{lem:signedtrianglelb},
\[\gamma_{\tri} = \frac{1}{2}\E_{G \sim \calP}[f_{\tri}(G)] \geq \frac{1}{2} \times  \binom{n}{3} \left( \frac{k}{n}\right)^3  \frac{p^3 \log^{3/2}(1/p)}{C_{\ref{lem:signedtrianglelb}}\sqrt{d}} = \Omega\left(\frac{k^3p^3\log^{3/2}(1/p)}{\sqrt{d}}\right)\,.
\]
Thus, 
\[\frac{\Var_{G \sim \calQ}[f_{\tri}(G)]}{\gamma_{\tri}^2} = O\left(\frac{dn^3}{k^6p^3\log^3{(1/p)}}\right)\,,
\]
where the last term is $o(1)$ if $d \ll \frac{k^6p^3}{n^3}\log^3(1/p)$.

\paragraph*{Type II error.} Similar to the type I error, from $2\gamma_{\tri} = \E_{G \sim \calP}[f_{\tri}(G)]$ we have
\[\Prob_{G \sim \calP}(f_{\tri}(G) \leq  \gamma_{\tri}) = \Prob_{G \sim \calP}(f_{\tri}(G) - \E_{G \sim \calP}[f_{\tri}(G)] \leq -\gamma_{\tri}) \leq   \frac{\Var_{G \sim \calP}[f_{\tri}(G)]}{\gamma_{\tri}^2}\,.
\]
Thus, it suffices to control the variance of the global signed triangle count under $\calP$.  We state this as a lemma, as it will be used elsewhere. %

    \begin{lemma}\label{lem:signedtrianglecountvar} If $1/k \leq p \leq 1/2$ and $d$ is sufficiently large with  $(5\log(1/p))^4 \leq d \ll k^6p^3 \log^3(1/p) / n^3$, then
        \[\gamma_{\tri} \gg \sqrt{\Var_{G \sim \calP}[f_{\tri}(G)]}\,.
        \]
    \end{lemma}
    The proof of Lemma~\ref{lem:signedtrianglecountvar} is deferred to Appendix~\ref{appendix:signedtrianglecountvar}. By the lemma, the type II error is $o(1)$ and this completes the proof.

\subsection{Scan test}

Here, we analyze the performance of the scan test and prove Theorem~\ref{thm:ubscandense}.

        \paragraph{Type I error.} For $A \subseteq [n]$, let $f_A(G) := \sum_{i < j < \ell \in A}(G_{ij}-p)(G_{j \ell}-p)(G_{i \ell}-p)$. Then,
    \begin{align*}
        \Prob_{G \sim \calQ}(f_{\scan}(G) > \gamma_{\scan}) &= \Prob_{G \sim \calQ}\left(\max_{A \in \binom{[n]}{k^-}} f_A(G) > \gamma_{\scan}\right) \\
        &\leq \binom{n}{k^-}\Prob_{G \sim \calQ}\left(f_{[k^-]}(G)> \gamma_{\scan}\right) \\
        &\leq n^{k^-} \cdot 2\exp\left(-\frac{1}{C_{\ref{lem:ertriconc}}}\left(\frac{\gamma_{\scan}^2 \log^3(1/p)}{(k^-)^3} \wedge \frac{\gamma_{\scan} \log^{3/2}(1/p)}{\sqrt{k^-}} \wedge \gamma_{\scan}^{2/3} \log(1/p)\right)\right) \\
        &\leq \exp\left(k\log n - \Omega\left(\frac{k^3p^6\log^6(1/p)}{d} \wedge \frac{k^{5/2}p^3\log^3(1/p)}{\sqrt{d}} \wedge \frac{k^2p^2\log^2(1/p)}{d^{1/3}}\right)\right) \\
        & \leq  \exp\left(k\log n - \Omega\left(\frac{k^3p^6\log^6(1/p)}{d} \wedge \frac{k^2p^2\log^2(1/p)}{d^{1/3}}\right)\right) \,, 
    \end{align*}
    where the first inequality holds by the union bound, the second inequality holds by Lemma~\ref{lem:ertriconc}, the third inequality holds by Lemma~\ref{lem:signedtrianglelb}, and the last inequality holds by $d \leq k^3p^6\log^6(1/p)$ from the assumption. 
    \begin{itemize}
        \item If $d \geq k^{3/2}p^6\log^6(1/p)$, the first term $\frac{k^3p^6\log^6(1/p)}{d}$ is the minimum. Then $\Prob_{G \sim \calQ}(f_{\scan}(G) > \gamma_{\scan}) \leq \exp(-k\log n) = o(1)$ if $d \le \frac{k^2p^6\log^6(1/p)}{C_0 \log n}$, by choosing a sufficiently large $C_0>0$.
        \item If $d < k^{3/2}p^6\log^6(1/p)$, the second term $\frac{k^2p^2\log^2(1/p)}{d^{1/3}}$ is the minimum. Then $\Prob_{G \sim \calQ}(f_{\scan}(G) > \gamma_{\scan}) \leq \exp(-k\log n) = o(1)$ if $d \leq \frac{k^3p^6\log^6(1/p)}{C_1\log^3 n}$, by choosing a sufficiently large $C_1 > 0$. This is implied by the preceding inequality given $k \geq C_{\ref{thm:ubscandense}} \log^2 n$, for any constant $C_{\ref{thm:ubscandense}} > C_1/C_0$. %
    \end{itemize}
    Thus, any choice of $C_{\ref{thm:ubscandense}} > (C_1/C_0) \vee C_0 \vee C_1$ suffices for $\mathbb{P}_{G \sim \calQ}(f_{\scan}(G) > \gamma_{\scan}) = o(1)$.
    \paragraph{Type II error. } 
    By \prettyref{eq:whpcommunitysize} with event $E = \{f_{\scan}(G) \leq \gamma_{\scan}\}$, it suffices to show that 
    \[
    \max_{s \in [k^-, k^+]}\Prob_{ G\sim \calP'_s }\left(\max_{A \in \binom{[n]}{k^-}} f_A(G) \leq \gamma_{\scan} \right)= o(1)\,,
    \]
   where $\calP'_s$ is a mixture of $\calP_S$ with $S$ chosen uniformly over size-$s$ sets. 
     For any such $S$, there exists a subset $S_0 \subseteq S$ of size exactly $k^-$. Consider any deterministic rule for choosing such $S_0$ (\eg, the smallest $k^-$ elements of $S$). Since the maximum over all subsets is at least the value on this specific $S_0$, we have $\max_{A \in \binom{[n]}{k^-}} f_A(G) \geq f_{S_0}(G)$. Thus,
    \begin{equation}\label{eq:ubdensergg}
    \begin{aligned}
        \max_{s \in [k^{-}, k^+]} \Prob_{G \sim \calP'_s}(E) 
        &\leq \max_{|S| \in [k^{-}, k^+]} \Prob_{G \sim \calP_S}(f_{S_0}(G) \leq  \gamma_{\scan})\\
        &= \max_{|S| \in [k^{-}, k^+]} \Prob_{G \sim \calP_S}(f_{S_0}(G) \leq \E_{G\sim \calP_{S}}[f_{S_0}(G)] - \gamma_{\scan}) \\
        &\leq \frac{\Var_{G \sim \calP_{S_0}}[f_{S_0}(G)]}{\gamma_{\scan}^2}\,,
    \end{aligned}
    \end{equation}
    from $\E_{G \sim \calP_{S}}[f_{S_0}(G)] = 2\gamma_{\scan}$ and Chebyshev's inequality. For the last term of \eqref{eq:ubdensergg} being $o(1)$, 
    we apply Lemma~\ref{lem:signedtrianglecountvar} to $f_{S_0}$ by treating $G[S_0]$ as a stand-alone graph of size $k^-$ with a full community.
    
    The required conditions of Lemma~\ref{lem:signedtrianglecountvar} hold by the assumption of the theorem. The lower bound on $d$ clearly follows. The condition $p \geq 1/k^-$ follows from $1 \leq d = O(k^2p^6\log^6(1/p)/\log n) \leq O(k^2p^2/\log n)$  which implies $p = \Omega(\sqrt{\log n}/k)$; the condition $d \ll (k^-)^3p^3\log^3(1/p)$ follows from $k \gg 1/\log n = \Omega(p^3\log^3(1/p)/\log n)$ which implies $d = O(k^2p^6\log^6(1/p)/\log n) \ll k^3p^3\log^3(1/p)$.

\subsection{Constrained scan test}
Finally, we consider the constrained scan test and prove Theorem~\ref{thm:ubscangeneral}.

Recall (from the previous subsection) the notation $f_A(G) = \sum_{i < j < \ell \in A} (G_{ij}-p)(G_{j\ell}-p)(G_{i\ell}-p)$, for $A \subseteq [n]$. Here we first analyze the type II error, as the proof is conceptually simpler. 
\paragraph*{Type II error.}
We proceed similarly as in \eqref{eq:ubdensergg}; the difference here is in the new error event %
\[E = \left\{\widetilde{f}_{\scan}(G) \leq \gamma_{\scan}\right\} = \left\{\max\left\{ f_A(G): A \in \binom{[n]}{k^-}, \eqref{eq:constrainedscantestvar} \text{ and } \eqref{eq:constrainedscantestrange} \text{ hold} \right\} \leq \gamma_{\scan} \right\}\,,
\]
and from \eqref{eq:whpcommunitysize} it suffices to show that
\[\max_{s \in [k^-, k^+]}\Prob_{G \sim \calP'_s}(E) = o(1)\,.
\]
For each $S$ where $G \sim \calP_S$, consider a subset $S_0 \subseteq S$ with size $k^-$ that is chosen by a deterministic rule. The key claim here is that for $A = S_0$, \eqref{eq:constrainedscantestvar} and \eqref{eq:constrainedscantestrange} holds with high probability, implying that $\widetilde{f}_{\scan}(G) \geq f_{S_0}(G)$ under that event. Conceptually, this holds because the conditional variance (the left hand side of \eqref{eq:constrainedscantestvar}) and the maximum difference (the left hand side of \eqref{eq:constrainedscantestrange}) concentrate around their expectations within the planted community $S_0$, and the parameters $\sigma^2$ and $B$ are chosen to bound these typical values.

\begin{lemma}\label{lem:ubgeneralrggwhp} Let $S_0 \subseteq S \subseteq [n]$ be such that $|S| \in [k^-, k^+]$ and $|S_0| = k^-$. Also, let $E' = E'(S_0)$ be the event such that \eqref{eq:constrainedscantestvar} and \eqref{eq:constrainedscantestrange} hold for $A = S_0$. Then under the assumptions of Theorem~\ref{thm:ubscangeneral},
\[\Prob_{G \sim \calP_S}(E') = 1 - o(1)\,.
\]
The $o(1)$ term is universal over all choices of $S_0$ and $S$.
\end{lemma}
The proof is deferred to Appendix~\ref{appendix:ubgeneralrggwhp}. 
By Lemma~\ref{lem:ubgeneralrggwhp}, we have
\begin{align*}
        \max_{s \in [k^{-}, k^+]} \Prob_{G \sim \calP'_s}(E) &\leq \max_{|S| \in [k^{-}, k^+]} \left(\Prob_{G \sim \calP_S}(f_{S_0}(G) \leq  \gamma_{\scan}) + \Prob_{G \sim \calP_S}\left((E')^c\right)\right)\\
        &\leq \max_{|S| \in [k^{-}, k^+]} \Prob_{G \sim \calP_S}(f_{S_0}(G) \leq \E_{G\sim \calP_{S}}[f_{S_0}(G)] - \gamma_{\scan}) + o(1) \\
        &\leq \Var_{G \sim \calP_{S}}[f_{S_0}(G)] / \gamma_{\scan}^2 + o(1)\,,
    \end{align*}
following the same lines of arguments in \eqref{eq:ubdensergg}. In particular, the same arguments right after \eqref{eq:ubdensergg} imply that the last term vanishes whenever $d \ll k^3p^3 \log^3(1/p)$, and $d$ is sufficiently large with $d \geq (5\log(1/p))^4$; these already hold from the theorem's assumption on $d$.

\paragraph*{Type I error.} For the error under $\calQ$, we have
\begin{equation}\label{eq:ubgeneraler}
\begin{alignedat}{2}
    \Prob_{G \sim \calQ}\left(\widetilde{f}_{\scan}(G) > \gamma_{\scan}\right) &\leq \binom{n}{k^-} \Prob_{G \sim \calQ}\Bigg(&&f_{[k^-]}(G) > \gamma_{\scan}, \\
    & &&\sum_{i < j \in [k^-]} \left(\sum_{\ell < i, \ell \in [k^-]} (G_{\ell i}-p)(G_{\ell j}-p)\right)^2 \leq \sigma^2, \\
    & &&\max_{i<j\in [k^-]} \left|\sum_{\ell < i, \ell \in [k^-]} (G_{\ell i}-p)(G_{\ell j}-p)\right| \leq B\Bigg)\,,
\end{alignedat}
\end{equation}
by union bound and symmetry (letting $A = [k^-]$).

Now we construct a martingale. First, let $\{\calF_r\}$ be the filtration defined by exposing the edges one at a time, in the order of
\[G_{12}, G_{13}, G_{23}, G_{14}, G_{24}, G_{34}, \dots, G_{(k^-)-1, k^-}\,.
\]
Furthermore, let $\{M_r\}$ be the Doob martingale of $f_{[k^-]}(G)$ with respect to the $\{\calF_r\}$, i.e., $M_r = \E_{G \sim \calQ}[f_{[k^-]}(G) | \calF_r]$. Consider the time $r$ at which edge $ij$ with $i < j$ is revealed. Then 
\[U_r :=\left|\sum_{\ell < i, \ell \in [k^-]} (G_{\ell i}-p)(G_{\ell j}-p)\right|
\]
is $\calF_{r-1}$-measurable and serves as an upper bound for the martingale difference, in that
\begin{equation}\label{eq:ubgeneralerrange}
    M_r - M_{r-1} = (G_{ij}-p)\sum_{\ell < i, \ell \in [k^-]}(G_{\ell i}-p)(G_{\ell j}-p) \leq U_r\,.
\end{equation}
Also, the conditional variance of the martingale difference is equal to
\begin{equation}\label{eq:ubgeneralervar}
\begin{aligned}
    \Var_{G \sim \calQ}[M_r - M_{r-1}|\calF_{r-1}] &= \E_{G \sim \calQ}[(G_{ij}-p)^2]\left(\sum_{\ell < i, \ell \in [k^-]}(G_{\ell i}-p)(G_{\ell j}-p)\right)^2 \\
    &= p(1-p)\left(\sum_{\ell < i, \ell \in [k^-]}(G_{\ell i}-p)(G_{\ell j}-p)\right)^2\,.
\end{aligned}
\end{equation}
From \eqref{eq:ubgeneralerrange} and \eqref{eq:ubgeneralervar}, we can apply Freedman's inequality (Lemma~\ref{lem:freedmanconc}), with parameters
\begin{align*}
    t &= \gamma = \Theta(k^3p^3\log^{3/2}(1/p)/\sqrt{d})\,, \\
    v &= p(1-p)\sigma^2 = \Theta(k^3p^3 + k^4p^5\log^2(1/p)/d)\,, \\
    m &= B = \Theta((kp^2 + 1)\log k)\,.
\end{align*}
Recall that the choice of $t$ is from Lemma~\ref{lem:signedtrianglelb}, and $v$ and $m$ are respectively from \eqref{eq:constrainedscantestvar} and \eqref{eq:constrainedscantestrange}; for further discussion on these choices, see Remark \ref{remark:ubgeneralrole}.
From \eqref{eq:ubgeneraler}, 
\begin{align*}
    &\Prob_{G \sim \calQ}\left(\widetilde{f}_{\scan}(G) > \gamma_{\scan}\right) \\
    &\leq \exp\left(k\log n - \left(\frac{\gamma_{\scan}^2}{4v} \wedge \frac{\gamma_{\scan}}{2B}\right)\right) \\
    &\leq \exp\left(k\log n - \Omega\left(\frac{k^3p^3\log^3 (1/p)}{d} \wedge k^2p\log(1/p) \wedge \frac{k^2p\log^{3/2}(1/p)}{\sqrt{d}\log k} \wedge \frac{k^3p^3\log^{3/2} (1/p)}{\sqrt{d}\log k}\right)\right)\,.
\end{align*}
This term is $o(1)$, by a proper choice of constant $C_0 > 0$ such that the following holds:
\begin{align}
    d &\leq \frac{1}{C_0}\left(\frac{k^2p^3\log^3(1/p)}{\log n} \wedge \frac{k^2p^2 \log^3 (1/p)}{(\log^2 n)(\log^2 k)} \wedge \frac{k^4p^6\log^3 (1/p)}{(\log^2 n)(\log^2 k)}\right), \label{eq:ubgeneralfinalcond1} \\
    kp\log(1/p) &\geq C_0\log n \,. \label{eq:ubgeneralfinalcond2}
\end{align}
If $d \leq k^2p^3\log^3(1/p)/(C_0(\log^2 n) (\log^2 k))$, from $n^\delta \leq d$ and $1/p \leq n^{1-\delta}$ we have $k^2p^3 \geq C_0n^{\delta} \log^2 k/((1-\delta) \log n)$, which is larger than $(\log n) (\log^2k)$ for all sufficiently large $n$. For \eqref{eq:ubgeneralfinalcond1}, this implies that the third term cannot be minimum, and any choice of $C_{\ref{thm:ubscangeneral}} > C_0$ suffices. For \eqref{eq:ubgeneralfinalcond2}, we have $k^3p^3\log^3(1/p) \geq k^2p^3\log^3(1/p) \geq  C_0 n^\delta (\log^2n)(\log^2 k)$, where any choice of $C_{\ref{thm:ubscangeneral}} > C_0^{1/3}$ suffices for all sufficiently large $n$.

\begin{remark}[Role of the constraints]\label{remark:ubgeneralrole}
We note that such constraints may be necessary, as the suboptimality of the unconstrained scan test (Theorem~\ref{thm:ubscandense}) seems to be inherent. Here we provide an explanation; as context, we refer to the proof of Theorem~\ref{thm:ubscandense}. There, for the error to vanish under $\calQ$, the event of signed triangle count being larger than $t = \widetilde{\Theta}(k^3p^3/\sqrt{d})$ should happen with very small probability---at most $1/\binom{n}{k^-} = \exp(-\widetilde{\Omega}(k))$. However, this event can be attained if a clique of size $t^{1/3}$ exists, which happens with probability at least $\exp(-\widetilde{O}(t^{2/3})) = \exp(-\widetilde{O}(k^2p^2/d^{1/3}))$ (also consistent with the large-deviation type scaling of the event). Combined, this requires $d = \widetilde{O}(k^3p^6)$, which is strictly worse than the threshold $d = \widetilde{O}(k^2p^3)$ we obtain for the constrained scan test when $p = \widetilde{o}(1/k^{1/3})$.

In this sense, the constraints can be considered as preventing events that are bad for concentration. For example, it can be checked that for $p = \widetilde{o}(1/k^{1/3})$, the size-$t^{1/3}$ clique implies that the left hand side of \eqref{eq:constrainedscantestrange} is at least $\Omega(t^{1/3})$, violating the corresponding condition when $d = \widetilde{O}(k^2p^3)$.
\end{remark}

\section{Proofs for information-theoretic lower bound}\label{sec:lb}

\subsection{Truncated second moment}

Our first approach is to view both $\calP$ and $\calQ$ as generated by thresholding certain random matrices. 
This essentially reduces the detection problem between two (binary-valued) \emph{random graphs} to a detection problem between two (real-valued) \emph{random matrices} \cite{bubeck2016rgg}.

Similar to \eqref{eq:whpcommunitysize}, we begin with
\begin{align*}
    \dTV(\calP, \calQ) &= \dTV(\E_{s \sim \Binom(n, k/n)} [\calP'_{s}], \calQ) \\
    &\leq \E_{s \sim \Binom(n, k/n)}[\dTV(\calP'_{s}, \calQ)] \\ 
        &\leq \Prob_{s \sim \Binom(n, k/n)}(s \notin [k^-, k^+]) + \max_{s \in [k^{-}, k^+]} \dTV(\calP'_{s}, \calQ)\\
        &\leq 2\exp(-\Omega(k)) + \max_{s \in [k^{-}, k^+]} \dTV(\calP'_{s}, \calQ)\,,
\end{align*}
where the first inequality is by Jensen inequality, the second inequality is by $\dTV \leq 1$, and the last inequality is by Chernoff bound. Thus, it suffices to show that the TV distance vanishes uniformly over $s \in [k^-, k^+]$.

\paragraph*{From random graphs to random matrices.} Now we consider the random matrices from which $\calP'_s$ and $\calQ$ are generated. For $\calQ$, let $M \sim dI_n + \sqrt{d}\GOE(n)$, i.e., %
for any $i,j\in[n]$ with $i\leq j$, 
\[
M_{ij} \overset{\text{ind.}}{\sim} d\bm{1}\{i = j\} + \sqrt{d}\calN(0, 1 + \bm{1}\{i = j\})\,,
\]
and $M_{ji}=M_{ij}$. 
Also, define the map $\alpha : \mathbb{R}^{n(n+1)/2} \to \{0, 1\}^{n(n-1)/2}$ which thresholds the off-diagonal entries of a symmetric matrix $X$ as follows: 
\[
\alpha(X)_{ij} := \bm{1}\{X_{ij} \geq \sqrt{d}\Phi^{-1}(1-p)\}, \quad \text{for all } i < j \in [n]\,,
\]
where $\Phi$ is the cumulative distribution function of $\calN(0, 1)$. Then we have 
\[\alpha(M) \sim \calQ = \calG(n, p)\,,
\]
because for $\alpha(M)$, each edge is drawn independently with probability $p$. For $\calP$, the random matrix should follow different distributions depending on the community membership. For $S \subseteq [n]$, define a symmetric random matrix $W^S \in \reals^{n \times n}$ whose entries are given as
\[(W^S)_{ij} := \begin{cases}
    W_{ij} & i,j \in S \\
    M_{ij} & \text{else}\,,
\end{cases}
\]
where $W \sim \Wishart(n, d)$, i.e., for any $i,j \in [n]$, 
\[
W_{ij} = \iprod{Z_i}{Z_j}\,, 
\]
where $Z_i \iiddistr \calN(0, I_d)$ for $i\in [n]$. The edge $ij$ within the community of $\calP$ is then realized by thresholding $W_{ij}/\sqrt{W_{ii}W_{jj}} = \iprod{Z_i/\norm{Z_i}}{Z_j/\norm{Z_j}}$. In particular, define $\beta: \reals^{n(n+1)/2} \to \{0, 1\}^{n(n-1)/2}$ as 
\[\beta(X)_{ij} := \bm{1}\left\{X_{ij}/\sqrt{X_{ii}X_{jj}} > \tau(p, d)\right\}, \quad \text{for all }   i < j \in [n]\,.
\]
To generate a sample from $\calP$, $\beta$ should be used for edges within the community, and $\alpha$ should be used otherwise. Formally, by defining a map $\beta^S: \reals^{n(n+1)/2} \to \{0, 1\}^{n(n-1)/2}$ such that
\[\beta^S(X)_{ij} = \beta(X)_{ij}\bm{1}\{i, j \in S\} + \alpha(X)_{ij}\bm{1}\{i \notin S \text{ or } j \notin S\}\,,
\]
we have
\[\beta^S(W^S) \sim \calP_S\,.
\]
In this sense, with a slight overload of notation, $\calP'_s$ can be written as $\calP'_s = \E_{S \big| |S| = s} [\beta^S(W^S)]$, which is a mixture distribution. 
Then
\begin{equation}\label{eq:lbdensetwoterms}
\begin{aligned}
    \dTV(\calP'_s, \calQ) &= \dTV\left(\E_{S \big| |S| = s} [\beta^S(W^S)], \alpha(M)\right) \\
    &\leq \dTV\left(\E_{S \big| |S| = s} [\beta^S(W^S)], \E_{S \big| |S| = s} [\alpha(W^S)]\right) + \dTV\left(\E_{S \big| |S| = s} [\alpha(W^S)], \alpha(M)\right) \\
    &\leq \underbrace{\dTV\left(\E_{S \big| |S| = s} [\beta^S(W^S)], \E_{S \big| |S| = s} [\alpha(W^S)]\right)}_{\text{(I)}} + \underbrace{\dTV\left(\E_{S \big| |S| = s} [W^S], M\right)}_{\text{(II)}}\,,
\end{aligned}
\end{equation}
where the last inequality is by data processing inequality applied to the map $\alpha$, after observing that the distribution of $\E_{S \big| |S| = s}[\alpha(W^S)]$ (mixture of pushforward) is equal to $\alpha\left(\E_{S \big| |S| = s} [W^S]\right)$ (pushforward of mixture). 

For the rest of the proof, we show that both (I) and (II) vanish uniformly over $s \in [k^-, k^+]$.

\paragraph*{First term (I).} By convexity (and Jensen's inequality) and symmetry, (I) in \eqref{eq:lbdensetwoterms} is at most
\begin{equation}\label{eq:lbdensewts1}
    \dTV\left(\beta^S(W^S), \alpha(W^S)\right) = \dTV\left((\beta^{[s]}(W^{[s]})_{ij})_{i< j \in [s]}, (\alpha(W^{[s]})_{ij})_{i<j \in [s]}\right)\,,
\end{equation}
An equivalent, explicit form of \eqref{eq:lbdensewts1} is
\begin{equation}\label{eq:lbdensewts1explicit}
    \dTV\left(\left(\bm{1}\left\{\frac{\iprod{Z_i}{Z_j}}{\norm{Z_i}\norm{Z_j}} > \tau(p, d)\right\}\right)_{i < j \in [s]}, \left(\bm{1}\left\{\frac{\iprod{Z_i}{Z_j}}{d} > \frac{\Phi^{-1}(1-p)}{\sqrt{d}}\right\}\right)_{i < j \in [s]}\right)\,.
\end{equation}
By triangle inequality, \eqref{eq:lbdensewts1explicit} can be further upper bounded as
\begin{align*}
    &\dTV\left(\left(\bm{1}\left\{\frac{\iprod{Z_i}{Z_j}}{\norm{Z_i}\norm{Z_j}} > \tau(p, d)\right\}\right)_{i < j \in [s]}, \left(\bm{1}\left\{\frac{\iprod{Z_i}{Z_j}}{\norm{Z_i}\norm{Z_j}} > \frac{\Phi^{-1}(1-p)}{\sqrt{d}}\right\}\right)_{i < j \in [s]}\right) \\
    &\quad+ \dTV\left(\left(\bm{1}\left\{\frac{\iprod{Z_i}{Z_j}}{\norm{Z_i}\norm{Z_j}} > \frac{\Phi^{-1}(1-p)}{\sqrt{d}}\right\}\right)_{i < j \in [s]}, \left(\bm{1}\left\{\frac{\iprod{Z_i}{Z_j}}{d} > \frac{\Phi^{-1}(1-p)}{\sqrt{d}}\right\}\right)_{i < j \in [s]}\right) \\
    &\leq s^2\Prob\left(\frac{\iprod{Z_1}{Z_2}}{\norm{Z_1}\norm{Z_2}} \text{ is between } \tau(p, d), \frac{\Phi^{-1}(1-p)}{\sqrt{d}}\right) + \dTV\left(\left(\frac{\iprod{Z_i}{Z_j}}{\norm{Z_i}\norm{Z_j}}\right)_{i < j \in [s]}, \left(\frac{\iprod{Z_i}{Z_j}}{d}\right)_{i < j \in [s]}\right)\,,
\end{align*}
where we used coupling with union bound (over $i < j \in [s]$) and data processing inequality applied to the thresholding function. For the first term, as the probability is at most $O(1/d)$, it is upper bounded as $O(k^2/d) = o(1)$. To see this, we use the fact that up to scaling, normal distribution and spherical marginal distribution are very similar. In particular, from \cite[Lemma 7]{bubeck2016rgg} (originally in \cite[Lemma 1]{sodin2007tail}), the probability is at most
\begin{align*}
    &\left|\Prob\left(\frac{\iprod{Z_1}{Z_2}}{\norm{Z_1}\norm{Z_2}} \leq \tau(p, d)\right) - \Prob\left(\frac{\iprod{Z_1}{Z_2}}{\norm{Z_1}\norm{Z_2}} \leq \frac{\Phi^{-1}(1-p)}{\sqrt{d}}\right)\right| \\
    &\leq p(O(1/d) + 1 - \exp(-C_0\Phi^{-1}(1-p)^4/d))\,,
\end{align*}
for some constant $C_0 >0$, assuming that  $\Phi^{-1}(1-p) / \sqrt{d}$ is sufficiently small. The last term is at most $O(p\log^2(1/p)/d)$, from $\Phi^{-1}(1-p) \lesssim \sqrt{\log(1/p)}$ (from $\Prob(\calN(0, 1) > x) \leq \exp(-x^2/2)$ for $x > 0$), which also implies that $d = \Omega(\log(1/p))$ suffices for the condition on $\Phi^{-1}(1-p)/\sqrt{d}$.

The second term is the TV distance between Wishart ensemble and ``spherical Wishart'' ensemble (termed in \cite{paquette2021random}), where both distributions are non-product and have Gram matrix structures. While one  may approach this by comparing both matrices against GOE, that requires $d \gg k^3$ \cite{jiang2015wishart, bubeck2016rgg}. By directly comparing the strictly upper triangular parts of the two matrices, we show that they are asymptotically equivalent when $d \gg k^2$. 

\begin{proposition}[Comparison between Wishart and spherical Wishart]\label{prop:sphwisandwis} Let $Z_i \overset{\mathrm{i.i.d.}}{\sim} \calN(0, I_d)$ for $i \in [k]$. If $d \gg k^2$ and $k \to \infty$,
\[
\dTV \left( \left(\frac{\iprod{Z_i}{Z_j}}{\norm{Z_i}\norm{Z_j}} \right)_{i < j \in [k]} , \left(\frac{\iprod{Z_i}{Z_j}}{d} \right)_{i < j \in [k]} \right) = o(1)\,.
\]
\end{proposition}

As $s = \Theta(k)$, Proposition~\ref{prop:sphwisandwis} concludes that $\text{(I)} = o(1)$.
The proof of \prettyref{prop:sphwisandwis} is deferred to Appendix~\ref{appendix:sphwisandwis}. Notably, this result identifies another matrix ensemble that converges to the Wishart distribution in the regime where the dependence between the entries remains. 

\begin{remark}[Relevant results in the literature]\label{remark:wisandsphwis} Several recent works have shown indistinguishability results with respect to (variants of) the Wishart distribution, which hold for ranges below $d = k^3$. In particular, some variants of Wishart distribution are known to be indistinguishable to GOE for $d \gg k^2$, e.g., \cite[Corollary 4.3]{brennan2021finetti}, \cite[Theorem 22]{bangachev2025random}. More relatedly, \cite[Theorem 1]{chetelat2019middlescale} provides an explicit characterization of a distribution indistinguishable from Wishart when $k^2 \ll d \ll k^3$. That result is not directly comparable to ours, as we are considering a Wishart distribution without the diagonal entries. %
\end{remark}

\paragraph*{Second term (II).} For (II) in \eqref{eq:lbdensetwoterms}, it is convenient to write the total variation in terms of the likelihood ratio. Let $w_{\ell, d}, m_{\ell, d}$ respectively be the densities of $\Wishart(\ell, d)$ and $dI_\ell + \sqrt{d}\GOE(\ell)$. Then the likelihood ratio $L_s$ between the distributions of $\E_{S \big| |S| = s} [W^S]$ and $M$ satisfies
\begin{equation}\label{eq:lbdensewts2}
    \dTV\left(\E_{S \big| |S| = s} [W^S], M\right) = \frac{1}{2}\E_{X \sim m_{n, d}}[|L_s(X)-1|]\,, \quad \text{where }L_s(X) := \E_{S \big| |S| = s}\left[\frac{w_{s, d}(X_{S \times S})}{m_{s, d}(X_{S \times S})}\right]  \,. 
\end{equation}
\[
\]
Here, we use $X_{A \times B} \in \reals^{|A| \times |B|}$ to denote the submatrix of $X$ corresponding to rows $A$ and columns $B$. For a conditioning argument, we define a truncated likelihood 
\[\widetilde{L}_s(X) := \E_{S \big| |S| = s}\left[\frac{w_{s, d}(X_{S \times S})\bm{1}\{X_{S \times S} \in \calE_s\}}{m_{s, d}(X_{S \times S})}\right]\,,
\]
where for any $\ell \leq s$, 
\begin{equation}\label{eq:lbdensewhpevent}
\calE_\ell := \left\{X \in \reals^{\ell \times \ell} \text{ symmetric: } \norm{X_{Q \times Q} - dI_{|Q|}}_{\op} \leq 10(1+\sqrt{\log s})\sqrt{d|Q|} \text{ for all }\emptyset \subsetneq Q \subseteq [\ell] \right\}\,.
\end{equation}
As we will see later, this prevents matrix entries from being too large---an obstacle for second moment type calculation. For (II) to be $o(1)$, it suffices to show 
\begin{align}
    \min_{s \in [k^-, k^+]}\E_{X \sim m_{n, d}}[\widetilde{L}_s(X)] &\to 1\,, \label{eq:lbdensetrlikelihood1}\\
    \max_{s \in [k^-, k^+]}\E_{X \sim m_{n, d}}[\widetilde{L}_s(X)^2] &= 1 + o(1) \label{eq:lbdensetrlikelihood2}\,.
\end{align}
This is from
\begin{align*}
    \E_{X \sim m_{n, d}}[|L_s(X)-1|] &\leq \E_{X \sim m_{n, d}}[|\widetilde{L}_s(X)-1|] + \E_{X \sim m_{n, d}}[L_s(X) - \widetilde{L}_s(X)] \\
    &\leq \sqrt{\E_{X \sim m_{n, d}}[\widetilde{L}_s(X)^2] - 1 + 2(1-\E_{X \sim m_{n, d}}[\widetilde{L}_s(X)])} + 1- \E_{X \sim m_{n, d}}[\widetilde{L}_s(X)]\,,
\end{align*}
where taking maximum over $s \in [k^-, k^+]$ along with \eqref{eq:lbdensetrlikelihood1} and \eqref{eq:lbdensetrlikelihood2} implies $\text{(II)} = o(1)$.

The first statement \eqref{eq:lbdensetrlikelihood1} is simple, following from the concentration of the Wishart distribution. First, we have
\[0 \leq \E_{X \sim m_{n, d}}[1 - \widetilde{L}_s(X)] = \Prob_{X \sim w_{n, d}}(X_{[s] \times [s]} \notin \calE_s)\,,
\]
from $\widetilde{L} \leq L$ and symmetry with respect to $S$. The latter can be further upper bounded as
\begin{align*}
    \Prob_{X \sim w_{n, d}}(X_{[s] \times [s]} \notin \calE_s) &\leq \sum_{q=1}^{s} \binom{s}{q}\Prob_{X \sim w_{n, d}}\left(\norm{X_{[q] \times [q]} - dI_q}_{\op} > 10(1+\sqrt{\log s})\sqrt{dq}\right) \\
    &\leq \sum_{q=1}^{s} 2\exp(q \log s - (3(1+\sqrt{\log s})\sqrt{q})^2/2) \leq s \times 2s^{-3}\,.
\end{align*}
where the first inequality is from union bound and symmetry and the second inequality is from standard Wishart concentration \cite[Theorem II.13]{davidsonszarek} with $d + 10(1+\sqrt{\log s})\sqrt{dq} > (\sqrt{d} + \sqrt{q} + 3(1+\sqrt{\log s})\sqrt{q})^2$ for all sufficiently large $d$. Since $s = \Omega(k)$ for $s \in [k^-, k^+]$ and $k \to \infty$, this proves \eqref{eq:lbdensetrlikelihood1}.

For the second statement \eqref{eq:lbdensetrlikelihood2}, let $S'$ be an i.i.d. copy of $S$. Then,
\begin{equation}\label{eq:lbdenselikelihoodsquared}
\begin{aligned}
    &\E_{X \sim m_{n, d}}[\widetilde{L}_s(X)^2] \\
    &= \E_{X \sim m_{n, d}} \left[\E_{S \big| |S| = s}\left[\frac{w_{s, d}(X_{S \times S})\bm{1}\{X_{S \times S} \in \calE_s\}}{m_{s, d}(X_{S \times S})}\right] \E_{S' \big| |S'| = s}\left[\frac{w_{s, d}(X_{S' \times S'})\bm{1}\{X_{S' \times S'} \in \calE_s\}}{m_{s, d}(X_{S' \times S'})}\right]\right] \\
    &= \E_{S, S' \big| |S| = |S'| = s} \left[\E_{X \sim m_{n, d}} \left[\frac{w_{s, d}(X_{S \times S})\bm{1}\{X_{S \times S} \in \calE_s\}w_{s, d}(X_{S' \times S'})\bm{1}\{X_{S' \times S'} \in \calE_s\}}{m_{s, d}(X_{S \times S})m_{s, d}(X_{S' \times S'})}\right]\right] \\
    &\leq  \E_{S, S' \big| |S| = |S'| = s} \left[\E_{X \sim m_{n, d}} \left[\frac{w_{s, d}(X_{S \times S})w_{s, d}(X_{S' \times S'})\bm{1}\{X_{(S \cap S') \times (S \cap S')} \in \calE_{|S \cap S'|}\}}{m_{s, d}(X_{S \times S})m_{s, d}(X_{S' \times S'})}\right]\right]\,.
\end{aligned}
\end{equation}
Our next step is to show that the inner expectation in \eqref{eq:lbdenselikelihoodsquared} only depends on the entries with indices that overlap over $S$ and $S'$, i.e., $X_{(S \cap S') \times (S \cap S')}$. For notational convenience, let $R := S \cap S'$ and $N := |R|$. Define the (conditional) densities $g, g', h, h'$ of submatrices\footnote{Below, $X_{S\times S \setminus R \times R}$ denotes the entries $X_{(S \setminus R) \times (S \setminus R)} \cup X_{(S\setminus R) \times R} \cup X_{R \times (S \setminus R)}$.} as follows:
\begin{align*}
    X_{S \times S \setminus R \times R}|X_{R \times R} \sim g, \quad X_{S' \times S' \setminus R \times R} | X_{R \times R} \sim g' \quad &\text{where }X \sim w_{n, d}\,, \\
    X_{S \times S \setminus R \times R}|X_{R \times R} \sim h, \quad X_{S' \times S' \setminus R\times R}| X_{R \times R} \sim h'\quad &\text{where }X \sim m_{n, d}\,.
\end{align*}
Notably, as $m_{s, d}$ is a product distribution, $h$ and $h'$ do not depend on $X_{R \times R}$. Then
\begin{align*}
    &\E_{X \sim m_{n, d}} \left[\frac{w_{s, d}(X_{S \times S})w_{s, d}(X_{S' \times S'})\bm{1}\{X_{(S \cap S') \times (S \cap S')} \in \calE_{|S \cap S'|}\}}{m_{s, d}(X_{S \times S})m_{s, d}(X_{S' \times S'})}\right] \\
    &= \E_{X \sim m_{n, d}} \left[\frac{w^2_{N, d}(X_{R \times R})\bm{1}\{X_{R \times R} \in \calE_{N}\}g(X_{S \times S \setminus R \times R} | X_{R \times R})g'(X_{S' \times S' \setminus R \times R} | X_{R \times R})}{m^2_{N, d}(X_{R \times R})h(X_{S \times S \setminus R \times R}) h'(X_{S' \times S' \setminus R \times R})}\right]\\
    &= \E_{X_{R \times R} \sim m_{N, d}} \bigg[\frac{w^2_{N, d}(X_{R \times R})\bm{1}\{X_{R \times R} \in \calE_{N}\}}{m^2_{N, d}(X_{R \times R})} \\
    &\quad \times \E_{X_{S \times S \setminus R \times R} \sim h}\left[\frac{g(X_{S \times S \setminus R \times R} | X_{R \times R})}{h(X_{S \times S \setminus R \times R}) }\right] \times \E_{X_{S' \times S' \setminus R \times R} \sim h'} \left[\frac{g'(X_{S' \times S' \setminus R \times R} | X_{R \times R})}{h'(X_{S' \times S' \setminus R \times R})} \right]\bigg]\\
    &= \E_{X_{R \times R} \sim m_{N, d}} \left[\frac{w^2_{N, d}(X_{R \times R})\bm{1}\{X_{R \times R} \in \calE_{N}\}}{m^2_{N, d}(X_{R \times R})}\right]\,,
\end{align*}
where the first equality is by definition, second equality is from $m_{n, d}$ being a product distribution, and the last equality is from $g$ and $g'$ being densities. Since the last term only depends on $X_{R \times R}$, by simplifying the notation as $X \sim m_{N, d}$, \eqref{eq:lbdenselikelihoodsquared} is at most
\begin{equation}\label{eq:lbdensetrlikelihood3}
    \E_{N \big| |S| = |S'| = s} \left[\E_{X \sim m_{N, d}} \left[\frac{w^2_{N, d}(X)\bm{1}\{X \in \calE_{N}\}}{m^2_{N, d}(X)}\right]\right]\,.
\end{equation}
The next step is to evaluate the inner expectation in \eqref{eq:lbdensetrlikelihood3} with respect to $X$ as a function of $N$ and $d$. This is done by using Taylor expansion to approximate $w_{N, d}^2/m_{N, d}^2$ as a simple function of the spectrum of $m_{N, d}$, along with a control of the entries from the high probability event \eqref{eq:lbdensewhpevent}. As the calculations are at least conceptually similar to those in prior works \cite{bubeck2016rgg, raczrichey}, here we focus on their implication as presented in the following lemma.

\begin{lemma}[Comparison between Wishart and GOE]\label{lem:lbdensef1throughf3} Let $d \gg k^2$ and $k \to \infty$. Then there exists a constant $C_{\ref{lem:lbdensef1throughf3}} > 0$ and positive random variables $f_1(X), f_2(X), f_3(X)$ such that
\begin{equation}\label{eq:lbdensef1throughf3}
     \frac{w^2_{N, d}(X)\bm{1}\{X \in \calE_{N}\}}{m^2_{N, d}(X)} \leq (1+o(1))f_1(X)f_2(X)f_3(X)\,,
\end{equation}
where each $f_i(X)$ (which depends on $N, d$) satisfies
\[\E_{X \sim m_{N, d}} [f_i(X)^3] \leq \exp\left(C_{\ref{lem:lbdensef1throughf3}}\frac{N^3}{d}\right), \quad 1 \leq i \leq 3\,.
\]
Here, the $1+o(1)$ factor in \eqref{eq:lbdensef1throughf3} is universal over $s \in [k^-, k^+]$.
\end{lemma}

The proof is deferred to Appendix~\ref{appendix:lbdensef1throughf3}. The upper bounds on $f_i(X)$ suggest that it suffices to calculate the expectation of $\exp(O(N^3/d))$ over $N = |S \cap S'|$, conditioned on $|S| = |S'| = s = \Theta(k)$. The following lemma shows that this expectation is $1+o(1)$ under the desired condition $d \gg k^2 \vee k^6/n^3$. We mention that this result previously appeared in \cite[Proposition 4.3]{kunisky2025directed} for a similar goal of establishing detection lower bound, although in a different model.

\begin{lemma}[Upper bound on cubed hypergeometric]\label{lem:hypergeomcubemgf} 
    Let $C_{\ref{lem:hypergeomcubemgf}} > 0$ be an arbitrary constant. If $k \leq n/5$ and $d \gg k^2 \vee k^6/n^3$ then
    \[\max_{s \in [k^-, k^+]}\E_{N \big| |S| = |S'| = s} \left[\exp\left(C_{\ref{lem:hypergeomcubemgf}}\frac{N^3}{d}\right)\right] = 1 + o(1)\,.
    \]
\end{lemma}
The proof is deferred to Appendix~\ref{appendix:hypergeomcubemgf}.
From \eqref{eq:lbdensetrlikelihood3} and \eqref{eq:lbdensef1throughf3}, we obtain
\begin{align*}
    &\max_{s \in [k^-, k^+]}\E_{N \big| |S| = |S'| = s} \left[\E_{X \sim m_{N, d}} \left[\frac{w^2_{N, d}(X)\bm{1}\{X \in \calE_{N}\}}{m^2_{N, d}(X)}\right]\right] \\
    &\leq 
    (1+o(1))\max_{s \in [k^-, k^+]}\E_{N \big| |S| = |S'| = s}[\E_{X \sim m_{N, d}}[f_1(X)f_2(X)f_3(X)]] \\ 
    &\leq (1+o(1))\prod_{i=1}^3 \left(\max_{s \in [k^-, k^+]}\E_{N \big| |S| = |S'| = s}[\E_{X \sim m_{N, d}}[f_i(X)^3]]\right)^{1/3} \\
    &\leq (1+o(1))\prod_{i=1}^3 \left(\max_{s \in [k^-, k^+]}\E_{N \big| |S| = |S'| = s}\left[\exp\left(C_{\ref{lem:lbdensef1throughf3}}\frac{N^3}{d}\right)\right]\right)^{1/3} \\
    &\leq 1+o(1)\,,
\end{align*}
where the second inequality is by H\"older inequality, the third inequality is by Lemma~\ref{lem:lbdensef1throughf3}, and the last inequality is by Lemma~\ref{lem:hypergeomcubemgf}. This proves \eqref{eq:lbdensetrlikelihood2} and completes the proof of Theorem~\ref{thm:lbdense}.

\subsection{Tensorization of KL divergence}
In this approach, we consider a sequential process where at each time $i = 1, \dots, n$, a new vertex and its edges with respect to the previous vertices are introduced.

Formally, let $B$ be the adjacency matrix of a sample from $\calQ$, with the neighborhood vector $B_i \in \{0, 1\}^{i-1}$ for $i \in [n]$ %
denoting the adjacency between $i$ and vertices $j < i$. Similarly, we define $A$ (and similarly $A_i \in \{0, 1\}^{i-1}$ for $ i \in [n]$) to be the adjacency matrix of a sample from $\calP$, which depends on the joint latent vector $X_i := (U_i, V_i) \in \reals^{d + 1}, i \in [n]$. To be specific, $U_i \sim \calU(\mathbb{S}^{d-1})$ (the feature vector) and $V_i \sim \mathrm{Ber}(k/n)$ (the community membership indicator) for all $i \in [n]$ independently. Note that here we assume every vertex $i$ to have a feature vector $U_i$, and it is the membership indicator $V_i$ that determines whether the feature vector would be used.

    Throughout, we will use notation such as $A_{[i]}:= (A_1, \dots, A_i)$. Then we have
    \begin{align*}
        2\dTV(\calP, \calQ)^2 = 2\dTV(A, B)^2 &\leq \dKL(A || B) \\
        &\leq \sum_{i=0}^{n-1}\E_{A_{[i]}} [\dKL(A_{i+1} | A_{[i]} || B_{i+1})]\\
        &\leq \sum_{i=0}^{n-1}\E_{A_{[i]}, X_{[i]}} [\dKL(A_{i+1} | A_{[i]}, X_{[i]} || B_{i+1})] \\
        &= \sum_{i=0}^{n-1}\E_{X_{[i]}} [\dKL(A_{i+1} |X_{[i]} || B_{i+1})]\,,
    \end{align*}
    from Pinsker's inequality, chain rule for KL divergence (plus the independence between $B_{i+1}$ and $B_{[i]}$), and the convexity of KL divergence and the fact that $A_{[i]}$ and $A_{i+1}$ are conditionally independent given $X_{[i]}$. By the same argument as in \cite[Claim 8.2]{liu2022rgg}, the last term can be further upper bounded as
    \[\sum_{i=0}^{n-1}\E_{X_{[i]}} [\dKL(A_{i+1} |X_{[i]} || B_{i+1})] \leq n\E_{X_{[n-1]}}[\dKL(A_n | X_{[n-1]}|| B_n)]\,.
    \]
    Thus, it suffices to show that
    \begin{equation}\label{eq:lbklwts}
        \E_{X_{[n-1]}}[\dKL(A_n | X_{[n-1]}|| B_n)] = o\left(\frac{1}{n}\right)\,.
    \end{equation}
    Let $\calP_n(\cdot | X_{[n-1]}): \{0, 1\}^{n-1} \to [0, 1]$ be  the law of $A_n|X_{[n-1]}$ and $\calQ_n(\cdot): \{0, 1\}^{n-1} \to [0, 1]$ be the law of $B_n$. Then \eqref{eq:lbklwts} can be rewritten as
    \begin{equation}\label{eq:lbklterm}
    \begin{aligned}
        \E_{X _{[n-1]}} [\dKL(A_n | X_{[n-1]} || B_n)] &= \E_{X_{[n-1]}}\left[\E_{\Gamma \sim \calP_n (\cdot | X_{[n-1]})} \left[\log \frac{\calP_n(\Gamma | X_{[n-1]})}{\calQ_n(\Gamma)} \right]\right]\\
        &= \E_{X_{[n-1]}}\left[\E_{\Gamma \sim \calP_n(\cdot | X_{[n-1]})} \left[\log\left(\Delta(\Gamma, X_{[n-1]}) + 1\right)\right]\right]\,,
    \end{aligned}
    \end{equation}
    where $\Gamma := (\Gamma_1, \dots, \Gamma_{n-1}) \in \{0, 1\}^{n-1}$ and  $\Delta(\Gamma, X_{[n-1]}) := \calP_n(\Gamma | X_{[n-1]})/\calQ_n(\Gamma) - 1$. 

    Our next argument is that the size of $\Delta$ can be significantly reduced in two ways. First, this gives a smaller ratio between the neighborhood distributions. This comes from the fact that vertex $n$ itself may not be in the community ($V_n$ is not revealed yet), in which case the neighborhood distribution is simply $\calQ_n$.   
    Indeed, we can write as
    \[\calP_n(\Gamma|X_{[n-1]}) = \frac{k}{n}\overline{\calP}_n(\Gamma|X_{[n-1]}) + \left(1-\frac{k}{n}\right)\calQ_n(\Gamma)\,,
    \]
    where $\overline{\calP}_n(\cdot | X_{[n-1]})$ is the distribution of $A_n | X_{[n-1]}$ conditioned on $V_n = 1$. Thus, 
    \begin{equation}\label{eq:lbgeneraldeltadef1}
        \Delta(\Gamma, X_{[n-1]}) = \frac{k}{n}\left(\frac{\overline{\calP}_n(\Gamma|X_{[n-1]})}{\calQ_n(\Gamma)} - 1\right)\,. 
    \end{equation}
    
    Second, $\Delta$ is small in that it typically compares neighborhoods of size much smaller than $n$. If a previous vertex is not in the community, its edge between vertex $n$ is $\Bern(p)$, canceling out the corresponding factor in $\calQ_n$; note that such information on the previous vertices is available from the conditioning on $X_{[n-1]}$. 
    This yields a smaller ratio between the neighborhood distributions. 
    As each previous vertex is in the community with probability $k/n$, on average, at most $O(k)$ vertices are relevant.

    In order to capture the average-case behavior explained so far, we consider the following event:
    \[\{(\Gamma, U, V) \in \calE\} := \left\{ \sum_{i=1}^{n-1} V_i\leq 2k, \sum_{i \in [n-1]: V_i = 1}\Gamma_i \leq 4kp, |\Delta(\Gamma, X_{[n-1]})| \leq \frac{k}{n} \right\}\,.
    \]
    Here, $U := (U_1, \dots, U_{n-1})$. The following lemma shows that this is indeed a high probability event.

    \begin{lemma}\label{lem:lbgeneralhighprobevent} Assume $p \in (0, 1/2]$. Then there exists a constant $C_{\ref{lem:lbgeneralhighprobevent}} > 0$ such that if $n$ is sufficiently large and
    \[kp \geq C_{\ref{lem:lbgeneralhighprobevent}}\log n \quad \text{and} \quad  d \geq C_{\ref{lem:lbgeneralhighprobevent}}(kp\log(1/p)\log(d/p))^2\log n\,,
    \]
    then the following holds:
    \begin{equation}\label{eq:lbgeneralhighprobevent}
    \begin{aligned}
        n\log(1/p)\Prob_{X_{[n-1]}, \Gamma \sim \calP_n(\cdot | X_{[n-1]})}((\Gamma, U, V) \in \calE^c) &= o(1/n)\,, \\
        n\log(1/p)\Prob_{X_{[n-1]}, \Gamma \sim \calQ_n}((\Gamma, U, V) \in \calE^c) &= o(1/n)\,.
    \end{aligned}
    \end{equation}
    \end{lemma}
    The proof is deferred to Appendix \ref{appendix:lbgeneralhighprobevent}.
    Continuing from \eqref{eq:lbklterm}, we have 
    \begin{align*}
        &\E_{X_{[n-1]}}\left[\E_{\Gamma \sim \calP_n(\cdot | X_{[n-1]})} \left[\log\left(\Delta(\Gamma, X_{[n-1]}) + 1\right)\right]\right] \\
        &\leq o\left(\frac{1}{n}\right) + \E_{X_{[n-1]}}\left[\E_{\Gamma \sim \calP_n(\cdot | X_{[n-1]})} \left[\log\left(\Delta(\Gamma, X_{[n-1]}) + 1\right)\bm{1}\{(\Gamma, U, V) \in \calE\}\right]\right] \\
        &\leq o\left(\frac{1}{n}\right) + \E_{X_{[n-1]}}\left[\E_{\Gamma \sim \calQ_n}\left[\left(\Delta(\Gamma, X_{[n-1]}) + 1\right)\Delta(\Gamma, X_{[n-1]})\bm{1}\{(\Gamma, U, V) \in \calE\}\right]\right] \\
        &\leq o\left(\frac{1}{n}\right) + \E_{X_{[n-1]}}\left[\E_{\Gamma \sim \calQ_n}\left[\Delta(\Gamma, X_{[n-1]})^2\bm{1}\{(\Gamma, U, V) \in \calE\}\right]\right]\,,
    \end{align*}
    where the first inequality is from Lemma~\ref{lem:lbgeneralhighprobevent} and $\log (\Delta + 1) \leq n\log(1/p)$ from $\calP_n \leq 1, \calQ_n \geq p^n$; the second inequality is from $\log(x+1) \leq x$; and the third inequality is from
    \begin{align*}
        &\left|\E_{X_{[n-1]}}\left[\E_{\Gamma \sim \calQ_n}\left[\Delta(\Gamma, X_{[n-1]})\bm{1}\{(\Gamma, U, V) \in \calE\}\right]\right]\right| \\
        &= |\Prob_{X_{[n-1]}, \Gamma \sim \calQ_n}((\Gamma, U, V) \in \calE^c) - \Prob_{X_{[n-1]}, \Gamma \sim \calP_n (\cdot | X_{[n-1]})}((\Gamma, U, V) \in \calE^c)| = o(1/n)\,,
    \end{align*}
    again by Lemma~\ref{lem:lbgeneralhighprobevent}.

For notational convenience, define $\nu$ to be the number of vertices $i \in [n-1]$ such that $V_i = 1$, and $i_1 < \dots < i_{\nu}$ be such indices. Then 
   \begin{equation}\label{eq:lbgeneraldeltadef2}
       \Delta(\Gamma, X_{[n-1]}) =\frac{k}{n}\left(\frac{\widetilde{\calP}_{\nu+1}(\Gamma_{i_1}, \dots, \Gamma_{i_{\nu}}|U_{i_1}, \dots, U_{i_{\nu}})}{p^{\sum_{j=1}^\nu \Gamma_{i_j}}(1-p)^{\nu - \sum_{j=1}^\nu \Gamma_{i_j}}}  - 1\right)\,,
   \end{equation}
    where $\widetilde{\calP}_{l+1}( \cdot | U_{m_1}, \dots, U_{m_l})$ denotes the distribution of edges between vertex $n$ and vertices $m_1 < \dots < m_l$ under $\calG(l+1, p, d)$ (over vertices $m_1, \dots, m_l, n$ in this case), conditioned on the corresponding latents $U_{m_1}, \dots, U_{m_l}$. Thus, it suffices to show that

    \begin{equation}\label{eq:lbgeneralchisqdiv}
    \begin{aligned}
        &\E_{X_{[n-1]}}\left[\E_{\Gamma \sim \calQ_n}\left[\Delta(\Gamma, X_{[n-1]})^2\bm{1}\{(\Gamma, U, V) \in \calE\}\right]\right] \\
        &=\left(\frac{k}{n}\right)^2 \E_{X_{[n-1]}}\left[\E_{\Gamma \sim \calQ_n}\left[\left(\frac{\widetilde{\calP}_{\nu+1}(\Gamma_{i_1}, \dots, \Gamma_{i_{\nu}}|U_{i_1}, \dots, U_{i_{\nu}})}{p^{\sum_{j=1}^\nu \Gamma_{i_j}}(1-p)^{\nu - \sum_{j=1}^\nu \Gamma_{i_j}}} - 1\right)^2\bm{1}\{(\Gamma, U, V) \in \calE\}\right]\right]\,,
    \end{aligned}
    \end{equation}
    is $o(1/n)$. For this, consider any fixed $V$ and $\Gamma$ with $\nu \leq 2k$ and $\sum_{j=1}^\nu \Gamma_{i_j} \leq 4kp$ (the first two conditions of $\calE$). By only invoking the randomness of $U_{i_1}, \dots, U_{i_\nu}$,
    \begin{align*}
        &\E_{U_{i_1}, \dots, U_{i_\nu}} \left[\left(\frac{\widetilde{\calP}_{\nu+1}(\Gamma_{i_1}, \dots, \Gamma_{i_{\nu}}|U_{i_1}, \dots, U_{i_{\nu}})}{p^{\sum_{j=1}^\nu \Gamma_{i_j}}(1-p)^{\nu - \sum_{j=1}^\nu \Gamma_{i_j}}} - 1\right)^2\bm{1}\{(\Gamma, U, V) \in \calE\}\right]\\
        &\leq \int_0^1\Prob_{U_{i_1}, \dots, U_{i_\nu}}\left(\left(\frac{\widetilde{\calP}_{\nu+1}(\Gamma_{i_1}, \dots, \Gamma_{i_{\nu}}|U_{i_1}, \dots, U_{i_{\nu}})}{p^{\sum_{j=1}^\nu \Gamma_{i_j}}(1-p)^{\nu - \sum_{j=1}^\nu \Gamma_{i_j}}} - 1\right)^2 \geq t\right)dt \\
        &\leq \int_0^{1/(k^2(\log n)/n)}1dt \\
        &\quad + \int_{1/(k^2(\log n)/n)}^1 \exp\left(-\frac{1}{C_{\ref{lem:capsandanticaps}}}\frac{dt}{C_0kp(kp\log(1/p) + \log(d/p))\log(1/p)\log(d/p)} + C_{\ref{lem:capsandanticaps}}\log(3k)\right)dt \\
        &\leq \frac{1}{\log n}\frac{n}{k^2} + \frac{C_1(kp\log(1/p)\log(d/p))^2}{d}\exp\left(-\frac{dn}{C_1(k^2p\log(1/p)\log(d/p))^2\log n}+ C_1\log(3k)\right)\,,
    \end{align*}
    for some constants $C_0 > 0$ and $C_1 > 0$. The first inequality is from $|\Delta(\Gamma, X_{[n-1]})| \leq k/n$ (third condition of $\calE$; see~\eqref{eq:lbgeneraldeltadef2}) and the second inequality is from Lemma~\ref{lem:capsandanticaps}, using $\nu \leq 2k$ and $\sum_{j=1}^\nu \Gamma_{i_j} \leq 4kp$; the final inequality is from $kp\log(1/p) + \log(d/p) \leq kp\log(1/p)\log(d/p)$. Thus, \eqref{eq:lbgeneralchisqdiv} is at most
    \[\frac{1}{n\log n} + \frac{C_1(k^2p\log(1/p)\log(d/p))^2}{n^2d}\exp\left(-\frac{dn}{C_1(k^2p\log(1/p)\log(d/p))^2\log n}+ C_1\log(3k)\right)\,.
    \]
    This is $o(1/n)$, as long as
    \[d \geq C_2\frac{k^4p^2}{n}\log^2(1/p)\log^2(d/p)\log^2n\,,
    \]
    for some constant $C_2 > 0$. In combination to Lemma~\ref{lem:lbgeneralhighprobevent}, we showed that for \eqref{eq:lbklwts} to hold it suffices to have
    \[d \geq (C_{\ref{lem:lbgeneralhighprobevent}} \vee C_2) \left(k^2p^2 \vee \frac{k^4p^2}{n}\right) \log^2(d/p) \log^2(1/p) \log^2 n \,.
    \]
    This is implied by $d \geq C_{\ref{thm:lbgeneral}} (k^2p^2 \vee k^4p^2/n) \log^2(k/p) \log^2(1/p) \log^3n $, for an appropriate choice of the constant $C_{\ref{thm:lbgeneral}} > 0$.

\section{Proofs for computational lower bound}\label{sec:comp}

\subsection{Low-degree lower bound}
In this subsection, we prove Theorem~\ref{thm:lowdeglb}.

For any subgraph $H$ of $K_n$, define its Fourier coefficient (with respect to the orthonormal basis in $\calL^2(\calQ)$) as
\[\Phi_\calP(H) := \E_{G \sim \calP}\left[\prod_{ij \in E(H)} \frac{G_{ij}-p}{\sqrt{p(1-p)}}\right]\,.
\]
From \cite[Claim 3.1]{bangachev2024fourier}, to show that there is no polynomial of degree at most $D$ achieving weak separation, it suffices to have
\begin{equation}\label{eq:lowdeglbwts}
    \sum_{H\subseteq K_n: 1 \leq e(H) \leq D}\Phi_\calP(H)^2 = o(1)\,.
\end{equation}
Now let $D = \lfloor(\log n / \log (\log n))^2\rfloor$.
Note that for each $H$, there are at most $n^{v(H)}$ subgraphs in $K_n$ that are isomorphic to $H$. Furthermore, such subgraphs all share the same value of $\Phi_\calP(H)$. Thus for \eqref{eq:lowdeglbwts}, it suffices to show
\begin{equation}\label{eq:lowdeglbwts2}
    \sum_{\text{non-iso. }H\subseteq K_n : 1 \leq e(H) \leq D}n^{v(H)}\Phi_\calP(H)^2 = o(1)\,,
\end{equation}
where non-iso. $H \subseteq K_n$ denotes the enumeration of all nonisomorphic subgraphs $H$ of $K_n$.
For each $H$, let $H_1, \dots, H_r$ be its connected components. Then $\Phi_{\calP}(H) = \prod_{i = 1}^r \Phi_{\calP}(H_i)$ and in particular if any of $H_i$ is a tree, $\Phi_{\calP}(H) = 0$ by Corollary~\ref{cor:treesimple}. Thus for $H$ with $\Phi_{\calP}(H) \neq 0$, one can assume that $v(H_i) \geq 3$ and $e(H_i) \geq v(H_i)$. Then for each $H_i$, we have
\begin{equation}\label{eq:lowdegfourierbound}
\begin{aligned}
    |\Phi_\calP(H_i)| &= \left|(p(1-p))^{-e(H_i)/2}\E_{G \sim \calP}\left[\prod_{j\ell \in E(H_i)}(G_{j\ell}-p)\right]\right| \\
    &\leq(p(1-p))^{-e(H_i)/2}(8p)^{e(H_i)}\left(\frac{C_{\ref{lem:rggfourier}}(\log^6 n)(\log^{3/2}d)}{\sqrt{d}}\right)^{\lceil (v(H_i) - 1)/2 \rceil} \left(\frac{k}{n}\right)^{v(H_i)}\\
    &\leq (8\sqrt{2p})^{e(H_i)}\left(\frac{C_{\ref{lem:rggfourier}}(\log^6 n)(\log^{3/2} d)}{\sqrt{d}}\right)^{v(H_i)/3}  \left(\frac{k}{n}\right)^{v(H_i)} \\
    &\leq 8^{e(H_i)}\left(\frac{4C_{\ref{lem:rggfourier}}k^{3}p^{3/2} (\log^6 n)(\log^{3/2}d)}{n^3\sqrt{d}}\right)^{v(H_i)/3}\,.
\end{aligned}
\end{equation}
Here, the first inequality is from Lemma~\ref{lem:rggfourier} with $v(H_i), e(H_i) \leq \log^3 n$, and Lemma~\ref{lem:fouriersimple}; the second inequality is from $1-p \geq 1/2$ and $\lceil(v(H_i)-1)/2 \rceil \geq v(H_i)/3$; the last inequality is from $\sqrt{2p} \leq 1$ and $e(H_i) \geq v(H_i)$.

From \eqref{eq:lowdegfourierbound}, we can bound the left hand side of \eqref{eq:lowdeglbwts2} as
\begin{equation}\label{eq:lowdegfinal}
\begin{aligned}
    &\sum_{\text{non-iso. }H\subseteq K_n : 1 \leq e(H) \leq D}n^{v(H)}\Phi_\calP(H)^2 \\
    &\leq \sum_{\text{non-iso. }H\subseteq K_n : 1 \leq e(H) \leq D} 64^{e(H)}\left(\frac{16C_{\ref{lem:rggfourier}}^2k^6p^3(\log^{12} n) (\log^3 d)}{n^3d}\right)^{v(H)/3} \\
    &\leq \sum_{\text{non-iso. }H\subseteq K_n : 1 \leq e(H) \leq D} 64^{e(H)}n^{-((\eps \wedge \delta)/6)v(H)} \leq \sum_{\text{non-iso. }H\subseteq K_n : 1 \leq e(H) \leq D} n^{-((\eps \wedge \delta)/12)v(H)}\,.
\end{aligned}
\end{equation}
Here, for each constant $\eps > 0$, the second inequality holds as long as $n$ is sufficiently large. To see this, fix a constant $C_0 > 0$ such that $d \mapsto d/\log^3 d$ is increasing for $d \geq C_0$. If $k^6p^3/n^3 < C_0$ then $\frac{16C_{\ref{lem:rggfourier}}^2k^6p^3(\log^{12} n) (\log^3 d)}{n^3d} \leq \frac{16C_{\ref{lem:rggfourier}}^2 C_0(\log^{12} n) (\log^3 d)}{d} \leq n^{-\delta/2}$ for all sufficiently large $n$; otherwise,  $\frac{16C_{\ref{lem:rggfourier}}^2k^6p^3(\log^{12} n) (\log^3 d)}{n^3d} \leq \frac{16C_{\ref{lem:rggfourier}}^2(\log^{12} n) (\log^3 (k^6p^3/n^{3-\eps}))}{n^\eps} \leq n^{-\eps/2}$ for all sufficiently large $n$. The last inequality in~\eqref{eq:lowdegfinal} follows from $e(H) \leq v(H)^2 \wedge D \leq v(H)\sqrt{D} \leq v(H)\log n / \log (\log n)$, which implies $64^{e(H)} \leq n^{((\eps \wedge \delta)/12)v(H)}$ for all sufficiently large $n$.

Our arguments for controlling the last term in \eqref{eq:lowdegfinal} are similar to those in \cite[Proposition 3.1]{bangachev2024fourier}. As $v(H) \leq 2e(H) \leq 2D$, we bound the number of non-isomorphic graphs as a function of $v(H)$. In particular, the number of non-isomorphic graphs with $v(H) = v$ and $e(H) \leq D$ is at most 
\begin{equation}\label{eq:noniso}
    2^{\binom{v}{2}} \wedge \left(\binom{v}{2}+1\right)^D \leq 2^{v^2} \wedge \exp(2D\log v)\,.
\end{equation}
Here, $2^{\binom{v}{2}}$ follows from enumerating all of the possible choices $\{0, 1\}^{\binom{v}{2}}$ for the edges, and $\left(\binom{v}{2} + 1\right)^D$ follows from deciding to add edge among the $\binom{v}{2}$ possible choices (or to not add) at each time.

Now fix any function $g: \mathbb{N} \to \mathbb{N}$ such that $g(n) = \omega(1)$ and $g(n) = o(\log (\log n))$. If $v \leq \sqrt{D}g(n)$ then \eqref{eq:noniso} is at most $2^{v^2} \leq 2^{v(\sqrt{D} g(n)+1)} \leq 2^{v \log n / \log (\log n) \times  g(n) + v} = 2^{o(v \log n)}$. Similarly, if $v > \sqrt{D}g(n)$ then \eqref{eq:noniso} is at most $\exp(v \times 2D(\log v) / v) \leq \exp(v \times 2D\log (\sqrt{D} g(n)) / (\sqrt{D}g(n))) \leq \exp(v \times 2\log n / (\log(\log n)) \times \log( \log n)/g(n)) = \exp(o(v \log n))$, for all sufficiently large $n$. From these, the last term in~\eqref{eq:lowdegfinal} is at most
\[\sum_{v=1}^{2D} \exp(o(v\log n)) n^{-((\eps \wedge \delta)/12)v} \leq \sum_{v=1}^{2D} n^{((\eps \wedge \delta)/24 - (\eps \wedge \delta)/12)v} \leq 2Dn^{-(\eps \wedge \delta)/24} = o(1)\,,
\]
as $\exp(o(v\log n)) \leq n^{((\eps \wedge \delta)/24)v}$ for all $1 \leq v \leq 2D$, for all sufficiently large $n$.

\subsection{Suboptimality of longer cycle counts}

Here, we prove Proposition~\ref{prop:signedcycle}.
A key result is the following proposition, which provides a tight characterization of the expectation of signed cycle under $\calG(n, p, d)$. This can be directly translated into bounds under $\calP = \calG(n, p, d, k)$ via Lemma~\ref{lem:fouriersimple}. Let $\mathrm{Cyc}_\ell$ be a length-$\ell$ cycle.

\begin{proposition}[Expectation of signed cycle count]\label{prop:signedcyclecount}
Suppose that $p \leq 1/2$ and $d$ is sufficiently large with
    \[d \geq  (5\log(1/p))^4\,.
    \]
    Then there exists a constant $C_{\ref{prop:signedcyclecount}} > 0$ such that for any $3 \leq \ell \leq n$,
    \[\frac{1}{C_{\ref{prop:signedcyclecount}}^\ell} \frac{p^\ell\log^{\ell/2} (1/p)}{d^{\ell/2-1}} \leq \E_{G \sim \calG(n, p, d)}\left[\prod_{ij \in E(\mathrm{Cyc}_\ell)}(G_{ij}-p)\right] \leq C_{\ref{prop:signedcyclecount}}^\ell \frac{p^\ell\log^{\ell/2} (1/p)}{d^{\ell/2-1}}\,.
    \]
\end{proposition}
The proof is deferred to Appendix~\ref{app:signedcyclecount}.

\begin{remark}[Tightness and implications] While there are previous works giving upper bounds on this value (\eg, \cite{liu2023expansion, bangachev2024fourier}), our upper bound identifies the correct asymptotic dependence in all parameters (up to a constant factor in the base, with $\ell$ as an exponent). Furthermore, we provide a matching lower bound which may be of independent interest. For example, we expect that it can be used to show the tightness of the bound on the second largest eigenvalue in \cite{cao2025spectra}; see also Remark~\ref{remark:spectraltest}. 
\end{remark}

As a direct corollary of Proposition~\ref{prop:signedcyclecount}, we show that longer cycle counts are strictly less powerful than the triangle count.

\begin{proof}[Proof of Proposition~\ref{prop:signedcycle}] Let $f_\ell(G)$ be the signed count of $\mathrm{Cyc}_\ell$ of $G$. By Proposition~\ref{prop:signedcyclecount} and Lemma~\ref{lem:fouriersimple}.
\[0 \leq \E_{G \sim \calP}[f_\ell(G)] - \E_{G \sim \calQ}[f_\ell(G)] \leq \binom{n}{\ell}\frac{(\ell-1)!}{2} \times \left(\frac{k}{n}\right)^\ell\frac{C_{\ref{prop:signedcyclecount}}^\ell p^\ell\log^{\ell/2}(1/p)}{d^{\ell/2-1}}\,,
\]
as the number of distinct length-$\ell$ cycles in $K_n$ is $\binom{n}{\ell}\frac{(\ell-1)!}{2}$. Furthermore, a straightforward calculation yields
\[\Var_{G \sim \calQ}[f_\ell(G)] = \binom{n}{\ell}\frac{(\ell-1)!}{2}(p(1-p))^\ell \,.
\]
Then arranging $\E_{G \sim \calP}[f_\ell(G)] - \E_{G \sim \calQ}[f_\ell(G)] \gg \sqrt{\Var_{G \sim \calQ}[f_\ell(G)]}$ yields the result.
\end{proof}

We briefly discuss spectral algorithms below, as signed cycle counts are closely related to those via trace method \cite{wignertracemethod, vutracemethod}. 

\begin{remark}[Spectral tests]\label{remark:spectraltest} Several recent papers \cite{liu2023expansion, bangachev2024fourier, cao2025spectra} have studied the spectral property of the full model $\calG(n, p, d)$. For notation, let $\overline{A}$ be the centered adjacency matrix of $G \sim \calG(n,p, d)$, i.e., $\overline{A}_{ij} := G_{ij} - p$ if $i \neq j$ and $\overline{A}_{ii} := 0$. \cite{cao2025spectra} showed that if $d = \widetilde{\Omega}(np)$,
\begin{enumerate}
    \item[(i)] The empirical distribution of the eigenvalues of $\overline{A}$ converges to the same semicircle law as in $\calG(n, p)$.
    \item[(ii)] The largest eigenvalue of $\overline{A}$ is of the same order (up to polylogarithmic factors) as that of $\calG(n, p)$.
\end{enumerate}
Informally, these results suggest that ``na\"ive'' spectral tests (e.g., thresholding the largest eigenvalue of $\overline{A}$) may be suboptimal compared to the global signed triangle count, which succeeds all the way up to $d = o(n^3p^3\log^3(1/p))$. The result of \cite{cao2025spectra} is based on the trace method, where in part Proposition~\ref{prop:signedcyclecount} can be applied. Based on that, we expect that (i) and (ii) would hold for $G \sim \calG(n,p,d,k)$ if $d = \widetilde{\Omega}(k^2p/n)$, and hence the corresponding suboptimality of na\"ive spectral tests would persist.

However, this has a simple fix, which is to consider $(\overline{A})^3$ instead of $\overline{A}$. Then the trace of $(\overline{A})^3$ is equal to\footnote{Formally, this is up to a factor of $6$ due to the duplicates among the summands in $\Tr((\overline{A})^3)$.} the global signed triangle count, and we expect that a test that thresholds the largest eigenvalue of $(\overline{A})^3$ would have the same performance. We choose not to pursue this direction, as it does not seem to have any particular advantage over (and is conceptually identical to) the global test (Theorem~\ref{thm:ubsignedtrianglecount}).
\end{remark}

\section{Discussion}

We studied %
the detection problem for a new random graph model, in which a small community with latent high-dimensional geometry is hidden inside a larger \ER graph.
The model is designed intentionally to be minimal: it introduces nontrivial geometric structures into a community without altering marginal vertex statistics. In that sense, the model provides a clean setting for understanding how a structural signal (rather than  increased density alone, for example) affects the limits of detection. We characterized detection thresholds via tests based on signed triangles, and examined hardness both information-theoretically and computationally, revealing a computational--statistical gap. En route, we characterized a regime where the Wishart and the spherical Wishart distributions are asymptotically equivalent while being distinct from GOE, as well as tight bounds on signed cycle counts which may be of independent interest.

Focusing on the detection problem, for any fixed $p$ our upper and lower bounds match up to a logarithmic factor. For vanishing $p$, however, a gap remains. Improving the dependence on $p$ is a major open question even for the full model where $k = n$~\cite{liu2022rgg}, and seems to require substantially new ideas. Other interesting open directions are to sharpen the dependence on $k/n$, or to develop fine-grained results in specific regimes (e.g., sparse $p = O(1/k)$) where different techniques may apply~\cite{liu2022rgg}.

In our model, the community structure introduces a new layer for statistical tasks beyond detection. An immediate question concerns the threshold for exact or approximate recovery of the community vertices. Certain  models with latent-based communities exhibit a detection-recovery gap~\cite{mao2023planteddensecycle, kunisky2025directed, mao2025planteddensecycle}, and it is plausible that a similar phenomenon occurs here. More broadly, revisiting the statistical and computational properties of the full model $\calG(n,p,d)$ (as recently established in, e.g., \cite{bangachev2024fourier, bangachev2025sandwiching, cao2025spectra}) in our localized setting $\calG(n,p,d,k)$ could shed light on the interplay between high-dimensional geometry and community structure.

\section*{Acknowledgements}
We thank Kiril Bangachev for helpful discussions on the reference \cite{liuracz} and the ideas therein, and Jiaming Xu for helpful discussions on the detection problem at an early stage of the project.

\addcontentsline{toc}{section}{References}
\bibliographystyle{alpha}
{\small
\bibliography{ref}

\newcommand{\etalchar}[1]{$^{#1}$}
\begin{thebibliography}{MNWS{\etalchar{+}}25}

\bibitem[ABAL{\etalchar{+}}26]{addario2026statistical}
Louigi Addario-Berry, Omer Angel, G{\'a}bor Lugosi, Mikl{\'o}s~Z. R{\'a}cz, and Tselil Schramm.
\newblock The statistical threshold for planted matchings and spanning trees.
\newblock {\em arXiv preprint arXiv:2602.07669}, 2026.

\bibitem[Abb17]{Abbe2018SBMSurvey}
Emmanuel Abbe.
\newblock Community detection and stochastic block models: recent developments.
\newblock {\em J. Mach. Learn. Res.}, 18(177):1--86, 2017.

\bibitem[ABD21]{avrachenkov2021geometric}
Konstantin Avrachenkov, Andrei Bobu, and Maximilien Dreveton.
\newblock Higher-order spectral clustering for geometric graphs.
\newblock {\em J. Fourier Anal. Appl.}, 27(22):1--29, 2021.

\bibitem[ABI24]{abdalla2024synch}
Pedro Abdalla, Afonso~S. Bandeira, and Clara Invernizzi.
\newblock Guarantees for spontaneous synchronization on random geometric graphs.
\newblock {\em SIAM J. Appl. Dyn. Syst.}, 23(1):779--790, 2024.

\bibitem[ABS21]{abbe2021community}
Emmanuel Abbe, Fran\c{c}ois Baccelli, and Abishek Sankararaman.
\newblock Community detection on {E}uclidean random graphs.
\newblock {\em Inf. Inference}, 10(1):109--160, 2021.

\bibitem[ACBL12]{arias-castro2012correlation}
Ery Arias-Castro, S\'ebastien Bubeck, and G\'abor Lugosi.
\newblock Detection of correlations.
\newblock {\em Ann. Statist.}, 40(1):412--435, 2012.

\bibitem[ACBL15]{arias-castro2015correlation}
Ery Arias-Castro, S\'ebastien Bubeck, and G\'abor Lugosi.
\newblock Detecting positive correlations in a multivariate sample.
\newblock {\em Bernoulli}, 21(1):209--241, 2015.

\bibitem[ACBLV18]{arias-castro2018markov}
Ery Arias-Castro, S\'ebastien Bubeck, G\'abor Lugosi, and Nicolas Verzelen.
\newblock Detecting {M}arkov random fields hidden in white noise.
\newblock {\em Bernoulli}, 24(4B):3628--3656, 2018.

\bibitem[ACV14]{pdsdense}
Ery Arias-Castro and Nicolas Verzelen.
\newblock Community detection in dense random networks.
\newblock {\em Ann. Statist.}, 42(3):940--969, 2014.

\bibitem[AKL24]{avrachenkov2024community}
Konstantin Avrachenkov, B.~R.~Vinay Kumar, and Lasse Leskel{\"a}.
\newblock Community detection on block models with geometric kernels.
\newblock {\em arXiv preprint arXiv:2403.02802}, 2024.

\bibitem[AS65]{abramowitz1965handbook}
Milton Abramowitz and Irene~A. Stegun.
\newblock {\em Handbook of mathematical functions: with formulas, graphs, and mathematical tables}, volume~55.
\newblock Courier Corporation, 1965.

\bibitem[ATK15]{AkogluTongKoutra2015AnomalySurvey}
Leman Akoglu, Hanghang Tong, and Danai Koutra.
\newblock Graph based anomaly detection and description: a survey.
\newblock {\em Data Min. Knowl. Discov.}, 29(3):626--688, 2015.

\bibitem[AW15]{adamczak2015concentration}
Rados{\l}aw Adamczak and Pawe{\l} Wolff.
\newblock Concentration inequalities for non-{L}ipschitz functions with bounded derivatives of higher order.
\newblock {\em Probab. Theory Related Fields}, 162(3-4):531--586, 2015.

\bibitem[Bar16]{Barabasi2016NetworkScience}
Albert-L{\'a}szl{\'o} Barab{\'a}si.
\newblock {\em Network Science}.
\newblock Cambridge University Press, 2016.

\bibitem[BB20]{brennan2020plantedcliquereduction}
Matthew Brennan and Guy Bresler.
\newblock Reducibility and statistical-computational gaps from secret leakage.
\newblock In {\em Proceedings of Thirty Third Conference on Learning Theory}, volume 125 of {\em PMLR}, pages 648--847, 2020.

\bibitem[BB24a]{bangachev2024geometry}
Kiril Bangachev and Guy Bresler.
\newblock Detection of {$L_{\infty}$} geometry in random geometric graphs: suboptimality of triangles and cluster expansion.
\newblock In {\em Proceedings of Thirty Seventh Conference on Learning Theory}, volume 247 of {\em PMLR}, pages 427--497, 2024.

\bibitem[BB24b]{bangachev2024fourier}
Kiril Bangachev and Guy Bresler.
\newblock On the {F}ourier coefficients of high-dimensional random geometric graphs.
\newblock In {\em {P}roceedings of the 56th {A}nnual {ACM} {S}ymposium on {T}heory of {C}omputing}, pages 549--560, 2024.

\bibitem[BB25a]{bangachev2025random}
Kiril Bangachev and Guy Bresler.
\newblock Random algebraic graphs and their convergence to {E}rd{\H o}s-{R}{\'e}nyi.
\newblock {\em Random Structures Algorithms}, 66(1):e21276, 1--43, 2025.

\bibitem[BB25b]{bangachev2025sandwiching}
Kiril Bangachev and Guy Bresler.
\newblock Sandwiching random geometric graphs and {E}rd{\H o}s-{R}{\'e}nyi with applications: sharp thresholds, robust testing, and enumeration.
\newblock In {\em Proceedings of the 57th Annual ACM Symposium on Theory of Computing}, pages 310--321, 2025.

\bibitem[BBAP05]{baik2005spiked}
Jinho Baik, G\'erard Ben~Arous, and Sandrine P\'ech\'e.
\newblock Phase transition of the largest eigenvalue for nonnull complex sample covariance matrices.
\newblock {\em Ann. Probab.}, 33(5):1643--1697, 2005.

\bibitem[BBCvdH20]{bet2020botnet}
Gianmarco Bet, Kay Bogerd, Rui~M. Castro, and Remco van~der Hofstad.
\newblock Detecting a botnet in a network.
\newblock {\em Math. Stat. Learn.}, 3(3-4):315--343, 2020.

\bibitem[BBH21]{brennan2021finetti}
Matthew Brennan, Guy Bresler, and Brice Huang.
\newblock De finetti-style results for {W}ishart matrices: combinatorial structure and phase transitions.
\newblock {\em arXiv preprint arXiv:2103.14011}, 2021.

\bibitem[BBH24]{brennan2024anisotropy}
Matthew Brennan, Guy Bresler, and Brice Huang.
\newblock Threshold for detecting high dimensional geometry in anisotropic random geometric graphs.
\newblock {\em Random Structures Algorithms}, 64(1):125--137, 2024.

\bibitem[BBN20]{brennan2020rgg}
Matthew Brennan, Guy Bresler, and Dheeraj Nagaraj.
\newblock Phase transitions for detecting latent geometry in random graphs.
\newblock {\em Probab. Theory Related Fields}, 178(3-4):1215--1289, 2020.

\bibitem[BCC{\etalchar{+}}10]{bhaskara2010logdensity}
Aditya Bhaskara, Moses Charikar, Eden Chlamtac, Uriel Feige, and Aravindan Vijayaraghavan.
\newblock Detecting high log-densities: an {$O(n^{1/4})$} approximation for densest $k$-subgraph.
\newblock In {\em Proceedings of the 42nd Annual ACM Symposium on Theory of Computing}, pages 201--210, 2010.

\bibitem[BDER16]{bubeck2016rgg}
S\'ebastien Bubeck, Jian Ding, Ronen Eldan, and Mikl\'os~Z. R\'acz.
\newblock Testing for high-dimensional geometry in random graphs.
\newblock {\em Random Structures Algorithms}, 49(3):503--532, 2016.

\bibitem[BDT{\etalchar{+}}20]{bagaria2020hamiltoniancycle}
Vivek Bagaria, Jian Ding, David Tse, Yihong Wu, and Jiaming Xu.
\newblock Hidden {H}amiltonian cycle recovery via linear programming.
\newblock {\em Oper. Res.}, 68(1):53--70, 2020.

\bibitem[BG18]{bubeck2018entropicclt}
S\'ebastien Bubeck and Shirshendu Ganguly.
\newblock Entropic {CLT} and phase transition in high-dimensional {W}ishart matrices.
\newblock {\em Int. Math. Res. Not. IMRN}, 2018(2):588--606, 2018.

\bibitem[BGK{\etalchar{+}}01]{brieden2001cap}
Andreas Brieden, Peter Gritzmann, Ravindran Kannan, Victor Klee, L\'aszl\'o{} Lov\'asz, and Mikl\'os Simonovits.
\newblock Deterministic and randomized polynomial-time approximation of radii.
\newblock {\em Mathematika}, 48(1-2):63--105, 2001.

\bibitem[BGPS25]{baguley2025toroidal}
Samuel Baguley, Andreas G{\"o}bel, Marcus Pappik, and Leon Schiller.
\newblock Testing thresholds and spectral properties of high-dimensional random toroidal graphs via {E}dgeworth-style expansions.
\newblock In {\em Proceedings of Thirty Eighth Conference on Learning Theory}, volume 291 of {\em PMLR}, pages 200--201, 2025.

\bibitem[BHJK25]{buhai2025quasi}
Rares-Darius Buhai, Jun-Ting Hsieh, Aayush Jain, and Pravesh~K. Kothari.
\newblock The quasi-polynomial low-degree conjecture is false.
\newblock In {\em IEEE 66th Annual Symposium on Foundations of Computer Science}, pages 2577--2590, 2025.

\bibitem[BKL19]{bringmann2019geometricinhomogeneous}
Karl Bringmann, Ralph Keusch, and Johannes Lengler.
\newblock Geometric inhomogeneous random graphs.
\newblock {\em Theoret. Comput. Sci.}, 760:35--54, 2019.

\bibitem[BMS25]{bet2025localized}
Gianmarco Bet, Riccardo Michielan, and Clara Stegehuis.
\newblock Localized geometry detection in scale-free random graphs.
\newblock {\em J. Appl. Probab.}, page 1–17, 2025.

\bibitem[BO04]{BarabasiOltvai2004NetworkBiology}
Albert-L{\'a}szl{\'o} Barab{\'a}si and Zolt{\'a}n~N. Oltvai.
\newblock Network biology: understanding the cell's functional organization.
\newblock {\em Nat. Rev. Genet.}, 5(2):101--113, 2004.

\bibitem[BR13]{berthet2013plantedcliquereduction}
Quentin Berthet and Philippe Rigollet.
\newblock Complexity theoretic lower bounds for sparse principal component detection.
\newblock In {\em Proceedings of the 26th Annual Conference on Learning Theory}, volume~30 of {\em PMLR}, pages 1046--1066, 2013.

\bibitem[CW19]{chetelat2019middlescale}
Didier Ch\'etelat and Martin~T. Wells.
\newblock The middle-scale asymptotics of {W}ishart matrices.
\newblock {\em Ann. Statist.}, 47(5):2639--2670, 2019.

\bibitem[CZ25]{cao2025spectra}
Yifan Cao and Yizhe Zhu.
\newblock Spectra of high-dimensional sparse random geometric graphs.
\newblock {\em arXiv preprint arXiv:2507.06556}, 2025.

\bibitem[DDC23]{rggsurvey}
Quentin Duchemin and Yohann De~Castro.
\newblock Random geometric graph: some recent developments and perspectives.
\newblock In {\em High dimensional probability {IX}---the ethereal volume}, volume~80 of {\em Progr. Probab.}, pages 347--392. Birkh\"auser, 2023.

\bibitem[DGLU11]{devroye2011rgg}
Luc Devroye, Andr\'as Gy\"orgy, G\'abor Lugosi, and Frederic Udina.
\newblock High-dimensional random geometric graphs and their clique number.
\newblock {\em Electron. J. Probab.}, 16:2481--2508, 2011.

\bibitem[DKMZ11]{DecelleEtAl2011AsymptoticSBM}
Aur{\'e}lien Decelle, Florent Krzakala, Cristopher Moore, and Lenka Zdeborov{\'a}.
\newblock Asymptotic analysis of the stochastic block model for modular networks and its algorithmic applications.
\newblock {\em Phys. Rev. E}, 84:066106, 2011.

\bibitem[DLW25]{deka2025rare}
Prabhanka Deka, Fangzhou Luo, and Baichuan Wu.
\newblock Rare event probabilities in random geometric graphs.
\newblock {\em arXiv preprint arXiv:2510.09196}, 2025.

\bibitem[DS01]{davidsonszarek}
Kenneth~R. Davidson and Stanislaw~J. Szarek.
\newblock Local operator theory, random matrices and {B}anach spaces.
\newblock In {\em Handbook of the geometry of {B}anach spaces, {V}ol. {I}}, pages 317--366. North-Holland, 2001.

\bibitem[DWXY23]{ding2023plantedmatching}
Jian Ding, Yihong Wu, Jiaming Xu, and Dana Yang.
\newblock The planted matching problem: sharp threshold and infinite-order phase transition.
\newblock {\em Probab. Theory Related Fields}, 187(1-2):1--71, 2023.

\bibitem[DX13]{dai2013approximation}
Feng Dai and Yuan Xu.
\newblock {\em Approximation Theory and Harmonic Analysis on Spheres and Balls}.
\newblock Springer Monographs in Mathematics. Springer, 2013.

\bibitem[EH25]{elimelech2025arbitrary}
Dor Elimelech and Wasim Huleihel.
\newblock Detecting arbitrary planted subgraphs in random graphs.
\newblock In {\em Proceedings of Thirty Eighth Conference on Learning Theory}, volume 291 of {\em PMLR}, pages 1691--1798, 2025.

\bibitem[EM20]{eldan2020anisotropy}
Ronen Eldan and Dan Mikulincer.
\newblock Information and dimensionality of anisotropic random geometric graphs.
\newblock In {\em Geometric aspects of functional analysis. {V}ol. {I}}, volume 2256 of {\em Lecture Notes in Math.}, pages 273--324. Springer, 2020.

\bibitem[ER59]{erdosrenyi}
Paul Erd\H{o}s and Alfr\'ed R\'enyi.
\newblock On random graphs {I}.
\newblock {\em Publ. Math. Debrecen}, 6:290--297, 1959.

\bibitem[For10]{Fortunato2010CommunityDetection}
Santo Fortunato.
\newblock Community detection in graphs.
\newblock {\em Phys. Rep.}, 486(3-5):75--174, 2010.

\bibitem[FVD{\etalchar{+}}16]{ferrara2016socialbots}
Emilio Ferrara, Onur Varol, Clayton Davis, Filippo Menczer, and Alessandro Flammini.
\newblock The rise of social bots.
\newblock {\em Commun. ACM}, 59(7):96–104, 2016.

\bibitem[GGNW24]{gaudio2024exact}
Julia Gaudio, Charlie Guan, Xiaochun Niu, and Ermin Wei.
\newblock Exact label recovery in {E}uclidean random graphs.
\newblock {\em arXiv preprint arXiv:2407.11163}, 2024.

\bibitem[GJ25]{gaudio2025exact}
Julia Gaudio and Andrew Jin.
\newblock Exact recovery in the geometric {SBM}.
\newblock {\em arXiv preprint arXiv:2512.22773}, 2025.

\bibitem[GJ26]{gaudio2026exactrecoverygeometrichidden}
Julia Gaudio and Andrew Jin.
\newblock Exact recovery in the geometric hidden community model.
\newblock {\em arXiv preprint arXiv:2601.17591}, 2026.

\bibitem[GMPS18]{galhotra2018geometricblock}
Sainyam Galhotra, Arya Mazumdar, Soumyabrata Pal, and Barna Saha.
\newblock The geometric block model.
\newblock {\em Proceedings of the AAAI Conference on Artificial Intelligence}, 32(1), 2018.

\bibitem[GMPS23]{galhotra2023geometricblock}
Sainyam Galhotra, Arya Mazumdar, Soumyabrata Pal, and Barna Saha.
\newblock Community recovery in the geometric block model.
\newblock {\em J. Mach. Learn. Res.}, 24(338):1--53, 2023.

\bibitem[GN02]{girvan2002community}
Michelle Girvan and Mark E.~J. Newman.
\newblock Community structure in social and biological networks.
\newblock {\em Proc. Natl. Acad. Sci. U.S.A.}, 99(12):7821--7826, 2002.

\bibitem[GNW24]{gaudio2024gsbm}
Julia Gaudio, Xiaochun Niu, and Ermin Wei.
\newblock Exact community recovery in the geometric {SBM}.
\newblock In {\em Proceedings of the 2024 {A}nnual {ACM}-{SIAM} {S}ymposium on {D}iscrete {A}lgorithms}, pages 2158--2184, 2024.

\bibitem[GSXY25]{gaudio2025plantedcycles}
Julia Gaudio, Colin Sandon, Jiaming Xu, and Dana Yang.
\newblock Finding planted cycles in a random graph.
\newblock {\em arXiv preprint arXiv:2511.04058}, 2025.

\bibitem[HHLM99]{HartwellEtAl1999ModularCellBiology}
Leland~H. Hartwell, John~J. Hopfield, Stanislas Leibler, and Andrew~W. Murray.
\newblock From molecular to modular cell biology.
\newblock {\em Nature}, 402(6761 Suppl):C47--C52, 1999.

\bibitem[HK11]{hazan2011plantedcliquegametheory}
Elad Hazan and Robert Krauthgamer.
\newblock How hard is it to approximate the best {N}ash equilibrium?
\newblock {\em SIAM J. Comput.}, 40(1):79--91, 2011.

\bibitem[HKK{\etalchar{+}}26]{hsieh2026rigorous}
Jun-Ting Hsieh, Daniel~M. Kane, Pravesh~K. Kothari, Jerry Li, Sidhanth Mohanty, and Stefan Tiegel.
\newblock Rigorous implications of the low-degree heuristic.
\newblock {\em arXiv preprint arXiv:2601.05850}, 2026.

\bibitem[HLL83]{holland1983sbm}
Paul~W. Holland, Kathryn~Blackmond Laskey, and Samuel Leinhardt.
\newblock Stochastic blockmodels: first steps.
\newblock {\em Soc. Netw.}, 5(2):109--137, 1983.

\bibitem[Hop18]{hopkins2018statistical}
Samuel Hopkins.
\newblock {\em Statistical inference and the sum of squares method}.
\newblock PhD thesis, Cornell University, 2018.

\bibitem[HRH02]{hoff2002latent}
Peter~D. Hoff, Adrian~E. Raftery, and Mark~S. Handcock.
\newblock Latent space approaches to social network analysis.
\newblock {\em J. Amer. Statist. Assoc.}, 97(460):1090--1098, 2002.

\bibitem[HRT07]{handcock2007latentcluster}
Mark~S. Handcock, Adrian~E. Raftery, and Jeremy~M. Tantrum.
\newblock Model-based clustering for social networks.
\newblock {\em J. Roy. Statist. Soc. Ser. A}, 170(2):301--354, 2007.

\bibitem[Hul22]{huleihel2022hidden}
Wasim Huleihel.
\newblock Inferring hidden structures in random graphs.
\newblock {\em IEEE Trans. Signal Inform. Process. Netw.}, 8:855--867, 2022.

\bibitem[HW21]{holmgren2021lowdeg}
Justin Holmgren and Alexander~S. Wein.
\newblock Counterexamples to the low-degree conjecture.
\newblock In {\em 12th Innovations in Theoretical Computer Science Conference}, volume 185 of {\em LIPIcs}, pages 75:1--75:9, 2021.

\bibitem[HWX15]{hajek2015computationalcommunity}
Bruce Hajek, Yihong Wu, and Jiaming Xu.
\newblock Computational lower bounds for community detection on random graphs.
\newblock In {\em Proceedings of The 28th Conference on Learning Theory}, volume~40 of {\em PMLR}, pages 899--928, 2015.

\bibitem[IS03]{ingster2003nonparametric}
Yu~I. Ingster and Irina~A. Suslina.
\newblock {\em Nonparametric Goodness-of-Fit Testing Under Gaussian Models}, volume 169 of {\em Lecture Notes in Statistics}.
\newblock Springer, 2003.

\bibitem[Jer92]{jerrum1992plantedclique}
Mark Jerrum.
\newblock Large cliques elude the {M}etropolis process.
\newblock {\em Random Structures Algorithms}, 3(4):347--359, 1992.

\bibitem[JL15]{jiang2015wishart}
Tiefeng Jiang and Danning Li.
\newblock Approximation of rectangular beta-{L}aguerre ensembles and large deviations.
\newblock {\em J. Theoret. Probab.}, 28(3):804--847, 2015.

\bibitem[Joh01]{johnstone2001spiked}
Iain~M. Johnstone.
\newblock On the distribution of the largest eigenvalue in principal components analysis.
\newblock {\em Ann. Statist.}, 29(2):295--327, 2001.

\bibitem[JP00]{juels2000plantedcliquecrypto}
Ari Juels and Marcus Peinado.
\newblock Hiding cliques for cryptographic security.
\newblock {\em Des. Codes Cryptogr.}, 20(3):269--280, 2000.

\bibitem[JV26]{jia2026lowdegree}
He~Jia and Aravindan Vijayaraghavan.
\newblock Low-degree method fails to predict robust subspace recovery.
\newblock {\em arXiv preprint arXiv:2603.02594}, 2026.

\bibitem[Kra07]{krasikov2007jacobi}
Ilia Krasikov.
\newblock An upper bound on {J}acobi polynomials.
\newblock {\em J. Approx. Theory}, 149(2):116--130, 2007.

\bibitem[KSWY25]{kunisky2025directed}
Dmitriy Kunisky, Daniel~A. Spielman, Alexander~S. Wein, and Xifan Yu.
\newblock Statistical inference of a ranked community in a directed graph.
\newblock In {\em {P}roceedings of the 57th {A}nnual {ACM} {S}ymposium on {T}heory of {C}omputing}, pages 2107--2117, 2025.

\bibitem[Ku{\v{c}}95]{kucera1995plantedclique}
Lud{\v{e}}k Ku{\v{c}}era.
\newblock Expected complexity of graph partitioning problems.
\newblock {\em Discrete Appl. Math.}, 57(2-3):193--212, 1995.

\bibitem[KWB22]{kunisky2019notes}
Dmitriy Kunisky, Alexander~S. Wein, and Afonso~S. Bandeira.
\newblock Notes on computational hardness of hypothesis testing: predictions using the low-degree likelihood ratio.
\newblock In {\em Mathematical analysis, its applications and computation}, volume 385 of {\em Springer Proc. Math. Stat.}, pages 1--50. Springer, 2022.

\bibitem[LKJ09]{lkj}
Daniel Lewandowski, Dorota Kurowicka, and Harry Joe.
\newblock Generating random correlation matrices based on vines and extended onion method.
\newblock {\em J. Multivariate Anal.}, 100(9):1989--2001, 2009.

\bibitem[LM00]{laurentmassart}
B{\'e}atrice Laurent and Pascal Massart.
\newblock Adaptive estimation of a quadratic functional by model selection.
\newblock {\em Ann. Statist.}, 28(5):1302--1338, 2000.

\bibitem[LMSY22]{liu2022rgg}
Siqi Liu, Sidhanth Mohanty, Tselil Schramm, and Elizabeth Yang.
\newblock Testing thresholds for high-dimensional sparse random geometric graphs.
\newblock In {\em {P}roceedings of the 54th {A}nnual {ACM} {S}ymposium on {T}heory of {C}omputing}, pages 672--677, 2022.

\bibitem[LMSY23]{liu2023expansion}
Siqi Liu, Sidhanth Mohanty, Tselil Schramm, and Elizabeth Yang.
\newblock Local and global expansion in random geometric graphs.
\newblock In {\em Proceedings of the 55th Annual ACM Symposium on Theory of Computing}, pages 817--825, 2023.

\bibitem[LPRZ25]{lee2025plantedsubgraph}
Daniel~Z. Lee, Francisco Pernice, Amit Rajaraman, and Ilias Zadik.
\newblock The fundamental limits of recovering planted subgraphs (extended abstract).
\newblock In {\em Proceedings of Thirty Eighth Conference on Learning Theory}, volume 291 of {\em PMLR}, pages 3578--3579, 2025.

\bibitem[LR23a]{liu2023noisyrgg}
Suqi Liu and Mikl\'os~Z. R\'acz.
\newblock Phase transition in noisy high-dimensional random geometric graphs.
\newblock {\em Electron. J. Stat.}, 17(2):3512--3574, 2023.

\bibitem[LR23b]{liuracz}
Suqi Liu and Mikl\'os~Z. R\'acz.
\newblock A probabilistic view of latent space graphs and phase transitions.
\newblock {\em Bernoulli}, 29(3):2417--2441, 2023.

\bibitem[LS23]{li2024spectralclusteringgaussianmixture}
Shuangping Li and Tselil Schramm.
\newblock Spectral clustering in the {G}aussian mixture block model.
\newblock {\em arXiv preprint arXiv:2305.00979}, 2023.

\bibitem[Mat13]{matousek2013lectures}
Ji\v{r}\'{\i} Matou\v{s}ek.
\newblock {\em Lectures on {D}iscrete {G}eometry}, volume 212 of {\em Graduate Texts in Mathematics}.
\newblock Springer-Verlag, 2013.

\bibitem[MMX21]{moharrami2021plantedmatching}
Mehrdad Moharrami, Cristopher Moore, and Jiaming Xu.
\newblock The planted matching problem: phase transitions and exact results.
\newblock {\em Ann. Appl. Probab.}, 31(6):2663--2720, 2021.

\bibitem[MNWS{\etalchar{+}}25]{mossel2025inference}
Elchanan Mossel, Jonathan Niles-Weed, Youngtak Sohn, Nike Sun, and Ilias Zadik.
\newblock Sharp thresholds in inference of planted subgraphs.
\newblock {\em Ann. Appl. Probab.}, 35(1):523--563, 2025.

\bibitem[MO18]{morselli2018collusion}
Carlo Morselli and Marie Ouellet.
\newblock Network similarity and collusion.
\newblock {\em Soc. Netw.}, 55:21--30, 2018.

\bibitem[MST19]{massoulie2019plantedtrees}
Laurent Massouli\'{e}, Ludovic Stephan, and Don Towsley.
\newblock Planting trees in graphs, and finding them back.
\newblock In {\em Proceedings of the Thirty-Second Conference on Learning Theory}, volume~99 of {\em PMLR}, pages 2341--2371, 2019.

\bibitem[MWX26]{mao2026rggsmooth}
Cheng Mao, Yihong Wu, and Jiaming Xu.
\newblock Random geometric graphs with smooth kernels: sharp detection threshold and a spectral conjecture.
\newblock {\em arXiv preprint arXiv:2602.14998}, 2026.

\bibitem[MWZ23]{mao2023planteddensecycle}
Cheng Mao, Alexander~S. Wein, and Shenduo Zhang.
\newblock Detection-recovery gap for planted dense cycles.
\newblock In {\em Proceedings of Thirty Sixth Conference on Learning Theory}, volume 195 of {\em PMLR}, pages 2440--2481, 2023.

\bibitem[MWZ25]{mao2025planteddensecycle}
Cheng Mao, Alexander~S. Wein, and Shenduo Zhang.
\newblock Information-theoretic thresholds for planted dense cycles.
\newblock {\em IEEE Trans. Inform. Theory}, 71(2):1266--1282, 2025.

\bibitem[MZ24]{mao2024latent}
Cheng Mao and Shenduo Zhang.
\newblock Impossibility of latent inner product recovery via rate distortion.
\newblock In {\em 2024 60th Annual Allerton Conference on Communication, Control, and Computing}, pages 01--08, 2024.

\bibitem[NC16]{newman2016annotated}
Mark E.~J. Newman and Aaron Clauset.
\newblock Structure and inference in annotated networks.
\newblock {\em Nat. Commun.}, 7:11863, 2016.

\bibitem[New10]{Newman2010Networks}
Mark E.~J. Newman.
\newblock {\em Networks: An Introduction}.
\newblock Oxford University Press, 2010.

\bibitem[PCMP05]{PriebeEtAl2005ScanEnron}
Carey~E. Priebe, John~M. Conroy, David~J. Marchette, and Youngser Park.
\newblock Scan statistics on {E}nron graphs.
\newblock {\em Comput. Math. Organiz. Theor.}, 11:229--247, 2005.

\bibitem[Pen03]{penrosebook}
Mathew Penrose.
\newblock {\em Random Geometric Graphs}, volume~5 of {\em Oxford Studies in Probability}.
\newblock Oxford University Press, 2003.

\bibitem[PHT{\etalchar{+}}21]{pacheco2021uncovering}
Diogo Pacheco, Pik{-}Mai Hui, Christopher Torres{-}Lugo, Bao~Tran Truong, Alessandro Flammini, and Filippo Menczer.
\newblock Uncovering coordinated networks on social media: methods and case studies.
\newblock In {\em Proceedings of the ICWSM}, volume~15, pages 455--466, 2021.

\bibitem[PP20]{peche2020community}
Sandrine Peche and Vianney Perchet.
\newblock Robustness of community detection to random geometric perturbations.
\newblock In {\em Advances in Neural Information Processing Systems}, volume~33, pages 17827--17837, 2020.

\bibitem[PW21]{paquette2021random}
Elliot Paquette and Andrew~Vander Werf.
\newblock Random geometric graphs and the spherical {W}ishart matrix.
\newblock {\em arXiv preprint arXiv:2110.10785}, 2021.

\bibitem[RR19]{raczrichey}
Mikl\'os~Z. R\'acz and Jacob Richey.
\newblock A smooth transition from {W}ishart to {GOE}.
\newblock {\em J. Theoret. Probab.}, 32(2):898--906, 2019.

\bibitem[SM03]{SpirinMirny2003ProteinModules}
Victor Spirin and Leonid~A. Mirny.
\newblock Protein complexes and functional modules in molecular networks.
\newblock {\em Proc. Natl. Acad. Sci. U.S.A.}, 100(21):12123--12128, 2003.

\bibitem[Sod07]{sodin2007tail}
Sasha Sodin.
\newblock Tail-sensitive {G}aussian asymptotics for marginals of concentrated measures in high dimension.
\newblock In {\em Geometric aspects of functional analysis}, volume 1910 of {\em Lecture Notes in Math.}, pages 271--295. Springer, 2007.

\bibitem[VAC15]{pdssparse}
Nicolas Verzelen and Ery Arias-Castro.
\newblock Community detection in sparse random networks.
\newblock {\em Ann. Appl. Probab.}, 25(6):3465--3510, 2015.

\bibitem[Ver18]{hdpbook}
Roman Vershynin.
\newblock {\em High-Dimensional Probability}, volume~47 of {\em Cambridge Series in Statistical and Probabilistic Mathematics}.
\newblock Cambridge University Press, 2018.

\bibitem[Vu07]{vutracemethod}
Van~H. Vu.
\newblock Spectral norm of random matrices.
\newblock {\em Combinatorica}, 27(6):721--736, 2007.

\bibitem[War16]{warnke2016typical}
Lutz Warnke.
\newblock On the method of typical bounded differences.
\newblock {\em Combin. Probab. Comput.}, 25(2):269--299, 2016.

\bibitem[Wei25]{wein2025computational}
Alexander~S. Wein.
\newblock Computational complexity of statistics: new insights from low-degree polynomials.
\newblock {\em arXiv preprint arXiv:2506.10748}, 2025.

\bibitem[Wig58]{wignertracemethod}
Eugene~P. Wigner.
\newblock On the distribution of the roots of certain symmetric matrices.
\newblock {\em Ann. of Math. (2)}, 67:325--327, 1958.

\bibitem[WK19]{wachs2019collusion}
Johannes Wachs and J{\'a}nos Kert{\'e}sz.
\newblock A network approach to cartel detection in public auction markets.
\newblock {\em Sci. Rep.}, 9(1):10818, 2019.

\bibitem[WM25]{wee2025cluster}
Timothy L.~H. Wee and Cheng Mao.
\newblock Cluster expansion of the log-likelihood ratio: optimal detection of planted matchings.
\newblock {\em arXiv preprint arXiv:2512.14567}, 2025.

\bibitem[YZZ25]{yu2025countingstars}
Xifan Yu, Ilias Zadik, and Peiyuan Zhang.
\newblock Counting stars is constant-degree optimal for detecting any planted subgraph.
\newblock {\em Math. Stat. Learn.}, 8(1-2):105--164, 2025.

\end{thebibliography}
}
\newpage
\appendix

\section{Auxiliary lemmas}

\subsection{Concentration inequalities}
Here, we record several concentration inequalities that will be used throughout our analysis. Each inequality will be referenced in different parts, and readers may visit these results later as needed.

The first two results (Lemmas~\ref{lem:ertriconc} and \ref{lem:goecubetraceconc}) consider certain polynomials of well-behaved distributions. These can be derived as corollaries of a general concentration inequality for polynomials of subgaussian variables \cite[Theorem 1.4]{adamczak2015concentration}.

\begin{lemma}[Concentration of signed triangle count]\label{lem:ertriconc}
    Let $G \sim \calG(n, p)$. Then there exists a constant $C_{\ref{lem:ertriconc}} > 0$ such that for $f_{\tri}(G) = \sum_{i < j < \ell \in [n]}(G_{ij}-p)(G_{j\ell}-p)(G_{i\ell}-p)$, for any $t > 0$,
    \[\Prob(|f_{\tri}(G)| \geq t) \leq 2\exp\left(-\frac{1}{C_{\ref{lem:ertriconc}}}\left(\frac{t^2\log^3(1/p)}{n^3} \wedge \frac{t\log^{3/2}(1/p)}{\sqrt{n}} \wedge t^{2/3}\log(1/p)\right)\right)\,. 
    \]
\end{lemma}

\begin{lemma}[Concentration of trace of GOE]\label{lem:goecubetraceconc}
Let $Y \sim \GOE(n)$, that is, $Y \in \reals^{n \times n}$ is a symmetric matrix with $Y_{ii} \sim \calN(0, 2), i \in [n]$ and $Y_{ij} = Y_{ji} \sim \calN(0, 1), i < j \in [n]$ independently. Then there exists a constant $C_{\ref{lem:goecubetraceconc}} > 0$ such that for 
\begin{align*}
    f_2(Y) := \Tr(Y^2) \quad \text{and} \quad f_3(Y) := \Tr(Y^3)\,,
\end{align*}
for any $t > 0$,
\begin{align*}
    \Prob(|\Tr(Y^2) - (n^2 + n)| \geq t) &\leq 2\exp\left(-\frac{1}{C_{\ref{lem:goecubetraceconc}}} \left(\frac{t^2}{n^2} \wedge \frac{t}{n}\right)\right)\,, \\
    \Prob(|\Tr(Y^3)| \geq t) &\leq 2\exp\left(-\frac{1}{C_{\ref{lem:goecubetraceconc}}}\left(\frac{t^2}{n^3} \wedge t^{2/3}\right)\right)\,.
\end{align*}
\end{lemma}

We also use the following localized version of martingale concentration (i.e., Freedman's inequality). Its proof can be found in, \eg, \cite[Lemma 2.2]{warnke2016typical}. 
\begin{lemma}[Freedman's inequality]\label{lem:freedmanconc} Let $\{M_i\}_{0 \leq i \leq n}$ be a martingale with respect to filtration $\{\calF_i\}_{0 \leq i \leq n}$ and $U_i, i \in [n]$ be $\calF_{i-1}$-measurable random variable such that $M_i - M_{i-1} \leq U_i$ for all $i \in [n]$. Then for any $t > 0, v > 0$, and $m>0$,
\[\Prob\left(M_n - M_0 \geq  t, \sum_{i=1}^n \Var[M_i - M_{i-1} | \calF_{i-1}] \leq v, \max_{1 \leq i \leq n} U_i \leq m\right) \leq \exp\left(-\left(\frac{t^2}{4v} \wedge \frac{t}{2m}\right) \right)\,.
\]
\end{lemma}

The following result captures how the neighborhood distributions under $\calG(n, p)$ and $\calG(n, p, d)$ are different.\footnote{The inequality in Lemma~\ref{lem:capsandanticaps} is stated with logarithmic terms in $d, 1/p$ instead of logarithmic terms in $n$ as in \cite[Corollary 6.1]{liu2022rgg}. This is only to avoid the formal condition of $d \leq \mathrm{poly}(n)$ used in the latter, and the current form can be derived as in the proof of \cite[Corollary 6.1]{liu2022rgg} from \cite[Lemma 5.1]{liu2022rgg}.} Note that for a single vertex, the distributions are marginally equal; the lemma compares the distributions conditioned on the latents $\{U_i\}$.
\begin{lemma}[Concentration of spherical caps; {\cite[Lemma 5.1, Corollary 6.1]{liu2022rgg}}]\label{lem:capsandanticaps} Assume $1/n \leq p \leq 1/2$, and consider the neighborhood distributions of vertex $\ell+1$ with vertices $1, \dots, \ell$ under $\calG(n, p)$ and $\calG(n, p, d)$. That is, for $\gamma \in \{0, 1\}^\ell$, let
\begin{align*}
    \calQ_{\ell+1}(\gamma) &:= p^{\sum_{i=1}^\ell \gamma_i}(1-p)^{\ell - \sum_{i=1}^\ell \gamma_i}, \\
    \widetilde{\calP}_{\ell+1}(\gamma | U_1, \dots, U_\ell) &:= \E_{U_{\ell+1}}\left[\prod_{i \in [\ell]: \gamma_i = 1}\bm{1}\{\iprod{U_{\ell+1}}{U_i} \geq \tau\}\prod_{i \in [\ell]: \gamma_i = 0}\bm{1}\{\iprod{U_{\ell+1}}{U_i} < \tau\}\right]\,,
\end{align*}
where $U_i \overset{\mathrm{i.i.d.}}{\sim} \calU(\mathbb{S}^{d-1})$. Then there exists a constant $C_{\ref{lem:capsandanticaps}} > 0$ such that for all $t \geq 0$ and  sufficiently large $n$,
\[\Prob_{U_1, \dots, U_\ell}\left(\left|\frac{\widetilde{\calP}_{\ell+1}(\gamma | U_1, \dots, U_\ell)}{\calQ_{\ell+1}(\gamma)} - 1\right| \geq t\right) \leq C_{\ref{lem:capsandanticaps}}\exp\left(-\frac{d(t^2 \wedge 1)}{C_{\ref{lem:capsandanticaps}}M \log(1/p)\log(d/p)} + C_{\ref{lem:capsandanticaps}}\log n\right)\,,
\]
where $M = M(\ell, p, \gamma):= 
\left(\left(\sum_{i=1}^\ell \gamma_i\right)\log(1/p) + \ell p + \log(d/p)\right)\left(\sum_{i=1}^\ell \gamma_i + \ell p\right)$.
\end{lemma}

Finally, we record the following standard results. The first one is on the concentration of the $\chi^2$ distribution; the second one states that the eigenvalues of the spherical Wishart are concentrated around $1$.

\begin{lemma}[Concentration of $\chi^2(d)$]\label{lem:chisqconc}
    Let $Z_i \sim \calN(0, I_d), i \in [k]$ be independent. Then with probability at least $1 - 2k\exp(-d/1000)$,
    \[\frac{\norm{Z_i}}{\sqrt{d}} \in [0.9, 1.1] \text{ for all }i \in [k]\,.
    \]
\end{lemma}
\begin{proof}
    The proof follows from~\cite[Lemma 1]{laurentmassart} and a union bound.
\end{proof}

\begin{lemma}[Spectrum of spherical Wishart]\label{lem:sphwishartspec}
    Let $d \geq k$ and $U = \begin{pmatrix}U_1 | \dots | U_k\end{pmatrix}\in \reals^{d \times k}$ be such that $U_i \sim \calU(\mathbb{S}^{d-1}), i \in [k]$ are independent. Then there exists a constant $C_{\ref{lem:sphwishartspec}} > 0$ such that with probability at least $1 - (2k+2)\exp(-k/C_{\ref{lem:sphwishartspec}})$,
    \[\norm{U^TU - I_k}_{\op} \leq C_{\ref{lem:sphwishartspec}}\sqrt{\frac{k}{d}}\,.
    \]
\end{lemma}
\begin{proof}
    Write $U_i = Z_i / \norm{Z_i}$ where $Z_i \overset{\mathrm{i.i.d.}}{\sim}\calN(0, I_d), i \in [k]$ and $Z = (Z_1 | \dots | Z_k)$. Then by \cite[Theorem 4.6.1]{hdpbook}, we have
    \[\norm{\frac{Z^T Z}{d} - I_k}_{\mathrm{op}} = O\left(\sqrt{\frac{k}{d}}\right)\,,
    \]
    with probability at least $1 - 2\exp(-\Omega(k))$. As $U = ZD$ where $D := \mathrm{diag}(1/\norm{Z_1}, \dots, 1/\norm{Z_k})$, we have
    \begin{align*}
        \norm{U^TU - I_k}_{\op} &= \norm{D Z^TZD - I_k}_{\op} \\
        &\leq \norm{\sqrt{d}D\left(\frac{Z^TZ}{d} - I_k\right)\sqrt{d}D}_{\op} + \norm{dD^2 - I_k}_{\op} \\
        &= O\left(\sqrt{\frac{k}{d}}\right)\,,
    \end{align*}
    with probability at least $1 - (2k+2)\exp(-\Omega(k))$, where we used that $\Prob(|\|Z_1\|^2- d| > 4\sqrt{dk}) \leq 2\exp(-\Omega(k))$~\cite[Lemma 1]{laurentmassart} and union bound. 

\end{proof}

\subsection{Upper bound on signed subgraph counts}
Aside from specific cases, it is difficult to directly calculate the expectation of signed subgraph count under the random geometric graph. A recent work \cite{bangachev2024fourier} provided a neat solution to this, showing an upper bound that applies to all subgraphs of moderate size and only depends on the numbers of their vertices and edges. 
\begin{lemma}[Upper bound on signed subgraph count; {\cite[Theorem 1.1]{bangachev2024fourier}}]\label{lem:rggfourier} Assume that there exists a constant $\delta > 0$ such that $d \geq n^{\delta}$ and $n^{-1+\delta} \leq p \leq 1/2$. Then there exists a constant $C_{\ref{lem:rggfourier}} = C_{\ref{lem:rggfourier}}(\delta) > 0$ such that the following holds: for any connected graph $H$ that satisfies $C_{\ref{lem:rggfourier}}v(H)e(H)\log^{3/2}d \leq \sqrt{d}$,
\[\left|\E_{G \sim \calG(n, p, d)}\left[\prod_{ij \in E(H)} (G_{ij} - p)\right] \right| \leq (8p)^{e(H)}\left(\frac{C_{\ref{lem:rggfourier}}v(H)e(H)\log^{3/2}d}{\sqrt{d}}\right)^{\lceil (v(H)-1)/2 \rceil}\,.
\]
\end{lemma}

\section{Deferred proofs in Section~\ref{sec:ub}}

\subsection{Variance of signed triangle count (Lemma~\ref{lem:signedtrianglecountvar})}\label{appendix:signedtrianglecountvar}

We express the condition
\begin{equation}\label{eq:ubsignedtrianglecount2}
    \gamma_{\tri} \gg \sqrt{\Var_{G \sim \calP}[f_{\tri}(G)]}\,,
\end{equation}
as an inequality with respect to a specific parameter $\eps = \eps(p, d) > 0$, and invoke its lower bound to translate that as a condition on $d$. Namely, define
    \[\eps = \eps(p, d) := \frac{\Prob_{G \sim \calG(n, p, d)}(G_{12}G_{13}G_{23}=1|G_{23}=1)}{p^2}-1\,.
    \]
    Using the pairwise independence of edges incident to a common vertex (i.e., $ \E_{G \sim \calG(n, p, d)}[G_{12}G_{13}] = p^2$), the expectation of the signed triangle simplifies to:
    \[
          \E_{G \sim \calG(n, p, d)}[(G_{12}-p)(G_{13}-p)(G_{23}-p)] =  \E_{G \sim \calG(n, p, d)}[G_{12}G_{13}G_{23}] - p^3 = p^3 \eps\,,
    \]
    where the last equality holds by $ \E_{G \sim \calG(n, p, d)}[G_{12}G_{13}G_{23}] = p  \Prob_{G \sim \calG(n, p, d)}(G_{12}G_{13}G_{23}=1 | G_{23}=1)$ and the definition of $\eps$. Then, we get
    \begin{equation}\label{eq:signedtrianglecounteps}
        \gamma_{\tri} = \frac{1}{2}\E_{G \sim \calP}[f_{\tri}(G)] = \frac{1}{2}\binom{n}{3} \left(\frac{k}{n}\right)^3 \E_{G \sim \calG(n, p, d)}[(G_{12}-p)(G_{13}-p)(G_{23}-p)] = \Omega(k^3p^3\eps) \,,
    \end{equation}
    where the first equality holds by  Lemma~\ref{lem:fouriersimple}, and the second holds by the previous displayed equality. Following through the lines of \cite[Equations 21--25]{liu2022rgg} with Lemma~\ref{lem:fouriersimple}, 
    for $\overline{T}_{ijl} := (G_{ij}-p)(G_{jl}-p)(G_{il}-p) - \E_{G \sim \calP}[(G_{ij}-p)(G_{jl}-p)(G_{il}-p)]$ we obtain    \begin{equation}\label{eq:signedtrianglevardecomp}
        \Var_{G \sim \calP}[f_{\tri}(G)] = \binom{n}{3}\E_{G \sim \calP}[(\overline{T}_{123})^2] + 12\binom{n}{4}\E_{G \sim \calP}[\overline{T}_{123}\overline{T}_{124}] + 30\binom{n}{5}\E_{G \sim \calP}[\overline{T}_{123}\overline{T}_{145}]\,.
    \end{equation}
    For the first term in \eqref{eq:signedtrianglevardecomp}, we have  \begin{equation}\label{eq:signedtrianglevartrisquared}
    \begin{aligned}
         \E_{G \sim \calP}[(\overline{T}_{123})^2]
         & \leq \E_{G \sim \calP}[(T_{123})^2] \\&= (1-2p)^3\E_{G \sim \calP}[G_{12}G_{23}G_{13}] + 3(1-2p)^2p^4 + 3(1-2p)p^5 + p^6 \\
        &= (1-2p)^3p^3(1 + (k/n)^3\eps) + 3(1-2p)^2p^4 + 3(1-2p)p^5 + p^6 \\
        &= O\left(p^3(1 + (k/n)^3 \eps)\right)\,.
    \end{aligned}
    \end{equation}
    For the second term in \eqref{eq:signedtrianglevardecomp}, we first have
    \begin{align*}
        \E_{G \sim \calP}[\overline{T}_{123}\overline{T}_{124}] &\leq  \E_{G \sim \calP}[T_{123}T_{124}] \\ 
        &= \left(\frac{k}{n}\right)^4 \E_{G \sim \calG(n, p, d)}[T_{123}T_{124}] \\
        &\leq \left(\frac{k}{n}\right)^4 \E_{\iprod{U_1}{U_2}}\left[\left(\E_{U_1}[G_{13}G_{23} | \iprod{U_1}{U_2}]- p^2\right)^2\right]\,,
    \end{align*}
    where the equality is from Lemma~\ref{lem:fouriersimple}, and the second inequality is from \cite[Equation 24]{bubeck2016rgg}. The expectation in the last term is equal to a signed 4-cycle count under $\calG(n, p, d)$, since:
    \begin{align*}
        &\E_{\iprod{U_1}{U_2}}\left[\left(\E[G_{13}G_{23} | \iprod{U_1}{U_2}]- p^2\right)^2\right] \\&= \E_{\iprod{U_1}{U_2}}\left[\left(\E_{U_3}[(G_{13}-p)(G_{23}-p)|U_1, U_2]\right)^2\right] \\
        &= \E_{\iprod{U_1}{U_2}}\left[\E_{U_3, U_4}[(G_{13}-p)(G_{23}-p)(G_{14}-p)(G_{24}-p)|U_1, U_2]\right] \\
        &= \E_{\iprod{U_1}{U_2}}\left[\E_{U_3, U_4}[(G_{13}-p)(G_{23}-p)(G_{14}-p)(G_{24}-p)|\iprod{U_1}{U_2}]\right]\,,
    \end{align*}
    where the first and the last equalities follow from the rotational invariance (i.e., conditioning on $(U_1, U_2)$ is equivalent to conditioning on $\iprod{U_1}{U_2}$), and the second equality is from the conditional independence of $(G_{13}-p)(G_{23}-p)$ and $(G_{14}-p)(G_{24}-p)$ given $(U_1, U_2)$. Thus,
\begin{equation}    \begin{aligned}\label{eq:signedtrianglevarK4minusedge}
    \E_{G \sim \calP}[\overline{T}_{123}\overline{T}_{124}] &\leq \left(\frac{k}{n}\right)^4 \E_{G \sim \calG(n, p, d)}[(G_{13}-p)(G_{23}-p)(G_{14}-p)(G_{24}-p)] \\
    &= O\left(\left(\frac{k}{n}\right)^4\frac{p^4\log^2(1/p)}{d}\right)\,,
\end{aligned}
\end{equation}
where the last line follows from Proposition~\ref{prop:signedcyclecount}.
We note that the argument for this part in \cite{liu2022rgg} assumes the conditional independence of $(G_{13}-p)(G_{23}-p)$ and $(G_{14}-p)(G_{24}-p)$ given 
$G_{12}$; however, this does not hold as both quantities depend on the common latent vectors $U_1$ and $U_2$. We therefore take a different approach via 
Proposition~\ref{prop:signedcyclecount}, which involves a mild assumption of $d$ being sufficiently large and $d \geq (5\log(1/p))^4$.
Finally, for the third term in \eqref{eq:signedtrianglevardecomp}, we have
\begin{equation}\label{eq:signedtrianglevarribbon}
    \E_{G \sim \calP}[\overline{T}_{123}\overline{T}_{145}] \leq \E_{G \sim \calP}[T_{123}T_{145}] = \left(\frac{k}{n}\right)^5p^6\eps^2\,.
\end{equation}
Combining \eqref{eq:signedtrianglevartrisquared}, \eqref{eq:signedtrianglevarK4minusedge} and \eqref{eq:signedtrianglevarribbon} into  \eqref{eq:signedtrianglevardecomp}, we have
\[\Var_{G \sim \calP}[f_{\tri}(G)] = O(n^3p^3 + k^3p^3\eps + k^4p^4\log^2(1/p)/d + k^5p^6\eps^2)\,.
\]
Comparing this with \eqref{eq:signedtrianglecounteps}, for establishing \eqref{eq:ubsignedtrianglecount2} it suffices to show that
\begin{equation}\label{eq:signedtrianglecountfinal1}
    k^6p^6\eps^2 \gg n^3p^3 + k^3p^3\eps + k^4p^4\log^2(1/p)/d + k^5p^6\eps^2\,.
\end{equation}
Clearly $k^6p^6\eps^2 \gg k^5p^6\eps^2$ holds. Also, from Lemma~\ref{lem:signedtrianglelb}, we have $\eps = \Omega(\log^{3/2}(1/p)/\sqrt{d})$. Thus a sufficient condition for \eqref{eq:signedtrianglecountfinal1} is
\[d \ll \frac{k^6p^3\log^3(1/p)}{n^3} \wedge k^6p^6\log^3(1/p) \quad \text{and} \quad k^2p^2\log(1/p) \gg 1\,.
\]
If $p \geq 1/n$ then the first inequality on $d$ is satisfied as long as $d \ll k^6p^3\log^3(1/p)/n^3$. The second inequality on $k, p$ is satisfied as long as $p \gg 1/(k\sqrt{\log k})$.

\subsection{Typical behavior of signed wedge count (Lemma~\ref{lem:ubgeneralrggwhp})}\label{appendix:ubgeneralrggwhp}
Without loss of generality, let $S = [s]$ and $S_0 = [k^-]$. Throughout this subsection, we always consider the distribution of $G$ to be $\calP_{[s]}$ and omit the notation $G \sim \calP_{[s]}$ for brevity.

\paragraph*{Verifying \eqref{eq:constrainedscantestrange}.}
We begin with the event \eqref{eq:constrainedscantestrange} on the range surrogate, whose complement has probability at most (by union bound and Markov's inequality) 
\begin{equation}\label{eq:ubgeneralrggmarkov}
\begin{aligned}
    &\sum_{i < j \in [k^-]}\Prob\left(\left|\sum_{\ell < i, \ell \in [k^-]} (G_{\ell i}-p)(G_{\ell j}-p)\right| > B\right) \\
    &\leq k^2 \max_{i < j \in [k^-]}\Prob\left(\left|\sum_{\ell < i, \ell \in [k^-]} (G_{\ell i}-p)(G_{\ell j}-p)\right| > B\right) \\
    &\leq k^2 \frac{\max_{i < j \in [k^-]}\E\left[\left(\sum_{\ell < i, \ell \in [k^-]} (G_{\ell i}-p)(G_{\ell j}-p)\right)^{2m}\right]}{B^{2m}}\,,
\end{aligned}
\end{equation}
for any $m \geq 1$; later, we will choose $m = \lceil \log k \rceil$. To obtain an upper bound on the moment of the wedge count, consider any fixed $i < j \in [k^-]$. Then
\begin{equation}\label{eq:ubgeneralrggmoment}
\begin{aligned}
  &\E\left[\left(\sum_{\ell < i, \ell \in [k^-]}
     (G_{\ell i}-p)(G_{\ell j}-p)\right)^{2m}\right] \\
  &= \sum_{\ell_1, \dots, \ell_{2m}}\E\left[ \prod_{a = 1}^{2m}(G_{\ell_a i}-p)(G_{\ell_a j}-p)\right]\leq \sum_{b = 1}^{2m}\sum_{|\{\ell_1, \dots, \ell_{2m}\}| = b}
       \left|\E \left[\prod_{a = 1}^{2m}(G_{\ell_a i}-p)(G_{\ell_a j}-p)\right]\right|\,,
\end{aligned}
\end{equation}
where each $\ell_1, \dots, \ell_{2m}$ is less than $i$.

The next step is to express each summand in \eqref{eq:ubgeneralrggmoment} as a form of proper subgraph count, to apply Lemma~\ref{lem:rggfourier}. Consider any fixed $\ell_1, \dots, \ell_{2m}$ with $|\{\ell_1, \dots, \ell_{2m}\}| = b$. By (without loss of generality) letting the distinct elements to be $\ell_1, \dots, \ell_b$,
\begin{equation}\label{eq:ubrggwhpexpand}
    \left|\E \left[\prod_{a = 1}^{2m}(G_{\ell_a i}-p)(G_{\ell_a j}-p)\right]\right| = |\E [(G_{\ell_1i}-p)^{m_1}(G_{\ell_1j}-p)^{m_1}\dots (G_{\ell_bi}-p)^{m_b}(G_{\ell_bj}-p)^{m_b}]|\,,
\end{equation}
where $m_a \geq 1$ (each depending on $\ell_1, \dots, \ell_{2m}$) and $\sum_a m_a = 2m$. Since each $G_{\ell_ai}, G_{\ell_aj} \in \{0, 1\}$, we can always express any higher order term as a lower order term. In particular, we have
\begin{align*}
    (x-p)^{m_a} &= ((1-p)^{m_a}-(-p)^{m_a})(x-p)  +p(1-p)((1-p)^{m_a-1} - (-p)^{m_a-1})\\
    &=: C_a(x-p) + D_a\,,
\end{align*}
for $x \in \{0, 1\}$. As $C_a, D_a \geq 0$ (from $p \leq 1/2$) and $C_a \leq 1, D_a \leq p$, \eqref{eq:ubrggwhpexpand} is at most
\begin{align*}
    &|\E [(G_{\ell_1i}-p)^{m_1}(G_{\ell_1j}-p)^{m_1}\dots (G_{\ell_bi}-p)^{m_b}(G_{\ell_bj}-p)^{m_b}]|\\
    &= \left|\sum_{U \subseteq [b]} \sum_{U' \subseteq [b]}\left(\prod_{a \in U}C_a\right)\left(\prod_{a' \in U'}C_{a'}\right)\left(\prod_{a \notin U}D_a\right)\left(\prod_{a' \not \in U'}D_{a'}\right) \E \left[\prod_{a \in U}(G_{ai}-p)\prod_{a' \in U'}(G_{a'j}-p)\right]\right| \\
    &\leq \sum_{U \subseteq [b]} \sum_{U' \subseteq [b]}\left(\prod_{a \in U}C_a\right)\left(\prod_{a' \in U'}C_{a'}\right)\left(\prod_{a \notin U}D_a\right)\left(\prod_{a' \not \in U'}D_{a'}\right) \left|\E \left[\prod_{a \in U}(G_{ai}-p)\prod_{a' \in U'}(G_{a'j}-p)\right]\right| \\
&\leq  \sum_{U \subseteq [b]} \sum_{U' \subseteq [b]} p^{b - |U| + b - |U'|} (8p)^{|U| + |U'|}\leq 2^b \times 2^b \times (8p)^{2b} = (16p)^{2b}\,.
\end{align*}
Here, the first inequality is from $C_a, D_a \geq 0$, and the second inequality is from Lemma~\ref{lem:rggfourier} and $C_a \leq 1, D_a \leq p$. To be specific, the graph corresponding to edges $\{ai: a \in U\} \cup \{a'j: a' \in U'\}$ has at most 2 connected components with $|U| + |U'| \leq 4m$ edges. If $|U| = |U'| = 0$ then the expectation is $1 = (8p)^{|U| + |U'|}$; otherwise, by applying Lemma~\ref{lem:rggfourier} to each component and using that $C_{\ref{lem:rggfourier}} v(H)e(H) \log^{3/2}d / \sqrt{d} \leq C_{\ref{lem:rggfourier}} (4m)^2(\log^{3/2}d)/\sqrt{d} \leq 1$ (for all sufficiently large $d$, with $m = \lceil \log k \rceil$), we obtain the corresponding upper bound $(8p)^{|U| + |U'|}$.
As the final upper bound $(16p)^{2b}$ holds whenever $|\{\ell_1, \dots, \ell_{2m}\}| = b$, plugging this into \eqref{eq:ubgeneralrggmoment} yields
\begin{align*}
    \E\left[\left(\sum_{\ell < i, \ell \in [k^-]}
     (G_{\ell i}-p)(G_{\ell j}-p)\right)^{2m}\right] &\leq \sum_{b = 1}^{2m}\sum_{|\{\ell_1, \dots, \ell_{2m}\}| = b} (16p)^{2b} \\
     &\leq  \sum_{b=1}^{2m}\binom{j-2}{b}b^{2m} (16p)^{2b} \\
     &\leq \sum_{b=1}^{2m}k^{b}b^{2m}(16p)^{2b} = \sum_{b=1}^{2m} b^{2m} (256kp^2)^b\,.
\end{align*}
Here, the second inequality is from $\ell_1, \dots, \ell_{2m}$ being chosen from $[i-1] \subseteq [j-2]$, and the next inequality is from $j \leq k$.

If $256kp^2 \leq 1$, the last term is at most $\sum_{b=1}^{2m}b^{2m} \leq (2m)^{2m+1}$; otherwise, it is at most $\sum_{b=1}^{2m} b^{2m} (256kp^2)^{2m} \leq (2m)(512kp^2m)^{2m}$. Thus for \eqref{eq:ubgeneralrggmarkov}, we obtain an upper bound of
\[k^2 \frac{(2m)^{2m+1} \vee (2m)(512kp^2m)^{2m}}{B^{2m}}\,.
\]
Finally, we choose $m$ and $B$ such that this term is $o(1)$ as $k \to \infty$. In particular, $m = \lceil \log k \rceil$ and $B = 2048kp^2m + 8m$ is valid.

\paragraph*{Verifying \eqref{eq:constrainedscantestvar}.} The result on the variance surrogate is by a simple application of Chebyshev inequality (yet with tedious calculation). For this, we need to calculate the mean and variance of the statistic in \eqref{eq:constrainedscantestvar}. Writing as $W_{\ell ij} := (G_{\ell i}-p)^2(G_{\ell j}-p)^2$ and $Q_{\ell qij} := (G_{\ell i}-p)(G_{\ell j}-p)(G_{qi}-p)(G_{qj}-p)$, the statistic in \eqref{eq:constrainedscantestvar} is equal to
\begin{equation}\label{eq:ubgeneralrggvarterm}
\begin{aligned}
    &\sum_{i<j \in [k^-]} \sum_{\ell < i, \ell \in [k^-]} W_{\ell ij} + 2\sum_{i < j \in [k^-]} \sum_{\ell < q < i,  \ell \in [k^-], q \in [k^-]} Q_{\ell qij} \\
    &= \sum_{\ell < i < j\in [k^-]} W_{\ell ij} + 2\sum_{\ell < q < i < j \in [k^-]}  Q_{\ell qij}\,.
\end{aligned} 
\end{equation}
This follows from expansion:
\begin{align*}
&\sum_{i<j \in [k^-]} \left(\sum_{\ell < i, \ell \in [k^-]}(G_{\ell i}-p)(G_{\ell j}-p)\right)^2 \\
&= \sum_{i<j \in [k^-]} \sum_{\ell < i, \ell \in [k^-]}(G_{\ell i}-p)(G_{\ell j}-p) \sum_{q < i, q \in [k^-]}(G_{qi}-p)(G_{qj}-p) \\
&= \sum_{i<j \in [k^-]} \sum_{\ell < i, \ell \in [k^-]}(G_{\ell i}-p)^2(G_{\ell j}-p)^2 \\
&\quad + 2\sum_{i < j \in [k^-]} \sum_{\ell < q < i, \ell \in [k^-], q \in [k^-]}(G_{\ell i}-p)(G_{\ell j}-p)(G_{qi}-p)(G_{qj}-p)\,,
\end{align*}
where the last equality follows from categorizing based on $\ell = q$ or $\ell \neq q$.

We first calculate the mean of \eqref{eq:ubgeneralrggvarterm}. From $ \E [W_{\ell ij}] = p^2(1-p)^2$ and Proposition~\ref{prop:signedcyclecount} for cycle of length 4, the mean of \eqref{eq:ubgeneralrggvarterm} is at most
\begin{equation}\label{eq:ubgeneralrggmeanub}
\binom{k^-}{3}p^2(1-p)^2 + 2\binom{k^-}{4}  C_{\ref{prop:signedcyclecount}}^4\frac{p^4 \log^2(1/p)}{d} \leq \frac{\sigma^2}{2}\,,
\end{equation}
where the inequality follows from the choice of $\sigma^2$ in \eqref{eq:constrainedscantestvar}.

For the variance of \eqref{eq:ubgeneralrggvarterm}, consider the centered variables $\overline{W}_{\ell ij}:= W_{\ell ij} - \E[W_{\ell ij}]$ and $\overline{Q}_{\ell qij}:= Q_{\ell qij} - \E[Q_{\ell qij}]$. Then the variance of \eqref{eq:ubgeneralrggvarterm} is equal to
\begin{equation}\label{eq:ubgeneralrggvarterm2}
\begin{aligned}
&\E\left[\left(\sum_{\ell < i < j \in [k^-]} \overline{W}_{\ell ij} + 2\sum_{\ell < q < i < j \in [k^-]}  \overline{Q}_{\ell qij}\right)^2\right] \\
&\leq 2 \E\left[\left(\sum_{\ell < i < j \in [k^-]} \overline{W}_{\ell ij} \right)^2\right] + 8 \E\left[\left(\sum_{\ell < q < i < j \in [k^-]}  \overline{Q}_{\ell qij}\right)^2\right]\,.
\end{aligned}    
\end{equation}
The expectations after the inequality can be upper bounded using Lemma~\ref{lem:rggfourier}; it can be shown that (see Appendix~\ref{appendix:ubgeneralrggwedgevar} for details) \eqref{eq:ubgeneralrggvarterm2} is at most
\begin{equation}\label{eq:ubgeneralrggvarub}
O\left(k^4p^3 + \frac{k^5p^6\log^{3} d}{d} + \frac{k^7p^8\log^{9/2} d}{d^{3/2}}\right)\,.
\end{equation}
Finally, we invoke Chebyshev's inequality using the bounds on the mean \eqref{eq:ubgeneralrggmeanub} and the variance \eqref{eq:ubgeneralrggvarub}. Let $X$ be the sum of random variables as in \eqref{eq:ubgeneralrggvarterm}. Then the complementary event of \eqref{eq:constrainedscantestvar} has probability
\[\Prob(X > \sigma^2) \leq \Prob(X - \E[X] > \sigma^2/2) \leq 4\Var[X]/\sigma^4\,.
\]
Thus, it suffices to prove that $\sigma^4 \gg \Var[X]$. As $\sigma^2$ is of order $\Omega(k^3p^2 + k^4p^4\log^2(1/p)/d)$, it suffices to show that
\begin{equation}\label{eq:ubgeneralrggvarcomparison}
(k^3p^2)^2 + \left(\frac{k^4p^4 \log^2(1/p)}{d}\right)^2 \gg k^4p^3+ \frac{k^5p^6\log^3 d}{d} + \frac{k^7p^8\log^{9/2} d}{d^{3/2}}\,.
\end{equation}
We evaluate this asymptotic inequality, for each term on the right hand side of \eqref{eq:ubgeneralrggvarcomparison}. For the first term, $(k^4p^4\log^2(1/p)/d)^2 \gg k^4p^3$ is equivalent to
\[d \ll k^2p^{5/2}\log^2(1/p)\,.
\]
This holds due to the theorem's assumption $d \ll k^2p^3\log^3(1/p)$, combined with $p^{1/2}\log(1/p) = O(1)$.
For the remaining terms, by applying AM-GM inequalities on the left hand side of \eqref{eq:ubgeneralrggvarcomparison} (with respective weights $(1/2, 1/2), (1/4, 3/4)$), we obtain that it is at least
\[\Omega\left(\frac{k^7p^6\log^2(1/p)}{d} \vee \frac{k^{15/2}p^7\log^3 (1/p)}{d^{3/2}}\right)\,.
\]
For this to dominate the remaining terms $\frac{k^5p^6\log^{3} d}{d} + \frac{k^7p^8\log^{9/2} d}{d^{3/2}}$ in \eqref{eq:ubgeneralrggvarcomparison}, it should be that
\[k \gg \frac{\log^{3/2}d}{\log(1/p)} \vee \frac{p^2\log^9 d}{\log^6(1/p)}\,.
\]
From the condition $p \geq n^{-1+\delta}$ and $d = \widetilde{O}(k^2p^3)$, it can be observed that the right hand side is at most polylogarithmic in $n$, whereas the left hand side in at least polynomial in $n$. Thus for each fixed $\delta > 0$, this holds for all sufficiently large $n$.

\subsubsection{Variance of signed wedge count}\label{appendix:ubgeneralrggwedgevar}
Here we record the details on the earlier claim that \eqref{eq:ubgeneralrggvarterm2} is at most \eqref{eq:ubgeneralrggvarub}.

\paragraph*{$\overline{W}_{\ell ij}$ term.} We have
\begin{align*}
    & \E\left[\left(\sum_{\ell < i < j\in [k^-]} \overline{W}_{\ell ij} \right)^2\right] \\
    &= \sum_{\ell < i < j \in [k^-]} \sum_{\ell' < i' < j' \in [k^-]} \E\left[\overline{W}_{\ell ij}\overline{W}_{\ell'i'j'}\right] \\
    &\leq k^3 \E\left[\left(\overline{W}_{123}\right)^2\right] + k^4\left|\E\left[\overline{W}_{123}\overline{W}_{124}\right]\right| + k^4\left|\E\left[\overline{W}_{123}\overline{W}_{234}\right]\right|  + k^4 \left|\E \left[ \overline{W}_{134} \overline{W}_{234} \right]\right| \,,
\end{align*}
because for all other combinations of $ \ell, i, j$ and $\ell', i', j'$, the expectation is 0 by Corollary~\ref{cor:treesimple}. From $(x-p)^2 = (1-2p)(x-p) + p(1-p)$ for $x \in \{0, 1\}$, we have $\overline{W}_{\ell ij} = (1-2p)^2(G_{\ell i}-p)(G_{\ell j}-p) + (1-2p)p(1-p)(G_{\ell i}-p + G_{\ell j}-p)$. Thus
\begin{align*}
    \E\left[\left(\overline{W}_{123}\right)^2\right] &= \E\left[\left((1-2p)^2(G_{12}-p)(G_{13}-p) + (1-2p)p(1-p)(G_{12}-p + G_{13}-p)\right)^2\right] \\
    &\leq 3\E[(G_{12}-p)^2(G_{13}-p)^2] + 3p^2\E[(G_{12}-p)^2] + 3p^2 \E[(G_{13}-p)^2] \\
    &= O(p^2)\,,
\end{align*}
where the inequality follows from $(a+b+c)^2 \leq 3(a^2 + b^2 + c^2)$. By similar expansion, we have
\begin{align*} \E\left[\overline{W}_{123}\overline{W}_{124}\right] = &\E\Big[\left((1-2p)^2(G_{12}-p)(G_{13}-p) + (1-2p)p(1-p)(G_{12}-p + G_{13}-p)\right) \\
    &\quad \times \left((1-2p)^2(G_{12}-p)(G_{14}-p) + (1-2p)p(1-p)(G_{12}-p + G_{14}-p)\right)\Big] \\
    &= p^2(1-p)^2(1-2p)^2\E[(G_{12}-p)^2] \\
    &= O(p^3)\,,
\end{align*}
and
\begin{align*}
    \E\left[\overline{W}_{123}\overline{W}_{234}\right] &= \E\Big[\left((1-2p)^2(G_{12}-p)(G_{13}-p) + (1-2p)p(1-p)(G_{12}-p + G_{13}-p)\right) \\
    &\quad \times \left((1-2p)^2(G_{23}-p)(G_{24}-p) + (1-2p)p(1-p)(G_{23}-p + G_{24}-p)\right)\Big] \\
    &= p(1-p)(1-2p)^3\E[(G_{12}-p)(G_{13}-p)(G_{23}-p)]\\
    &= O(p^4\log^{3/2}(1/p)/d^{1/2})\,,
\end{align*}
where the last line follows from Proposition~\ref{prop:signedcyclecount}.  Also,
\begin{align*}
    \E\left[\overline{W}_{134}\overline{W}_{234}\right] &= \E\Big[\left((1-2p)^2(G_{13}-p)(G_{14}-p) + (1-2p)p(1-p)(G_{13}-p + G_{14}-p)\right) \\
    &\quad \times \left((1-2p)^2(G_{23}-p)(G_{24}-p) + (1-2p)p(1-p)(G_{23}-p + G_{24}-p)\right)\Big] \\
    &= (1-2p)^4\E[(G_{13}-p)(G_{14}-p)(G_{23}-p)(G_{24}-p)] \\
    &= O(p^4\log^2 (1/p)/d)\,,
\end{align*}
again from Proposition~\ref{prop:signedcyclecount}. 
Therefore,
\begin{equation}\label{eq:ubgeneralrggvarwedge}
\begin{aligned}
    \E\left[\left(\sum_{\ell < i < j \in [k^-]} \overline{W}_{\ell ij} \right)^2\right] &= O\left(k^3p^2 + k^4p^3 + \frac{k^4p^4\log^{3/2} (1/p)}{d^{1/2}}  + \frac{k^4p^4\log^2(1/p)}{d} \right) \\
    &= O(k^4p^3)\,,
\end{aligned}
\end{equation}
where the last equality follows from $kp \gtrsim 1$,  $p\log^{3/2}(1/p)/d^{1/2} = o(1),$ and $p\log^2(1/p)/d = o(1)$.
\paragraph*{$\overline{Q}_{\ell qij}$ term.}
Similarly, we have
\begin{align*}
    \E\left[\left(\sum_{\ell < q < i < j \in [k^-]}  \overline{Q}_{\ell qij}\right)^2\right] &= \sum_{\ell < q < i < j \in [k^-]}\sum_{\ell' < q' < i' < j' \in [k^-]}\E\left[\overline{Q}_{\ell qij}\overline{Q}_{\ell'q'i'j'}\right] \\
    &= \sum_{\ell < q < i < j \in [k^-]}\sum_{\ell' < q' < i' < j' \in [k^-]}(\E[Q_{\ell qij}Q_{\ell'q'i'j'}] - (\E[Q_{1234}])^2)\,.
\end{align*}

We use Lemma~\ref{lem:rggfourier}, categorizing based on the union graph $H$ formed by the configuration of $\ell, q, i, j$ and $\ell', q', i', j'$. In the following, we record the upper bound on $|\E[Q_{\ell qij}Q_{\ell'q'i'j'}]|$ based on different values of $(v(H), e(H))$, along with the number of distinct cases with respect to the indices.\footnote{Within this calculation, it is possible that $Q_{\ell qij}Q_{\ell'i'j'q'}$ includes factors that have exponent of 2 (\eg, $(G_{\ell i} - p)^2$), in which case we consider the bound based on its lower-order expansion, namely, $(G_{\ell i}-p)^2 = (1-2p)(G_{\ell i}-p) + p(1-p)$. The case where the cycles with respect to $\ell, q, i, j$ and $\ell', q', i', j'$ do not share a common vertex is not included, as each contribution there is $\E[Q_{\ell qij}Q_{\ell'q'i'j'}] - (\E[Q_{1234}])^2 = 0$.}
\begin{itemize}
    \item $(v(H), e(H)) = (4, 4)$: upper bound of $O(p^4)$ holds; there are $O(k^4)$ such cases.
    \item $(v(H), e(H)) = (4, 6)$: upper bound of $O(p^6(\log d)^3/d)$ holds; there are $O(k^4)$ such cases.
    \item $(v(H), e(H)) = (5, 6)$: upper bound of $O(p^6(\log d)^3/d)$ holds; there are $O(k^5)$ such cases.
    \item $(v(H), e(H)) = (5, 7)$: upper bound of $O(p^7(\log d)^3/d)$ holds; there are $O(k^5)$ such cases.
    \item $(v(H), e(H)) = (6, 7)$: upper bound of $O(p^7(\log d)^{9/2}/d^{3/2})$ holds; there are $O(k^6)$ such cases.
    \item $(v(H), e(H)) = (6, 8)$ (from now on, the two 4-cycles no longer share an edge): upper bound of $O(p^8 (\log d)^{9/2}/d^{3/2})$ holds; there are $O(k^6)$ such cases.
    \item $(v(H), e(H)) = (7, 8)$: upper bound of $O(p^8 (\log d)^{9/2}/d^{3/2})$ holds; there are $O(k^7)$ such cases.
\end{itemize}
Combining these, we have\footnote{Note that $k^6p^7(\log d)^{9/2}/d^{3/2} \ll k^7p^8(\log d)^{9/2}/d^{3/2}$ holds, which is equivalent to $kp \gg 1$. To see this, from the assumption of Theorem~\ref{thm:ubscangeneral} we have $k^3p^3\log^3(1/p)\geq k^2p^3\log^3(1/p) = \Omega(n^{\delta})$ and $\log(1/p) = O(\log n)$, which together implies $kp = \Omega(n^{\delta/3}/\log n)$.}
\begin{equation}\label{eq:ubgeneralrggvarcycle}
    \E\left[\left(\sum_{\ell < q < i < j \in [k^-]} \overline{Q}_{\ell qij} \right)^2\right] = O\left(k^4p^4 + \frac{k^5p^6\log^{3}d}{d}+ \frac{k^7p^8\log^{9/2}d}{d^{3/2}}\right)\,.
\end{equation}
Combining \eqref{eq:ubgeneralrggvarwedge} and \eqref{eq:ubgeneralrggvarcycle} (plus $k^4p^3 \gtrsim k^4p^4$), we obtain \eqref{eq:ubgeneralrggvarub}.

\section{Deferred proofs in Section~\ref{sec:lb}}

\subsection{TV distance between Wishart and spherical Wishart (\prettyref{prop:sphwisandwis})}\label{appendix:sphwisandwis}

    Let $K := \binom{k}{2}$, $Y \in \reals^K$ be a vector with entries $\frac{\iprod{Z_i}{Z_j}}{\norm{Z_i}\norm{Z_j}}, i < j \in [k]$ and $W \in \reals^K$ be a vector with entries $\frac{\iprod{Z_i}{Z_j}}{d}, i < j \in [k]$. 

    The starting point is to express $W$ as a nearly isotropic linear transformation of $Y$. Namely, let $D \in \reals^{k \times k}$ be diagonal with $\norm{Z_i}/\sqrt{d}$ being its $(i, i)$ entry, and $T_D: \reals^K \to \reals^K$ be a linear map defined as $(T_D(x))_{ij} := D_{ii}D_{jj}x_{ij}, i < j \in [k]$. As key facts, one can observe that $W = T_D(Y)$, and $D$ is independent of $Y$ (due to the independence between $Z_i/\norm{Z_i}$ and $\norm{Z_i}$). These together imply that
    \[
    W \overset{\mathrm{d}}{=}\E_D[T_D(Y)] \,.
    \]
    Furthermore, the map $T_D$ is close to an identity map: by defining $\calE$ as $\{D \in \calE\} = \{D_{ii} \in [0.9, 1.1] \text{ for all } i\}$, we have $\Prob(D \notin \calE) = o(1)$ by Lemma~\ref{lem:chisqconc}.
    Then,
    \begin{equation}\label{eq:sphwisandwis2}
    \begin{aligned}
        \dTV(Y, W) &= \dTV(Y, \E_D[T_D(Y)]) \\
        & \overset{(a)}{\leq} \E_D[\dTV(Y, T_D(Y))] \\
        &\overset{(b)}{\leq} \Prob(D \notin \calE) + \E_D[\dTV(Y, T_D(Y))\bm{1}\{D \in \calE\}] \\
        &\overset{}{\leq}  o(1) + \E_D[\dTV(Y, T_D(Y))\bm{1}\{D \in \calE\}] \\
        &\overset{(c)}{\leq} o(1) + \sqrt{\E_D[\dTV(Y, T_D(Y))^2\bm{1}\{D \in \calE\}]} \,,
    \end{aligned}
    \end{equation}
    where $(a)$ holds by Jensen's inequality and the convexity of TV distance; $(b)$ holds by $\dTV \leq 1$; $(c)$ holds by Cauchy-Schwarz inequality. %

    To upper bound $\dTV(Y, T_D(Y))$ appearing in the last line of~\eqref{eq:sphwisandwis2}, we invoke the explicit formulae of their densities. 
    For $y \in \reals^K$, define $\overline{y} \in \reals^{k \times k}$ to be a symmetric matrix such that the diagonals are $0$ and the off-diagonals are $y_{ij}$. Then, the density of $Y$ at $y$ is given as follows\footnote{The distribution is also known in the literature as LKJ distribution \cite{lkj}.} (see, \eg, \cite[Lemma 12]{paquette2021random}):
    \begin{equation}\label{eq:sphwisdensity}
        f(y) = \det(I_k + \overline{y})^{(d-k-1)/2}\bm{1}\{I_k + \overline{y} \succ 0\} / Z(k, d)\,,
    \end{equation}
    where $Z(k, d)$ is a normalization constant. As a linear transformation of $Y$ (with $D \in \calE$), the density $g_D$ of $T_D(Y)$ at $w$ is
    \begin{equation}\label{eq:sphwisandwisvariablechange}
        g_D(w) = f\left(T_D^{-1}(w)\right) \left |\det \frac{\partial T_D^{-1}(w)}{\partial w} \right| = f\left(T_D^{-1}(w)\right) \det(D)^{-(k-1)}\,.
    \end{equation}
    Note that for $D \in \calE$, $T_D^{-1}$ is well-defined.

    In order to prevent the ratio $g_D/f$ being unbounded, we compare the densities in a smaller set $\calE'$ defined as
    \[\calE' := \{y \in \reals^K: \norm{\overline{y}}_{\op} \leq  0.1\}\,.
    \]
    Notably, both $Y \in \calE' \Leftrightarrow T_D(Y) \in T_D(\calE')$ and $Y \in T_D(\calE')$ are high probability events, for any $D \in \calE$. This follows from: $\Prob(Y \notin \calE') \leq 2\exp(-\Omega(k))$ by Lemma~\ref{lem:sphwishartspec} with $k/d=o(1)$, and similarly $\Prob_Y(Y \notin T_D(\calE')) \leq 2\exp(-\Omega(k))$ from $\norm{\overline{T_D^{-1}(Y)}}_{\op} = \norm{D^{-1}\overline{Y}D^{-1}}_{\op} \leq \norm{D^{-1}}_{\op}^2\norm{\overline{Y}}_{\op} \leq (1/0.9)^2 \norm{\overline{Y}}_{\op}$. By letting $r_D(y) := g_D (y)/ f(y)$, 
    \begin{align*}
        \dTV(Y, T_D(Y)) &= \frac{1}{2}\int_{\reals^K}|f(y) - g_D(y)|dy \\
        &\leq \frac{1}{2}\Prob(Y \not \in \calE') + \frac{1}{2}\Prob_Y(Y \not \in T_D(\calE')) + \frac{1}{2}\int_{\reals^K}|f(y) - g_D(y)|\bm{1}\{y \in T_D(\calE')\} dy\\
        &\leq o(1) + \frac{1}{2}\E_{Y}[|r_D(Y)-1|\bm{1}\{Y \in T_D(\calE')\}]\,.
    \end{align*}
    By taking square on both sides with $(a+b)^2 \leq 2(a^2+b^2)$, we have
    \begin{align*}
        \dTV(Y, T_D(Y))^2 
        & \leq o(1) + \frac{1}{2}(\E_Y[|r_D(Y)-1|\bm{1}\{Y \in T_D(\calE')\}])^2 \\
        & = o(1)+ \frac{1}{2}(\E_Y[|r_D(Y)-1|\bm{1}\{Y \in T_D(\calE')\}])^2 \,.
   \end{align*}
    The second term can be further upper bounded as
    \begin{equation}\label{eq:sphwisandwisnonsmooth}
    \begin{aligned}
        \frac{1}{2}(\E_Y[|r_D(Y)-1|\bm{1}\{Y \in T_D(\calE')\}])^2 
        & \overset{(a)}{\leq} \frac{1}{2}\frac{(\E_Y[|r_D(Y)-1|\bm{1}\{Y \in T_D(\calE')\}])^2}{\E_Y[r_D(Y)/3 + 2/3]} \\
        &\overset{(b)}{\leq} \frac{1}{2}\E_Y\left[\frac{(r_D(Y) - 1)^2}{r_D(Y)/3 + 2/3}\bm{1}\{Y \in T_D(\calE')\}\right]\\
        &\overset{(c)}{\leq} \E_Y[(r_D(Y)\log r_D(Y) - r_D(Y)+1)\bm{1}\{Y \in T_D(\calE')\}]\,,
    \end{aligned}
    \end{equation}
    where $(a)$ holds because $r_D(Y) \geq 0$ and $ \E_Y[r_D(Y)/3 + 2/3] \leq 1$; $(b)$ holds by Cauchy-Schwarz inequality; $(c)$ holds by $(x-1)^2/(x/3 + 2/3) \leq 2(x\log x - x + 1)$ for all $x \geq 0$.

    The next step is to calculate the final term in \eqref{eq:sphwisandwisnonsmooth}, which is a KL-type expectation. The calculation turns out to be rather simple, once we can apply the integration by parts formula (i.e., the divergence theorem). However, this requires a differentiable approximation for the (nonsmooth) indicator function. The function is formally defined as follows; the proof is deferred to later.

    \begin{lemma}\label{lem:sphwisandwissmoothfunc}
        There exists a continuously differentiable function $\phi: \reals^{K} \to [0, 1]$ such that the following holds:
        \begin{enumerate}
            \item[(a)] $\phi(y) = 1$ if $\norm{\overline{y}}_{\op} \leq 0.1$.
            \item[(b)] $\phi(y) = 0$ if $\norm{\overline{y}}_{\op} \geq 0.15$.
            \item[(c)] $\norm{\nabla \phi(y)} = O(1)$ for all $0.1 \leq \norm{\overline{y}}_{\op} \leq 0.15$.
        \end{enumerate}
    \end{lemma}
    Then, by \prettyref{lem:sphwisandwissmoothfunc}, the last term in \eqref{eq:sphwisandwisnonsmooth} can be upper bounded as
    \begin{equation}\label{eq:sphwisandwissmooth}
    \begin{aligned}
        &\E_Y[(r_D(Y)\log r_D(Y) - r_D(Y)+1)\bm{1}\{T_D^{-1}(Y) \in \calE'\}] \\
        &\leq \E_Y[(r_D(Y)\log r_D(Y) - r_D(Y)+1)\phi(T_D^{-1}(Y))]\\
        & \leq \E_Y[(r_D(Y)\log r_D(Y))\phi(T_D^{-1}(Y))] + o(1)\,,
    \end{aligned}   
    \end{equation}
    where the last inequality is from
    \begin{align*}
        \E_Y[\phi(T_D^{-1}(Y))] - \E_Y[\phi(T_D^{-1}(Y))r_D(Y)] &= \E_Y[\phi(T_D^{-1}(Y))] - \E_{Y}[\phi(Y)] \\
        &= \E_Y[1 - \phi(Y)] - \E_{Y}[1 - \phi(T_D^{-1}(Y))] \\
        &\leq\Prob\left(\norm{\overline{Y}}_{\op} \geq 0.1\right)\,,
    \end{align*}
    where the inequality follows from $0 \leq \phi \leq 1$ and $\phi(Y) = 1$ for $\norm{\overline{Y}}_{\op} \leq 0.1$, with the last term being $o(1)$ by Lemma~\ref{lem:sphwishartspec}.
    The last term of \eqref{eq:sphwisandwissmooth} can be expressed as
    \begin{equation}\label{eq:sphwisandwiskl}
    \begin{aligned}
        &\E_Y[(r_D(Y)\log r_D(Y))\phi(T_D^{-1}(Y))] \\
        &= \int_{y \in T_D(2\calE')}g_D(y)\phi(T_D^{-1}(y))\log \frac{g_D(y)}{f(y)} dy \\
        &= \int_{w \in 2\calE'}f(w)\phi(w)\log \frac{f(w)(\det D)^{-(k-1)}}{f(T_D(w))}dw \\
        &= \E_{Y}\left[\left(-\frac{d-k-1}{2}(\log\det(I_k + D\overline{Y}D) - \log \det(I_k + \overline{Y})) - (k-1)\log \det D\right)\phi(Y)\right] \\
        &= \E_{Y}\left[\left(-\frac{d-k-1}{2}\log\det(I_k + A) - (k-1)\log \det D\right)\phi(Y)\right]\,,
    \end{aligned}
    \end{equation}
    where the second equality follows from the change of variables $y = T_D(w)$ with Jacobian $dy = (\det D)^{k-1} dw$ and \eqref{eq:sphwisandwisvariablechange}, 
    the third equality holds by 
    substituting the density form from \eqref{eq:sphwisdensity} and using $\overline{T_D(w)} = D\overline{w}D$,
    and the last equality holds by the identity
    \[
        \det(I_k + A)
        = \det \left((I_k+\overline{Y})^{-1/2}(I_k+D\overline{Y}D)(I_k+\overline{Y})^{-1/2}\right)
        = \frac{\det(I_k+D\overline{Y}D)}{\det(I_k+\overline{Y})}\,,
    \]
    where $A := (I_k + \overline{Y})^{-1/2}(D\overline{Y}D-\overline{Y})(I_k + \overline{Y})^{-1/2}$.
    Next, we approximate the log-determinant term in \eqref{eq:sphwisandwiskl}. We claim that 
    \begin{align}
        |\log \det(I_k + A) - \Tr(A)| \leq  \Tr \left(A^2\right) \,.\label{eq:taylor_trace}
    \end{align}
   By  Taylor expansion, and letting $\{\mu_i\}_{i=1}^k$ denote the eigenvalues of $A$, we have
    \[
    \log \det(I_k + A) = \sum_{i=1}^k \log (1+\mu_i) \,.
    \]
    These eigenvalues are small for $Y \in 2\calE'$, as
    \begin{align*}
        \norm{A}_{\op} &\leq \norm{(I_k + \overline{Y})^{-1/2}}_{\op}^2 \norm{D\overline{Y}D - \overline{Y}}_{\op} \\
        &\leq \norm{(I_k + \overline{Y})^{-1/2}}_{\op}^2 \left(\norm{D-I_k}_{\op}^2 \norm{\overline{Y}}_{\op} + 2\norm{D - I_k}_{\op} \norm{\overline{Y}}_{\op}\right) \\
        &\leq (1-0.2)^{-1}(0.1^2 \times 0.2 + 2 \times 0.1 \times 0.2) \leq 1/2\,,
    \end{align*}
    where the penultimate inequality is from $\norm{\overline{Y}}_{\op} \leq 0.2$ and $D \in \calE$. Then, since $|\log(1+x) - x| \leq x^2$ for $|x| \leq 1/2$,  we have
    \[ \left|\sum_{i=1}^k (\log(1+\mu_i) - \mu_i) \right| \leq \sum_{i=1}^k \mu_i^2\,.
    \]
    Hence, \prettyref{eq:taylor_trace} follows. Applying this to \eqref{eq:sphwisandwiskl}, we obtain
    \begin{equation}\label{eq:sphwisandwisdecomp}
    \begin{aligned}
        &\E_{Y}\left[\left(-\frac{d-k-1}{2}\log\det(I_k + A) - (k-1)\log \det D\right)\phi(Y)\right] \\
        &\leq \underbrace{\E_{Y}\left[\left(-\frac{d-k-1}{2}\Tr(A) -(k-1)\log \det D\right)\phi(Y)\right]}_{\text{(I)}} + \underbrace{\E_{Y}\left[\frac{d-k-1}{2}\Tr(A^2)\phi(Y)\right]}_{\text{(II)}} \,.
    \end{aligned}
    \end{equation}
    For (I) in \eqref{eq:sphwisandwisdecomp}, in the first term we have
    \begin{align}
        -\frac{d-k-1}{2}\E_{Y}[\Tr(A)\phi(Y)] 
        &= -\frac{d-k-1}{2}\E_{Y}[\Tr((I_k+\overline{Y})^{-1}(D\overline{Y}D - \overline{Y})\phi(Y)] \nonumber \\
        &\overset{(a)}{=} -\int_{2\calE'} f(y)\iprod{\frac{\partial \log  f}{\partial y}}{\phi(y)(T_D(y) - y)}_{\reals^K}dy  \nonumber \\
        &= -\int_{2\calE'} \iprod{\frac{\partial f}{\partial y}}{\phi(y)(T_D(y) - y)}_{\reals^K}dy  \nonumber \\
        &\overset{(b)}{=} \int_{2\calE'}\mathrm{div}_y(\phi(y)(T_D(y) - y))f(y)dy \nonumber \\
        &\quad- \int_{\partial (2\calE')} f(y)\iprod{\phi(y)(T_D(y) - y)}{\nu(y)}_{\reals^K}d\Sigma(y)\nonumber \\
        &\overset{(c)}{\leq} \E_{Y}[\phi(Y)]\sum_{i < j \in [k]} (D_{ii}D_{jj} - 1) + o(1) \,. \label{eq:termI_1}
    \end{align}
    Here, $(a)$ follows from
    \begin{align*}
    \left\langle \frac{\partial \log f}{\partial y}(y),\,T_D(y)-y\right\rangle_{\reals^K}
    &= \frac{d-k-1}{2}\left\langle \frac{\partial}{\partial y}\log\det(I_k+\overline Y),\,T_D(y)-y\right\rangle_{\mathbb{R}^K} \\
    &= \frac{d-k-1}{2}\left.\frac{d}{dt}\right|_{t=0}
       \log\det\left(I_k+\overline Y + t(D\overline YD-\overline Y)\right) \\
    &= \frac{d-k-1}{2}\,\Tr\!\left((I_k+\overline Y)^{-1}(D\overline YD-\overline Y)\right)\,, 
    \end{align*}
    $(b)$ is from the divergence theorem for the vector field $v(y) := \phi(y)(T_D(y) - y)$ which vanishes on $\partial(2\calE')$
    with:
    $\mathrm{div}_y v(y) := \sum_{\ell=1}^K \frac{\partial v_\ell}{\partial y_\ell}(y)$,
    $\nu(y)$ being the outward unit normal vector on $\partial(2\calE')$,
    and $d\Sigma(y)$ being the surface measure, i.e.,
    \[
    -\int_{2\calE'} \Big\langle \frac{\partial f}{\partial y}(y), v(y) \Big\rangle_{\mathbb{R}^K} dy
    = \int_{2\calE'} \mathrm{div}_y v(y)\,f(y)\,dy
      - \int_{\partial(2\calE')} f(y)\Big\langle v(y), \nu(y)\Big\rangle_{\mathbb{R}^K} d\Sigma(y)\,,
    \]
    and $(c)$ follows from 
    \begin{align*}
        \mathrm{div}_y(\phi(y)(T_D(y)-y)) &= \sum_{i<j \in [k]} \frac{\partial}{\partial y_{ij}}\left[\phi(y)(T_D(y)-y)_{ij}\right] \\
    &= \sum_{i<j \in [k]} \phi(y) (D_{ii}D_{jj}-1) + \iprod{\nabla \phi(y)}{T_D(y) - y}_{\reals^K}\,,
    \end{align*}
    where $|\E_{Y}[ \iprod{\nabla \phi(Y)}{T_D(Y) - Y}_{\reals^K}]| \leq \sqrt{\E_{Y}[\norm{\nabla \phi(Y)}^2] \E_Y[\norm{T_D(Y) - Y}^2]} \leq O(k/\sqrt{d})\sqrt{\Prob\left(\norm{\overline{Y}}_{\op} \geq 0.1\right)} = o(1)$ by Lemma~\ref{lem:sphwishartspec}.
   Combining \prettyref{eq:termI_1} with $(k-1)\log \det D = \sum_{i < j \in [k]}\log(D_{ii}D_{jj})$, we have 
    \begin{align*}
        \text{(I)} 
        &\leq \E[\phi(Y)] \sum_{i < j \in [k]}(D_{ii}D_{jj} - 1 - \log(D_{ii}D_{jj})) + o(1) \\
        &\leq \sum_{i < j \in [k]}(D_{ii}D_{jj}-1)^2 + o(1)\,,
    \end{align*}
    where the second inequality holds by $D_{ii}D_{jj} \in [0.9^2, 1.1^2]$ for all $i < j \in [k]$ when $D \in \calE$ and $0 \leq x - 1 - \log x \leq (x-1)^2$ for all $x \geq 1/2$.

    For (II) in \eqref{eq:sphwisandwisdecomp}, we have
    \begin{align*}
        &\frac{d-k-1}{2}\E_{Y}[\Tr(A^2)\phi(Y)] \\
        &= \frac{d-k-1}{2}\E_{Y}[\Tr((I_k+\overline{Y})^{-1}(D\overline{Y}D - \overline{Y})(I_k+\overline{Y})^{-1}(D\overline{Y}D - \overline{Y}))\phi(Y)] \\
        &\leq \frac{d-k-1}{2}\E_{Y}\left[\norm{(I_k+\overline{Y})^{-1}(D\overline{Y}D - \overline{Y})}_{\mathrm{F}}^2\phi(Y)\right] \\
        &\leq  \frac{d-k-1}{2}\E_{Y}\left[\norm{(I_k+\overline{Y})^{-1}}_{\op}^2\norm{D\overline{Y}D - \overline{Y}}_{\mathrm{F}}^2\phi(Y)\right] \\
        &\leq 2\sum_{i < j \in [k]} (D_{ii}D_{jj}-1)^2\,,
    \end{align*}
    where the last inequality is from $\norm{(I_k + \overline{Y})^{-1}}_{\op} \leq (1-0.15)^{-1}$ for $Y \in \mathrm{supp}({\phi})$ and $\E_Y[Y_{ij}^2] = 1/d$.
    Summing up our calculations on (I) and (II), for \eqref{eq:sphwisandwisdecomp} we have
    \begin{equation}\label{eq:sphwisandwisklfinal}
        \text{(I)} + \text{(II)} \leq 3\sum_{i<j\in[k]}(D_{ii}D_{jj}-1)^2 + o(1)\,.
    \end{equation}
    Now by tracing back from \eqref{eq:sphwisandwisklfinal} to \eqref{eq:sphwisandwis2}, we have
    \begin{align*}
        \dTV(Y, W) &\leq o(1) + \sqrt{o(1) + 3\E_{D}\left[\sum_{i<j\in[k]}(D_{ii}D_{jj}-1)^2\bm{1}\{D \in \calE\}\right]} \\
        &\leq o(1) + \sqrt{o(1) + 3k^2\E_{D}\left[(D_{11}D_{22}-1)^2\right]} \\
        &= o(1) + \sqrt{o(1) + 6k^2\left(1 - \left(\E\left[\frac{\norm{Z_1}}{\sqrt{d}}\right]\right)^2\right)} \\
        &\leq o(1) + \sqrt{o(1) + O(k^2/d)}\,,
    \end{align*}
    where the equality follows from the independence of $D_{11}$ and $D_{22}$ and $\E[D_{ii}^2]=1$, which implies $\E[(D_{11}D_{22}-1)^2] = \E[D_{11}^2]\E[D_{22}^2] - 2(\E [D_{11}])^2 + 1 = 2(1 - (\E [D_{11}])^2)$; the last inequality is from $\E\left[\norm{Z_1}/\sqrt{d}\right] = \sqrt{2}\Gamma((d+1)/2)/(\Gamma(d/2)\sqrt{d}) = 1-1/(4d) + O(1/d^2)$ \cite[Equation 6.1.47]{abramowitz1965handbook}. The last term is $o(1)$, as desired.

We complete the proof of Proposition~\ref{prop:sphwisandwis} by verifying the technical lemma that was used earlier for approximating the indicator function (Lemma~\ref{lem:sphwisandwissmoothfunc}).
\begin{proof}[Proof of Lemma~\ref{lem:sphwisandwissmoothfunc}]
    First, it is easy to see that there exists a univariate function with such properties. Namely, there exists a continuously differentiable function $g: \reals \to [0, 1]$ such that $g(x) = 1$ for $|x| \leq 0.11$, $g(x) = 0$ for $|x| \geq 0.15$, and $|g'(x)| \leq C_0$ for some constant $C_0 > 0$.

    Thus, for our purpose, it suffices to find an approximation $h(y)$ of the operator norm $\norm{\overline{y}}_{\op}$ and take the composition of the two functions, i.e.,
    \[\phi := g \circ h\,.
    \]
    We use the standard log-sum-exp function, which is smooth:
    \[h(y) := \frac{1}{\beta}\log(\Tr(\exp(\beta \overline{y})) + \Tr(\exp(-\beta \overline{y})))\,,
    \]
    where $\beta := 100\log(2k)$. As a function of eigenvalues, one can check that $\norm{\overline{y}}_{\op} \leq h(y) \leq \norm{\overline{y}}_{\op} + 0.01$. Furthermore, from
    \[\frac{\partial}{\partial y_{ij}}h(y) = 2\left(\frac{\exp(\beta \overline{y}) - \exp(-\beta \overline{y})}{\Tr(\exp(\beta \overline{y})) + \Tr(\exp(-\beta \overline{y}))}\right){_{ij}}, i < j \in [k]\,,
    \]
    we have
    \[\norm{\nabla h(y)} \leq 2\norm{\frac{\exp(\beta \overline{y})}{\Tr(\exp(\beta \overline{y})) + \Tr(\exp(-\beta \overline{y}))}}_{F} + 2\norm{\frac{\exp(-\beta \overline{y})}{\Tr(\exp(\beta \overline{y})) + \Tr(\exp(-\beta \overline{y}))}}_{F} \leq 4\,,
    \]
    because each matrix $\exp(\pm \beta \overline{y})/(\Tr(\exp(\beta \overline{y})) + \Tr(\exp(-\beta \overline{y})))$ is positive semidefinite with trace at most 1.
    To check the properties (a)--(c) in Lemma~\ref{lem:sphwisandwissmoothfunc}: (a) and (b) are straightforward from the definition, and (c) follows from $\norm{\nabla \phi(y)} = \norm{\nabla h(y)}|g'(h(y))| \leq 4C_0$.
\end{proof}

\subsection{Decomposition of likelihood ratio (\prettyref{lem:lbdensef1throughf3})}\label{appendix:lbdensef1throughf3}

The statement is in two parts; the first is to identify the functions $f_i(X)$ for $1 \leq i\leq 3$, and the second is to show their moment bounds as a function of $N$ and $d$. 

\paragraph*{Characterizing $f_i(X)$.}
From the formula for densities $w_{N, d}$ and $m_{N, d}$ (see, e.g., \cite[Section 2]{raczrichey}) we have 
\[
w^2_{N, d}(X) / m^2_{N, d}(X) = \exp(2\alpha_{N, d}(X))\,,
\]
where for eigenvalues $\lambda_1, \dots, \lambda_N$ of $X$,
\begin{align*}
\alpha_{N,d}(X) &:= \frac{1}{2} \sum_{i=1}^{N}
\left((d - N - 1)\log\lambda_i - \lambda_i + \frac{1}{2d}(\lambda_i - d)^2\right) \\
&\quad + \left(\frac{N(N+3)}{4} - \frac{dN}{2}\right)\log 2 + \frac{N}{2}\log \pi  + \frac{N(N+1)}{4}\log d - \sum_{i=1}^N\log \Gamma\left(\frac{d+1-i}{2}\right)\,.
\end{align*}
From Stirling's formula $\log \Gamma(z) = (z-1/2)\log z - z + (1/2)\log(2\pi) + O(1/z)$, the last term can be approximated as
\begin{equation}\label{eq:lbdensewishartgoeexpansion1}
\begin{aligned}
\alpha_{N,d}(X) &= \frac{1}{2} \sum_{i=1}^{N}
\left((d - N - 1)\log\lambda_i - \lambda_i + \frac{1}{2d}(\lambda_i - d)^2\right) \\
&\quad + \frac{N(N+1)}{4} \log d 
       - \frac{1}{2} \sum_{i=1}^{N}(d - i)\log(d + 1 - i)  + \frac{1}{2} \sum_{i=1}^{N}(d + 1 - i) 
       + O\left(\frac{N}{d}\right)\,.
\end{aligned}
\end{equation}
Since $N \leq s \leq \lceil 1.1k \rceil$, the $O(N/d)$ term is of order at most $O(k/d) = o(1)$. From
\[\log(d+1-i) = \log d + \log\left(1-\frac{i-1}{d}\right) = \log d - \frac{i-1}{d} - \frac{(i-1)^2}{2d^2} + O\left(\frac{(i-1)^3}{d^3}\right)
\]
and $0 \leq \sum_{i=1}^N (d-i) (i-1)^3/d^3 \leq \sum_{i=1}^N (i-1)^3/d^2 = O(k^4/d^2) = o(1)$, the expansion \eqref{eq:lbdensewishartgoeexpansion1} can be concisely written as
\begin{equation}\label{eq:lbdensewishartgoeexpansion2}
    \alpha_{N, d}(X) = \sum_{i=1}^N t(\lambda_i) - \frac{N^3}{12d} + o(1)\,,
\end{equation}
where
\[t(x) := \frac{1}{2}\left((d-N -1)\log(x/d)-(x-d)+\frac{1}{2d}(x-d)^2\right)\,.
\]
Now consider the third-order Taylor expansion of $t$ at $x = d$. Then
\begin{align*}
    t(x) &= -\frac{N+1}{2d}(x - d) + \frac{N+1}{4d^2}(x - d)^2 
+ \frac{d - N - 1}{6d^3}(x - d)^3 - \frac{d - N - 1}{8\xi^4}(x - d)^4 \\
&\leq -\frac{N+1}{2d}(x - d) + \frac{N+1}{4d^2}(x - d)^2 
+ \frac{d - N - 1}{6d^3}(x - d)^3\,,
\end{align*}
where $\xi$ is between $x$ and $d$, and $d \gg k^2$ implies $d > N+1$.

From now on, it is more convenient to consider the affine transformation $Y$ of $X$, defined as
\[Y := (X-dI_N)/\sqrt{d}\,.
\]
Then for $X \sim m_{N, d}$, $Y_{ii} \sim \calN(0, 2) , i \in [N]$ and $Y_{ij} = Y_{ji} \sim \calN(0, 1), i < j \in [N]$ independently. Furthermore, $X \in \calE_N \Leftrightarrow Y \in \calF_N$, where
\[\calF_N := \left\{Y \in \reals^{N \times N}: \norm{Y_{Q \times Q}}_{\op} \leq 10\sqrt{|Q|}(1+\sqrt{\log s}) \text{ for all } \emptyset \subsetneq Q \subseteq [N]\right\}\,.
\]
Plugging in the upper bound on $t$ to \eqref{eq:lbdensewishartgoeexpansion2} yields
\begin{equation}\label{eq:lbdensewishartgoeexpansion3}
\begin{aligned}
    \alpha_{N, d}(X) &\leq -\frac{N+1}{2d^{1/2}}  \sum_{i=1}^{N} \mu_i + \frac{N+1}{4d} \sum_{i=1}^{N} \mu_i^2  + \frac{d - N - 1}{6d^{3/2}}  \sum_{i=1}^{N} \mu_i^3 + o(1)\,, \\
    &= -\frac{N+1}{2d^{1/2}}\Tr(Y) + \frac{N+1}{4d} \Tr(Y^2) + \frac{d - N - 1}{6d^{3/2}}\Tr(Y^3) + o(1)\,,
\end{aligned}
\end{equation}
where $\mu_1, \dots, \mu_N$ are eigenvalues of $Y$. From \eqref{eq:lbdensewishartgoeexpansion3} for \eqref{eq:lbdensef1throughf3} we have
\begin{align*}
    &\exp(2\alpha_{N, d}(X))\bm{1}\{X \in \calE_N\} \\
    &= \exp(2\alpha_{N, d}(X))\bm{1}\{Y \in \calF_N\} \\
    &\leq (1+o(1))\underbrace{\exp\left(-\frac{N+1}{d^{1/2}}\Tr(Y)\right)}_{=:f_1(X)}\underbrace{\exp\left(\frac{N+1}{2d} \Tr(Y^2)\right)}_{=:f_2(X)} \underbrace{\exp\left(\frac{d - N - 1}{3d^{3/2}}\Tr(Y^3)\right)\bm{1}\{Y \in \calF_N\}}_{=:f_3(X)}\,.
\end{align*}
Here, we identified $f_i(X)$ for $1 \leq i \leq 3$ as exponentiated traces of powers of $Y$.

\paragraph*{Calculating expectations.} Given their definitions, we show that $f_1(X)$ through $f_3(X)$ has expectation (over $X \sim m_{N, d}$) at most $\exp(O(N^3/d))$. Crucial to those calculations is the concentration of traces of GOE (Lemma~\ref{lem:goecubetraceconc}).

\begin{itemize}
    \item For $f_1(X)$,
    we have $\Tr(Y) \sim \calN(0, 2N)$. Thus from the moment generating function of normal distribution, 
    \[\E_{X \sim m_{N, d}}[f_1(X)^3] = \exp\left(\frac{9(N+1)^2}{2d} \times 2N\right) = \exp\left(O\left(\frac{N^3}{d}\right)\right)\,.
    \]
    \item For $f_2(X)$, 
    $\Tr(Y^2)$ concentrates around $N(N+1)$ with subexponential tail (with norm $\norm{\cdot}_{\psi_1} = O(N)$)  
    by Lemma~\ref{lem:goecubetraceconc}. By standard conversion from subexponential concentration to moment generating function \cite[Proposition 2.7.1]{hdpbook},
    \[\E_{X \sim m_{N, d}}[f_2(X)^3] \leq \exp\left(\frac{3N(N+1)^2}{2d} + \left(\frac{3(N+1)}{2d}\right)^2 \times O(N^2)\right) = \exp\left(O\left(\frac{N^3}{d}\right)\right)\,.
    \]
    Note that here we used $N^4/d^2 = O(k^4/d^2) = o(1)$ and $N^4/d^2 \leq (N^3/d) \times O(k/d) = o(N^3/d)$.
    \item For $f_3(X)$,  
    first from $\exp(x) \leq 1 + x + (x^2/2)\exp(|x|)$ we have
        \begin{equation}\label{eq:lbdensef3}
    \begin{aligned}
        &\E_{X \sim m_{N, d}} \left[ f_3(X)^3 \right] \\
        &= \E_{X \sim m_{N, d}} \left[ \exp\left(\frac{d - N - 1}{d^{3/2}}\Tr(Y^3)\right) \bm{1}\{Y \in \calF_N\} \right] \\
        &\leq 1 + \frac{1}{2}\E_{X \sim m_{N, d}} \left[ \left(\frac{(d-N-1)\Tr(Y^3)}{d^{3/2}}\right)^2 \exp\left(\frac{(d-N-1)|\Tr(Y^3)|}{d^{3/2}}\right)\bm{1}\{Y \in \calF_N\} \right] \\
        &\leq 1 + \frac{1}{2d}\sqrt{\E_{X \sim m_{N, d}} \left[ \Tr(Y^3)^4 \right]} \sqrt{\E_{X \sim m_{N, d}} \left[ \exp\left(\frac{2|\Tr(Y^3)|}{d^{1/2}}\right)\bm{1}\{Y \in \calF_N\} \right]} \\
        &\leq 1 + O\left(\frac{N^3}{d}\right) \times \sqrt{\E_{X \sim m_{N, d}} \left[ \exp\left(\frac{2|\Tr(Y^3)|}{d^{1/2}}\right)\bm{1}\{Y \in \calF_N\} \right]}\,.
    \end{aligned}
\end{equation}

    Here, for the second inequality we used that $\Tr(Y^3)\bm{1}\{Y \in \calF_N\}$ has mean 0, as each $Y_{ij}$ is symmetric around 0 (hence also after symmetric truncation $Y \in \calF_N$); the third inequality is by Cauchy-Schwarz, with $0 \leq (d-N-1)/d^{3/2} \leq 1/d^{1/2}$; the final inequality is from Gaussian hypercontractivity $\sqrt{\E[(\Tr(Y^3))^4]} \lesssim \E[(\Tr(Y^3))^2] = O(N^3)$. 

    Note that for $Y \in \calF_N$, we have
    \[|\Tr(Y^3)| \leq N\norm{Y}_{\op}^3 \leq N \times (10\sqrt{N}(1+\sqrt{\log s}))^3 \leq  C_0N^{5/2}\log^{3/2}s\,,
    \]
    where the first inequality follows from $|\sum_{i=1}^N \mu_i^3| \leq N \max_{1 \leq i \leq N} |\mu_i|^3$, and the last inequality holds for some constant $C_0 > 0$.
    Thus,
    \begin{equation}\label{eq:lbdensef3terms1}
    \begin{aligned}
        &\E_{X \sim m_{N, d}} 
        \left[\exp((2/d^{1/2})|\Tr(Y^3)|)\bm{1}\{Y \in \calF\} \right]\\
        &\leq \int_{0}^{\exp(2C_0N^{5/2}\log^{3/2} s/d^{1/2})} \Prob_{X \sim m_{N, d}}(\exp((2/d^{1/2})|\Tr(Y^3)|) \geq u)du \\
        &\leq 1 + 2\sqrt{\frac{N^3}{d}} \int_{0}^{C_0N\log^{3/2} s} \Prob_{X \sim m_{N, d}}(|\Tr(Y^3)| \geq N^{3/2}u)\exp\left(2u\sqrt{\frac{N^3}{d}}\right)du \\
        &\leq 1 + 4\sqrt{\frac{N^3}{d}} \int_0^{N^{3/4}} \exp\left(-\frac{1}{C_{\ref{lem:goecubetraceconc}}}{u^2} + 2u\sqrt{\frac{N^3}{d}}\right)du \\
        &\quad + 4\sqrt{\frac{N^3}{d}}\int_{N^{3/4}}^{C_0N\log^{3/2} s}\exp\left(-\frac{Nu^{2/3}}{C_{\ref{lem:goecubetraceconc}}} + 2u\sqrt{\frac{N^3}{d}}\right)du \\
        &\leq 1 + O\left(\sqrt{\frac{N^3}{d}}\right)\exp\left(O\left(\frac{N^3}{d}\right)\right) + 4\sqrt{\frac{N^3}{d}}\int_{N^{3/4}}^{C_0N\log^{3/2} s}\exp\left(-\frac{Nu^{2/3}}{C_{\ref{lem:goecubetraceconc}}} + 2u\sqrt{\frac{N^3}{d}}\right)du\,,
    \end{aligned}
    \end{equation}
    where the second inequality follows from the change of variables $v = \exp((2N^{3/2}/d^{1/2})u)$, and the third inequality holds by Corollary~\ref{lem:goecubetraceconc} (splitting the integral at $u=N^{3/4}$, where the tail behavior changes).
    Here, note that the function $u \mapsto \exp(-Nu^{2/3}/C_{\ref{lem:goecubetraceconc}} + 2u\sqrt{N^3/d})$ for $u \geq 0$ is decreasing for $u \leq  (\sqrt{d^3/N^3})/(27C_{\ref{lem:goecubetraceconc}}^3)$ and increasing afterwards. Hence, the maximum of this function over $u \in [N^{3/4}, C_0N\log^{3/2} s]$ is obtained at the boundary which has value 
    \begin{align*}
        &\exp(-N^{3/2}/C_{\ref{lem:goecubetraceconc}} + 2N^{9/4}/\sqrt{d}) \vee \exp(-(C_0^{2/3}/C_{\ref{lem:goecubetraceconc}})N^{5/3} \log s + 2C_0N^{5/2} \log^{3/2} s/\sqrt{d}) \\
        &= \exp\left(-\frac{N^{3/2}}{C_{\ref{lem:goecubetraceconc}}}\left(1 - O\left(\frac{N^{3/4}}{\sqrt{d}}\right)\right)\right) \vee \exp\left(-\left(\frac{C_0^{2/3}}{C_{\ref{lem:goecubetraceconc}}}\right)N^{5/3} \log s\left(1 - O\left(\frac{N^{5/6}\log^{1/2}s}{\sqrt{d}
        }\right)\right)\right)\\
        & \le \exp\left(-\frac{N^{3/2}}{2C_{\ref{lem:goecubetraceconc}}}\right) \vee \exp\left(-\frac{C_0^{2/3}}{2C_{\ref{lem:goecubetraceconc}}}N^{5/3}\log s\right)\,,
    \end{align*}
    where the inequality holds for all sufficiently large $k$, as $N = O(k)$ implies $N^{3/4}/\sqrt{d} = o(1)$ and $N^{5/6}\log^{1/2} s/\sqrt{d} = o(1)$ for $d \gg k^2$. 
    Then, it follows 
    \begin{equation}\label{eq:lbdensef3terms2}
    \begin{aligned}
        &\int_{N^{3/4}}^{C_0N\log^{3/2} s}\exp\left(-Nu^{2/3} + 2u\sqrt{\frac{N^3}{d}}\right)du \\
        &\leq C_0N\log^{3/2}s \left(\exp\left(-\frac{N^{3/2}}{2C_{\ref{lem:goecubetraceconc}}}\right) \vee \exp\left(-\frac{C_0^{2/3}}{2C_{\ref{lem:goecubetraceconc}}}N^{5/3}\log s\right)\right) = O\left(\frac{\log^{3/2}s}{N}\right)\,.
    \end{aligned}
    \end{equation}
    By  \eqref{eq:lbdensef3}, \eqref{eq:lbdensef3terms1} and \eqref{eq:lbdensef3terms2}, we obtain
    \begin{align*}
        \E_{X \sim m_{N, d}}[f_3(X)^3] 
        &\leq 1 + O\left(\frac{N^3}{d}\right)\left(1 + O\left(\sqrt{\frac{N^3}{d}}\right)\exp\left(O\left(\frac{N^3}{d}\right)\right) + O\left(\sqrt{\frac{N}{d}}\log^{3/2}k\right)\right) \\
        &\leq 1 + O\left(\frac{N^3}{d}\right)\left(1 + o(1) +  O\left(\sqrt{\frac{N^3}{d}}\right)\exp\left(O\left(\frac{N^3}{d}\right)\right)\right) \\
        &\leq \exp\left(O\left(\frac{N^3}{d}\right)\right)\,,
    \end{align*}
    where the second inequality is from $O(\sqrt{N/d}\log^{3/2}k) = o(1)$ given that $N = O(k)$ and $ d \gg k^2$, and the last inequality is from $1+x(2+\sqrt{x}e^{x}) \leq 1+2x+xe^{2x} \leq 1 + 2x + e^{2x}(e^x-1) \leq e^{3x}$ for all $x \geq 0$.
\end{itemize}

\subsection{Upper bound on cubed hypergeometric (\prettyref{lem:hypergeomcubemgf})}\label{appendix:hypergeomcubemgf}
    Let $C_0> 0$ be arbitrary and consider any fixed $s \in [k^-, k^+]$. Then the conditional distribution of $N = |S \cap S'|$ given $|S| = |S'| = s$ is $\mathrm{Hypergeom}(n, s, s)$,
    which is stochastically dominated by $M \sim \Binom(s, s/(n-s))$. Then
    \begin{align*}
        \E_{N \big||S| = |S'| = s} [\exp(C_0N^3/d)] &\leq \E [\exp(C_0M^3/d)] \\
        &\leq \exp\left(64C_0\frac{s^6}{n^3d}\right) + \int_{64s^6/n^3}^{s^3} \Prob(M \geq u^{1/3})\frac{C_0}{d}\exp\left(\frac{C_0u}{d}\right)du \\
        &\leq \exp\left(64C_0\frac{s^6}{n^3d}\right) + \int_{64s^6/n^3}^{s^3} \Prob(M - \E [M]\geq (1/2)u^{1/3})\frac{C_0}{d}\exp\left(\frac{C_0u}{d}\right)du \\
        &\leq \exp\left(64C_0\frac{s^6}{n^3d}\right) + \frac{C_0}{d}\int_{64s^6/n^3}^{s^3}\exp(-u^{1/3}/100 + C_0u/d)du\,,
    \end{align*}
    where the second and third inequalities are from 
    $\E [M] = s^2/(n-s) \leq 2s^2/n \leq (1/2)u^{1/3}$ for $u \in [64s^6/n^3, s^3]$ and Chernoff bound. For all such $u$, since $C_0u/d \leq C_0u^{1/3}{(k^+)^2/d}$ where $(k^+)^2/d = o(1)$, there exists a constant $C_1 > 0$ that only depends on $C_0$ such that 
    for all sufficiently large $k$,
    \[\exp(-u^{1/3}/100 + C_0u/d) \leq \exp(-u^{1/3}/C_1)\,.
    \]
    Thus we have
    \begin{align*}
        &\E_{N \big| |S| = |S'| = s} [\exp(C_0N^3/d)] \\
        &\leq \exp\left(64C_0\frac{s^6}{n^3d}\right) + \frac{C_0}{d}\int_{64s^6/n^3}^{s^3}\exp(-u^{1/3}/C_1)du \\
        &\leq \exp\left(64C_0\frac{s^6}{n^3d}\right) + \frac{3C_0}{d}(C_1(64s^6/n^3)^{2/3} + 2C_1^2(64s^6/n^3)^{1/3} + 2C_1^3)\exp(-4s^2/(nC_1)) \\
        &\leq \exp\left(64C_0\frac{(k^+)^6}{n^3d}\right) + 3C_0\left(C_1\left(\frac{64(k^+)^6}{n^3d}\right)^{2/3} + 2C_1^2 \left(\frac{64(k^+)^6}{n^3d}\right)^{1/3} +  \frac{2C_1^3}{d} \right) \\
        &= 1 + o(1)\,,
    \end{align*}
    as desired. Here, the second inequality is from $\int_b^{\infty} \exp(-u^{1/3}/a)du = 3(ab^{2/3} + 2a^2b^{1/3} + 2a^3)\exp(-b^{1/3}/a)$, and the third inequality is from $(1/d)x^{2/3}\leq (x/d)^{2/3}, (1/d)x^{1/3} \leq (x/d)^{1/3}, \exp(-4s^2/(nC_1))\leq 1$ with $s \leq k^+ = \lceil 1.1k \rceil$. The final line  follows from the assumption $d \gg k^2 \vee k^6/n^3$.

\subsection{Typical behavior of neighborhood distributions (Lemma~\ref{lem:lbgeneralhighprobevent})}\label{appendix:lbgeneralhighprobevent} 
Recall that $U = (U_1, \dots, U_{n-1}), V = (V_1, \dots, V_{n-1})$. 
The probability $\Prob((\Gamma, U, V) \in \calE^c)$ (regardless of $\Gamma \sim \calP_n (\cdot | X_{[n-1]})$ or $\calQ_n$) is upper bounded by
\begin{align*}
    &\underbrace{\Prob_{X_{[n-1]}, \Gamma}\left(\sum_{i=1}^{n-1}V_i > 2k\right)}_{\text{(I)}} + \underbrace{\Prob_{X_{[n-1]}, \Gamma}\left(\sum_{i \in [n-1]: V_i = 1} \Gamma_i> 4kp, \sum_{i=1}^{n-1}V_i \leq 2k\right)}_{\text{(II)}} \\
    &\quad+ \underbrace{\Prob_{X_{[n-1]}, \Gamma}\left(|\Delta(\Gamma, U, V)| > \frac{k}{n}, \sum_{i=1}^{n-1}V_i \leq 2k, \sum_{i \in [n-1]: V_i = 1} \Gamma_i \leq 4kp\right)}_{\text{(III)}}\,.
\end{align*}
We show that each term is $o(1/(n^2\log(1/p)))$; the following analysis for the first two terms (I) and (II) apply for both $\Gamma \sim \calP_n(\cdot | X_{[n-1]})$ and $\Gamma \sim \calQ_n$.

Since 
$\sum_{i=1}^{n-1}V_i \sim \Binom(n-1, k/n)$,
(I) is at most $\exp(-k/3)$ by standard Chernoff bound. This is $o(1/(n^2\log(1/p)))$ as long as $k \gtrsim \log n$.

For (II), first note that for any 
$V$,
$\Gamma$ follows $\mathrm{Ber}(p)^{\otimes (n-1)}$.
Thus %
$\sum_{i \in [n-1]: V_i = 1} \Gamma_i \sim \Binom(\sum_{i=1}^{n-1}V_i, p)$ conditioned on $V$,
which is stochastically dominated by $\Binom(2k, p)$ under the event $\{\sum_{i=1}^{n-1}V_i \leq 2k\}$. Thus, (II) is at most $\exp(-2kp/3)$ again by Chernoff bound; this is $o(1/(n^2\log(1/p)))$ as long as $kp \gtrsim \log n$.

For (III), first consider the case where $\Gamma \sim \calQ_n$. Let $\nu$ be the number of $i \in [n-1]$ with $V_i = 1$ (i.e., vertices included in the community), and $i_1 < \dots < i_{\nu}$ be such indices. Then by Lemma~\ref{lem:capsandanticaps}, the probability conditioned on $\Gamma$ and $V$ is at most
\[\exp\left(-\frac{1}{C_{\ref{lem:capsandanticaps}}}\frac{d}{M \log(1/p)\log(d/p)} + C_{\ref{lem:capsandanticaps}} \log (3k)\right)\,,  
\]
where $M = M(\nu, p, (\Gamma_{i_1}, \dots, \Gamma_{i_\nu}))= 
((\sum_{j=1}^\nu \Gamma_{i_j})\log(1/p) + \nu p + \log(d/p))(\sum_{j=1}^\nu \Gamma_{i_j} + \nu p) \leq C_0k^2p^2\log(1/p)\log(d/p)$, for some constant $C_0 > 0$ from 
$\nu \leq 2k, \sum_{j=1}^\nu  \Gamma_{i_j} \leq 4kp$ and $kp \geq 1$. Thus, (III) is at most
\begin{equation}\label{eq:lbgeneralconcunderQ}
\exp\left(-\frac{1}{C_{\ref{lem:capsandanticaps}}}\frac{d}{C_0(kp\log(1/p)\log(d/p))^2} + C_{\ref{lem:capsandanticaps}} \log(3k)\right)\,,
\end{equation}
which is $o(1/(n^2\log(1/p)))$ as long as $d \gtrsim (kp\log(1/p)\log(d/p))^2 \log n$.

Now we consider (III) with $\Gamma \sim \calP_n(\cdot | X_{[n-1]})$.
First, (III) conditioned on $V$ (with $\nu \leq 2k$) is equal to

\begin{equation}\label{eq:lbgeneralterm3wts}
\begin{aligned}
    &\Prob_{U_{i_1}, \dots, U_{i_\nu}, (\Gamma_{i_1}, \dots, \Gamma_{i_\nu}) \sim \widetilde{\calP}_{\nu+1}(\cdot | U_{i_1}, \dots, U_{i_\nu})}\left(|\Delta(\Gamma, U, V)| > \frac{k}{n}, \sum_{j = 1}^{\nu} \Gamma_{i_j} \leq 4kp\right) \\
    &= \E_{U_{i_1}, \dots, U_{i_\nu}, (\Gamma_{i_1}, \dots, \Gamma_{i_\nu}) \sim \calQ_{\nu+1}}\left[\bm{1}\left\{|\Delta(\Gamma, U, V)| > \frac{k}{n}, \sum_{j = 1}^{\nu} \Gamma_{i_j} \leq 4kp\right\} \frac{\widetilde{\calP}_{\nu+1}(\Gamma_{i_1}, \dots, \Gamma_{i_{\nu}}|U_{i_1}, \dots, U_{i_{\nu}})}{p^{\sum_{j=1}^\nu \Gamma_{i_j}}(1-p)^{\nu - \sum_{j=1}^\nu \Gamma_{i_j}}}\right] \\
    &\leq \left(\E_{U_{i_1}, \dots, U_{i_\nu}, (\Gamma_{i_1}, \dots, \Gamma_{i_\nu}) \sim \calQ_{\nu + 1}}\left[\bm{1}\left\{|\Delta(\Gamma, U, V)| > \frac{k}{n}, \sum_{j = 1}^{\nu} \Gamma_{i_j} \leq 4kp\right\}\right]\right)^{1/2}
     \\
     &\quad \times\left(\E_{U_{i_1}, \dots, U_{i_\nu}, (\Gamma_{i_1}, \dots, \Gamma_{i_\nu}) \sim \calQ_{\nu + 1}}\left[\frac{\widetilde{\calP}_{\nu+1}^2(\Gamma_{i_1}, \dots, \Gamma_{i_{\nu}}|U_{i_1}, \dots, U_{i_{\nu}})}{p^{2\sum_{j=1}^\nu \Gamma_{i_j}}(1-p)^{2(\nu - \sum_{j=1}^\nu \Gamma_{i_j})}}\bm{1}\left\{\sum_{j = 1}^{\nu} \Gamma_{i_j} \leq 4kp\right\}\right]\right)^{1/2} \\
    &\leq \exp\left(-\frac{1}{2C_{\ref{lem:capsandanticaps}}}\frac{d}{C_0(kp\log(1/p)\log(d/p))^2} + \frac{C_{\ref{lem:capsandanticaps}}}{2} \log(3k)\right) \\
    &\quad \times \left(\E_{U_{i_1}, \dots, U_{i_\nu}, (\Gamma_{i_1}, \dots, \Gamma_{i_\nu}) \sim \calQ_{\nu + 1}}\left[\frac{\widetilde{\calP}_{\nu+1}^2(\Gamma_{i_1}, \dots, \Gamma_{i_{\nu}}|U_{i_1}, \dots, U_{i_{\nu}})}{p^{2\sum_{j=1}^\nu \Gamma_{i_j}}(1-p)^{2(\nu - \sum_{j=1}^\nu \Gamma_{i_j})}}\bm{1}\left\{\sum_{j = 1}^{\nu} \Gamma_{i_j} \leq 4kp\right\}\right]\right)^{1/2}\,,
\end{aligned}
\end{equation}
where the first inequality is from Cauchy-Schwarz, and the second inequality is from~\eqref{eq:lbgeneralconcunderQ} with $kp \geq 1$.

As the notation $i_1, \dots, i_{\nu}$ reads complicated, without loss of generality we change those to simpler indices. Note that 
\[\frac{\widetilde{\calP}_{\nu+1}^2(\Gamma_{i_1}, \dots, \Gamma_{i_{\nu}}|U_{i_1}, \dots, U_{i_{\nu}})}{p^{2\sum_{j=1}^\nu \Gamma_{i_j}}(1-p)^{2(\nu - \sum_{j=1}^\nu \Gamma_{i_j})}}\bm{1}\left\{\sum_{j = 1}^{\nu} \Gamma_{i_j} \leq 4kp\right\}
\]
over $U_{i_1}, \dots, U_{i_\nu}, (\Gamma_{i_1}, \dots, \Gamma_{i_{\nu}}) \sim \calQ_{\nu+1}$ (with $i_1 < \dots < i_{\nu}$ being the indices in $[n-1]$ with $V_i = 1$) has the same distribution as 
\[\frac{\widetilde{\calP}_{\nu+1}^2(\Gamma_{[\nu]} | U_{[\nu]})}{\calQ_{\nu+1}^2(\Gamma_{[\nu]})}\bm{1}\left\{\sum_{i=1}^\nu \Gamma_i \leq 4kp\right\}
\]
over $U_{[\nu]}, \Gamma_{[\nu]} \sim \calQ_{\nu + 1}$ (with $[\nu]$ being the indices in $[n-1]$ with $V_i = 1$) due to symmetry; recall that $\Gamma_{[\nu]} = (\Gamma_1, \dots, \Gamma_\nu)$ and $U_{[\nu]} = (U_1, \dots, U_\nu)$. Hence for the second factor in the last term of~\eqref{eq:lbgeneralterm3wts}, with the corresponding change in $V$ we have
\begin{align*}
    &\E_{U_{i_1}, \dots, U_{i_\nu}, (\Gamma_{i_1}, \dots, \Gamma_{i_\nu}) \sim \calQ_{\nu + 1}}\left[\frac{\widetilde{\calP}_{\nu+1}^2(\Gamma_{i_1}, \dots, \Gamma_{i_{\nu}}|U_{i_1}, \dots, U_{i_{\nu}})}{p^{2\sum_{j=1}^\nu \Gamma_{i_j}}(1-p)^{2(\nu - \sum_{j=1}^\nu \Gamma_{i_j})}}\bm{1}\left\{\sum_{j = 1}^{\nu} \Gamma_{i_j} \leq 4kp\right\}\right] \\
    &= \E_{U_{[\nu]}, \Gamma_{[\nu]} \sim \calQ_{\nu+1}}\left[\frac{\widetilde{\calP}_{\nu+1}^2(\Gamma_{[\nu]} | U_{[\nu]})}{\calQ_{\nu+1}^2(\Gamma_{[\nu]})}\bm{1}\left\{\sum_{i=1}^\nu \Gamma_i \leq 4kp\right\}\right]\,.
\end{align*}
Now the main technical portion lies in controlling the term
\begin{equation}\label{eq:lbgeneralsecondmoment}
    \E_{U_{[\nu]}, \Gamma_{[\nu]} \sim \calQ_{\nu+1}}\left[\frac{\widetilde{\calP}_{\nu+1}^2(\Gamma_{[\nu]} | U_{[\nu]})}{\calQ_{\nu+1}^2(\Gamma_{[\nu]})}\bm{1}\left\{\sum_{i=1}^\nu \Gamma_i \leq 4kp\right\}\right]\,.
\end{equation}
For the rest of the proof, we will show that this term is $O(1)$. If this holds, then the last term in~\eqref{eq:lbgeneralterm3wts} is indeed $o(1/(n^2\log(1/p)))$ for $d \gtrsim (kp\log(1/p)\log(d/p))^2 \log n$.

We mention that a na\"ive application of Lemma~\ref{lem:capsandanticaps} is not sufficient for controlling \eqref{eq:lbgeneralsecondmoment} within our desired regime. In particular, the squared likelihood ratio $\widetilde{\calP}^2_{\nu+1} / \calQ^2_{\nu+1}$ can be as large as $\exp(\widetilde{\Omega}(kp))$, and dominating this with the bound $\exp(-d/\widetilde{\Omega}(k^2p^2))$ in Lemma~\ref{lem:capsandanticaps} and \eqref{eq:lbgeneralterm3wts} requires $d = \widetilde{\Omega}(k^3p^3)$. Our strategy is to directly bound the second moment of the likelihood ratio, instead of controlling \eqref{eq:lbgeneralsecondmoment} from the concentration. This is done by using the martingale structure of $\widetilde{\calP}_{\nu+1} / \calQ_{\nu+1}$, where we recursively bound the current step's moment with the previous step's moment bound and martingale difference's moment bound (see, e.g., \eqref{eq:lbgeneralcaprecur}). By iterating this, we show that \eqref{eq:lbgeneralsecondmoment} is at most $\exp(\widetilde{O}(k^2p^2/d))$ which is $O(1)$ for $d = \widetilde{\Omega}(k^2p^2)$.

From now on, we further condition on $\Gamma_{[\nu]}$ (with $\sum_{i=1}^\nu \Gamma_i \leq 4kp$). As notation, let $\rho$ be the uniform probability measure on $\mathbb{S}^{d-1}$. Also, for any distribution $\mu$ on $\mathbb{S}^{d-1}$, denote $\norm{\mu}_{\infty} := \sup_x |(d\mu / d\rho)(x)|$. In particular, if $\mu = \rho (\cdot \cap L) / \rho(L)$ for a measurable set $L \subseteq \mathbb{S}^{d-1}$, we have $\norm{\mu}_{\infty} = 1 / \rho(L)$. 

We begin with the following lemma, which will later characterize the concentration (and hence control the moment) of each ``martingale difference'', as will be defined shortly.

\begin{lemma}[{\cite[Corollary 4.9]{liu2022rgg}}]\label{lem:onecapconcsmallerthan1}
    There exists a constant $C_{\ref{lem:onecapconcsmallerthan1}} > 0$ such that the following holds: 
    let $\mu$ be a distribution on $\mathbb{S}^{d-1}$, and for $z \in \mathbb{S}^{d-1}$ let $X(z) := \Prob_{x \sim \mu}(\iprod{x}{z} > \tau(p, d))$. For any $0 \leq s \leq 1$,
    \[\Prob_{z \sim \rho}(|X(z) - p| > ps) \leq 2\exp\left(-\frac{ds^2}{C_{\ref{lem:onecapconcsmallerthan1}} \log(1/p)\log(d/p)(\log \norm{\mu}_{\infty} + \log(d/p))}\right).
    \]
\end{lemma}
We use the same martingale as in \cite[Observation 5.3]{liu2022rgg}. 
Given the sequence $\Gamma_i, i \in [\nu]$ from conditioning, define the following sequence of random sets $E_i, i \in [\nu]$ of \emph{$p$-caps} (when $\Gamma_i = 1$) and \emph{$p$-anticaps} (when $\Gamma_i = 0$) as
\[E_i := \begin{cases}
    \{z \in \mathbb{S}^{d-1}: \iprod{U_i}{z} \geq \tau(p, d)\} & \Gamma_i = 1, \\
    \{z \in \mathbb{S}^{d-1}: \iprod{U_i}{z} < \tau(p, d)\} & \Gamma_i = 0\,.
\end{cases}
\]
Note that among $E_i$, $\sum_{i=1}^\nu \Gamma_i$ of those are $p$-caps and $\nu - \sum_{i=1}^\nu \Gamma_i$ 
of those are $p$-anticaps. Now for $0 \leq t \leq \nu$, define the following quantity
\[R_t := \frac{\rho(L_t)}{\prod_{i=1}^t \rho(E_i)} = \frac{\rho(L_t)}{p^{\sum_{i=1}^t \Gamma_i}(1-p)^{t - \sum_{i=1}^t \Gamma_i}}\,,
\]
where $R_0 := 1$ and $L_t := \cap_{i=1}^t E_i$, with $L_0 := \mathbb{S}^{d-1}$; note that the equality follows from the definition of each $E_i$. By \cite[Observation 5.3]{liu2022rgg}, $\{R_t\}$ is a martingale with respect to the filtration $\calF_t := \sigma(\{U_1, \dots, U_t\})$. Furthermore, by definition we have
\begin{equation}\label{eq:lbgeneralRequivalence}
    \frac{\widetilde{\calP}_{\nu+1}(\Gamma_{[\nu]} | U_{[\nu]})}{\calQ_{\nu+1}(\Gamma_{[\nu]})} | \Gamma_{[\nu]} \overset{\mathrm{d}}{=} R_\nu\,.
\end{equation}
Thus, it suffices to control the second moment of $R_\nu$. This is done by factorizing $R_\nu$ as a product of (multiplicative) martingale differences, and analyzing each martingale difference using Lemma~\ref{lem:onecapconcsmallerthan1}. Note that this is different from the approach in \cite[Lemma 5.1]{liu2022rgg}, where $R_\nu$ is expressed as a sum of (additive) martingale differences. Formally, we have
\[R_\nu = R_0 \prod_{i=1}^\nu \varphi_i\,,
\]
where for $i \in [\nu]$,
\[\varphi_i := R_i / R_{i-1} = \frac{\rho(E_i \cap L_{i-1})}{\rho(E_i)\rho(L_{i-1})} \,.
\]
The next step is to build recursive inequalities between $\E[R_t^2]$ and $\E[R_{t-1}^2]$. For any $t \in [\nu]$,
\begin{equation}\label{eq:lbgeneralRrecurbase}
\begin{aligned}
    \E[R^2_t] = \E[R_{t-1}^2\E[\varphi_t^2|\calF_{t-1}]] &= \E[R_{t-1}^2 + R_{t-1}^2 \E[(\varphi_t^2-1)|\calF_{t-1}]]  \\
    &= \E[R_{t-1}^2 + R_{t-1}^2 \E[(\varphi_t-1)^2|\calF_{t-1}]]\,,
\end{aligned}
\end{equation}
where the last equality follows from $\E[\varphi_t|\calF_{t-1}] = 1$.

\paragraph*{Case 1: $\Gamma_t = 1$.} First, consider the case when $\Gamma_t = 1$, i.e., $E_t$ is a $p$-cap. If $R_{t-1} \geq a$ (the value of $a$ will be determined later), then
\[\rho(L_{t-1}) = R_{t-1}  \prod_{i=1}^{t-1} \rho(E_i) \geq ap^{\sum_{i=1}^\nu \Gamma_i}(1-p)^{\nu - \sum_{i=1}^\nu \Gamma_i}\,,
\]
which implies
\begin{equation}\label{eq:lbgeneralmuinfty}
    \norm{\mu}_{\infty} = 1/\rho(L_{t-1}) \leq 1/(ap^{\sum_{i=1}^\nu \Gamma_i}(1-p)^{\nu - \sum_{i=1}^\nu \Gamma_i})\,.
\end{equation}
Then by applying Lemma~\ref{lem:onecapconcsmallerthan1} with $\mu = \rho(\cdot | L_{t-1}) / \rho(L_{t-1})$, on $\{R_{t-1} \geq a\} \in \calF_{t-1}$ we have
\begin{equation}\label{eq:lbgeneralcapRlarge}
\begin{aligned}
    &\E[(\varphi_t-1)^2|\calF_{t-1}]\bm{1}\{R_{t-1} \geq a\}\\
    &\leq \int_{0}^{(1/p)^2} \Prob((\varphi_t-1)^2 > s|\calF_{t-1})ds \bm{1}\{R_{t-1} \geq a\} \\
    &= \int_{0}^{1/p}2u\Prob(|\varphi_t - 1| >u|\calF_{t-1})du \bm{1}\{R_{t-1} \geq a\} \\
    &\leq \left(\int_0^1 4u\exp\left(-\frac{du^2}{C_{\ref{lem:onecapconcsmallerthan1}}\log(1/p)\log(d/p)\lambda_a}\right)du + \frac{2}{p^2}\exp\left(-\frac{d}{C_{\ref{lem:onecapconcsmallerthan1}}\log(1/p)\log(d/p)\lambda_a}\right)\right) \bm{1}\{R_{t-1} \geq a\}\\
    &\leq \left(\frac{C_{\ref{lem:onecapconcsmallerthan1}}\log(1/p)\log(d/p)\lambda_a}{d} + \frac{2}{p^2}\exp\left(-\frac{d}{C_{\ref{lem:onecapconcsmallerthan1}}\log(1/p)\log(d/p)\lambda_a}\right)\right)\bm{1}\{R_{t-1} \geq a\}\,,
\end{aligned}    
\end{equation}
where $\lambda_a := \log \norm{\mu}_{\infty} + \log(d/p)$; the first inequality follows from $\varphi_t \leq 1/p$, the second inequality follows from $\Prob(|\varphi_t - 1| > u | \calF_{t-1}) = \Prob(|X(z) - p|>pu)$ in Lemma~\ref{lem:onecapconcsmallerthan1} (for $u \in [1, 1/p]$ we simply use the bound for $u = 1$). Note that by \eqref{eq:lbgeneralmuinfty},
\begin{equation}\label{eq:lbgeneralMbound}
\begin{aligned}
    \lambda_a &\leq \log(1/a) + \left(\sum_{i=1}^\nu \Gamma_i \right)\log(1/p) + \left(\nu - \sum_{i=1}^\nu \Gamma_i \right)\log(1/(1-p)) + \log(d/p) \\
    &\leq \log(1/a) +  10kp\log(1/p) + \log(d/p)=: \lambda'_a\,,
\end{aligned}
\end{equation}
from $\nu \leq 2k$ and $\sum_{i=1}^\nu \Gamma_i \leq 4kp$, along with $\log(1/(1-p)) \leq 2p$ and $\log(1/p) \geq 2/3$. In particular, while $\lambda_a$ itself depends on $t$, the last term $\lambda'_a$ in \eqref{eq:lbgeneralMbound} does not depend on $t$.

On the other hand, on $\{R_{t-1}< a\} \in \calF_{t-1}$ we have 
\begin{equation}\label{eq:lbgeneralcapRsmall}
    \E[R_t^2\bm{1}\{R_{t-1} < a\}]=\E[R_{t-1}^2\E[\varphi_t^2|\calF_{t-1}]\bm{1}\{R_{t-1} < a\}]  \leq a^2/p^2\,,
\end{equation}
from $\varphi_t \leq 1/p$.
Combining \eqref{eq:lbgeneralcapRlarge} and \eqref{eq:lbgeneralcapRsmall} for \eqref{eq:lbgeneralRrecurbase}, for $\Gamma_t = 1$ we have
\begin{equation}\label{eq:lbgeneralcaprecur}
    \E[R_t^2] = \E[R_t^2(\bm{1}\{R_{t-1} \geq a\} + \bm{1}\{R_{t-1} < a\})] \leq \E[R_{t-1}^2](1 + \eta_a) + \frac{a^2}{p^2}\,,
\end{equation}
where
\begin{equation}\label{eq:lbgeneralcapfactor}
    \eta_a := \frac{C_{\ref{lem:onecapconcsmallerthan1}}\log(1/p)\log(d/p)\lambda'_a}{d} + \frac{2}{p^2}\exp\left(-\frac{d}{C_{\ref{lem:onecapconcsmallerthan1}}\log(1/p)\log(d/p)\lambda'_a}\right)\,.
\end{equation}

\paragraph*{Case 2: $\Gamma_t = 0$.}
The bound for the case $\Gamma_t = 0$ can be obtained similarly as in the first case. On $\{R_{t-1} \geq a\}$,
\begin{align*}
    &\E[(\varphi_t-1)^2|\calF_{t-1}]\bm{1}\{R_{t-1} \geq a\}\\
    &\leq \int_{0}^{1}2u\Prob(|\varphi_t - 1| >u|\calF_{t-1})du \bm{1}\{R_{t-1} \geq a\} \\
    &\leq \int_0^{p/(1-p)} 4u\exp\left(-\frac{d(1-p)^2u^2}{C_{\ref{lem:onecapconcsmallerthan1}}p^2\log(1/p)\log(d/p)\lambda'_a}\right)du\bm{1}\{R_{t-1} \geq a\} \\
    &\quad+ \int_{p/(1-p)}^{1} 4u\exp\left(-\frac{d}{C_{\ref{lem:onecapconcsmallerthan1}}\log(1/p)\log(d/p)\lambda'_a}\right)du \bm{1}\{R_{t-1} \geq a\}\\
    &\leq \left(\frac{2C_{\ref{lem:onecapconcsmallerthan1}}p^2\log(1/p)\log(d/p)\lambda'_a}{d(1-p)^2} + 4\exp\left(-\frac{d}{C_{\ref{lem:onecapconcsmallerthan1}}\log(1/p)\log(d/p)\lambda'_a}\right)\right)\bm{1}\{R_{t-1} \geq a\} \\
    &\leq \left(\frac{8C_{\ref{lem:onecapconcsmallerthan1}}p^2\log(1/p)\log(d/p)\lambda'_a}{d} + 4\exp\left(-\frac{d}{C_{\ref{lem:onecapconcsmallerthan1}}\log(1/p)\log(d/p)\lambda'_a}\right)\right)\bm{1}\{R_{t-1} \geq a\}\,,
\end{align*} 
where the first inequality is from $|\varphi_t - 1| \leq 1 \vee p/(1-p) \leq 1$, and the second inequality is from Lemma~\ref{lem:onecapconcsmallerthan1} with $\mu = \rho(\cdot | L_{t-1}) / \rho(L_{t-1})$, where $\Prob(|\varphi_t - 1| > u | \calF_{t-1}) = \Prob(|X(z) - p|>(1-p)u)$ (applied for $u < p/(1-p)$; for $ u \in [p/(1-p), 1]$, the bound for $u = p/(1-p)$ is used); 
the final inequality is from $p \leq 1/2$. 

On $\{R_{t-1}< a\} \in \calF_{t-1}$, we have
\[\E[R_t^2\bm{1}\{R_{t-1} < a\}]=\E[R_{t-1}^2\E[\varphi_t^2|\calF_{t-1}]\bm{1}\{R_{t-1} < a\}]  \leq a^2/(1-p)^2 \leq 4a^2\,.
\]
Combining these as in \eqref{eq:lbgeneralcaprecur}, we have
\begin{equation}\label{eq:lbgeneralanticaprecur}
    \E[R_t^2] \leq \E[R_{t-1}^2](1 + \theta_a) + 4a^2\,,
\end{equation}
where
\begin{equation}\label{eq:lbgeneralanticapfactor}
    \theta_a := \frac{8C_{\ref{lem:onecapconcsmallerthan1}}p^2\log(1/p)\log(d/p)\lambda'_a}{d} + 4\exp\left(-\frac{d}{C_{\ref{lem:onecapconcsmallerthan1}}\log(1/p)\log(d/p)\lambda'_a}\right)\,.
\end{equation}
Now by successively applying \eqref{eq:lbgeneralcaprecur} and \eqref{eq:lbgeneralanticaprecur} for $t = \nu, \nu-1, \cdots, 1$, we have (recall $\E[R_0^2] = 1$)
\begin{align*}
    \E[R_\nu^2] &\leq (1+\eta_a)^{\sum_{i=1}^\nu \Gamma_i}(1+\theta_a)^{\nu - \sum_{i=1}^\nu \Gamma_i}\left(1 + \left(\sum_{i=1}^\nu \Gamma_i\right)\frac{a^2}{p^2} + \left(\nu - \sum_{i=1}^\nu \Gamma_i\right)4a^2\right) \\
    &\leq \exp\left(\left(\sum_{i=1}^\nu \Gamma_i\right)\eta_a + \left(\nu - \sum_{i=1}^\nu \Gamma_i\right)\theta_a \right)\times \left(1 + \left(\sum_{i=1}^\nu \Gamma_i\right)\frac{a^2}{p^2} + \left(\nu - \sum_{i=1}^\nu \Gamma_i\right)4a^2\right) \\
    &\leq \exp(4kp\eta_a + 2k\theta_a) \times \left(1 + \frac{4ka^2}{p} + 8ka^2\right)\,,
\end{align*}
where the second inequality is from $1 + x \leq \exp(x)$, and the last inequality is from $\nu \leq 2k, \sum_{i=1}^\nu \Gamma_i \leq 4kp$.

Finally, set $a = p/k$. Then as $\lambda'_a = \log(dk/p^2) + 10kp \log(1/p)$, we have
\[\E[R_\nu^2] \leq \exp\left(\frac{C_1kp\log(1/p)\log(d/p)\lambda'_a}{d} + \frac{C_1k}{p}\exp\left(-\frac{d}{C_1\log(1/p)\log(d/p)\lambda'_a}\right)\right) \times (1+ C_1p/k)\,,
\]
for some constant $C_1 > 0$, following from the definitions of $\eta_a$ and $\theta_a$ respectively from \eqref{eq:lbgeneralcapfactor} and \eqref{eq:lbgeneralanticapfactor}. The last term (and hence \eqref{eq:lbgeneralsecondmoment}, from \eqref{eq:lbgeneralRequivalence}) is $O(1)$, as long as $d \gtrsim kp\log(1/p) \log(d/p)\lambda'_a = kp\log(1/p)\log(d/p)(10kp\log(1/p) +\log(dk/p^2))$ and $kp \gtrsim \log n$.

Combining all lower bounds on $d$ and $kp$ for terms (I)--(III), we have that Lemma~\ref{lem:lbgeneralhighprobevent} holds if
\[kp \geq C_{\ref{lem:lbgeneralhighprobevent}} \log n \quad \text{and} \quad d \geq C_{\ref{lem:lbgeneralhighprobevent}}(kp\log(1/p)\log(d/p))^2 \log n\,,
\]
for an appropriate choice of the constant $ C_{\ref{lem:lbgeneralhighprobevent}} > 0$, as desired.

\section{Deferred proofs in Section~\ref{sec:comp}}

\subsection{Tight bounds on signed cycle counts (Proposition~\ref{prop:signedcyclecount})}\label{app:signedcyclecount}
Our proof strategy for Proposition~\ref{prop:signedcyclecount} is rather straightforward (despite involving technical calculations), starting by writing down the expectation of signed subgraph with respect to an appropriate orthonormal basis for the latents $U_1, \dots, U_{\ell}$. The orthonormality substantially simplifies the expression, reducing it to a tractable sum of combinatorial objects.

\subsubsection{Gegenbauer polynomials and spherical harmonics}
Here we present several key facts on the orthonormal basis used for the calculation. These can be found in textbooks on harmonic analysis, \eg, \cite{dai2013approximation}; for a summary, see \cite[Section 3 \& Lemma 4.10]{li2024spectralclusteringgaussianmixture}. For independent $U_1, U_2 \sim \calU(\mathbb{S}^{d-1})$, the distribution $\mu$ of $\iprod{U_1}{U_2}$ is given as
\[\mu(dx) = \frac{\Gamma(d/2)}{\Gamma((d-1)/2)\sqrt{\pi}}(1-x^2)^{(d-3)/2}dx, \quad x \in [-1, 1]\,.
\]
The polynomials $q_0, q_1, \dots$ that are orthonormal in $\mathcal{L}^2(\mu)$ are known as Gegenbauer polynomials. The first few polynomials are given as:
\[q_0(x) = 1, \quad q_1(x) = \sqrt{d}x, \quad q_2(x) = \frac{1}{\sqrt{2}}\sqrt{\frac{d+2}{d-1}}(dx^2 - 1)\,.
\]
Furthermore, these polynomials satisfy the following recursive property:
\begin{equation}\label{eq:gegenbauerrecursive}
    q_{m+1}(x) = \sqrt{\frac{(2m+d)(2m+d-2)}{(m+1)(m+d-2)}}xq_m(x) - \sqrt{\frac{m(m+d-3)(m+d/2)}{(m+1)(m+d-2)(m+d/2 - 2)}}q_{m-1}(x)
\end{equation}
The Gegenbauer polynomials admit further decomposition into spherical harmonics, which are orthonormal with respect to the distribution of $U_1$. Namely, there exists a set of functions $\{\phi_{m, t}: m \geq 0, t \in [N_m]\}$ such that
\[\E_{U_1}[\phi_{m, t}(U_1)\phi_{m', t'}(U_1)] = \bm{1}\{m = m', t = t'\}\,,
\]
and
\[q_m(\iprod{U_1}{U_2}) = \frac{1}{\sqrt{N_m}}\sum_{t \in [N_m]}  \phi_{m, t}(U_1)\phi_{m, t}(U_2)\,.
\]
Here, $N_m$ denotes the number of distinct degree-$m$ spherical harmonics (corresponding to $\mathbb{S}^{d-1}$), which satisfies
\[N_0 = 1, \quad N_m = \frac{d+2m-2}{m}\binom{d+m-3}{m-1}, m \geq 1\,.
\]
Because $\bm{1}\{\cdot \geq \tau(p, d)\} \in \calL^2(\mu)$, we can write as
\begin{equation}\label{eq:thresassphharmonics}
    \bm{1}\{\iprod{U_1}{U_2} \geq \tau(p, d)\} = \sum_{m=0}^\infty c_m q_m(\iprod{U_1}{U_2}) = \sum_{m=0}^\infty \frac{c_m}{\sqrt{N_m}}\sum_{t=1}^{N_m}  \phi_{m, t}(U_1)\phi_{m, t}(U_2)\,,
\end{equation}
where $c_m := \iprod{\bm{1}\{\cdot \geq \tau(p, d)\}}{q_m}_{\calL^2(\mu)}$.

\subsubsection{Tight bounds on signed cycle counts}
After establishing these facts, we move on to proving Proposition \ref{prop:signedcyclecount}. Recall the main quantity of interest:
\[\E_{G \sim \calG(n, p, d)}\left[\prod_{ij \in E(\mathrm{Cyc}_\ell)}(G_{ij}-p)\right]\,.
\]
First, note that any strict subgraph of $\mathrm{Cyc}_\ell$ is a forest. Thus by Corollary~\ref{cor:treesimple},
\begin{equation}\label{eq:signedcyclecountsimple}
    \E_{G \sim \calG(n, p, d)}\left[\prod_{ij \in E(\mathrm{Cyc}_\ell)}G_{ij}\right] - p^\ell = \E_{G \sim \calG(n, p, d)}\left[\prod_{ij \in E(\mathrm{Cyc}_\ell)}(G_{ij}-p)\right]\,.
\end{equation}
Then by \eqref{eq:thresassphharmonics},
\begin{align*}
    &\E_{G \sim \calG(n, p, d)}\left[\prod_{ij \in E(\mathrm{Cyc}_\ell)}G_{ij}\right] \\
    &= \E_{U_1, \dots, U_{\ell}}\left[\prod_{ij \in E(\mathrm{Cyc}_\ell)}\bm{1}\{\iprod{U_i}{U_j} \geq \tau(p, d)\}\right] \\
    &= \E_{U_1, \dots, U_{\ell}}\left[\prod_{ij \in E(\mathrm{Cyc}_\ell)}\left(\sum_{m=0}^\infty \frac{c_m}{\sqrt{N_m}}\sum_{t \in [N_m]}\phi_{m, t}(U_i)\phi_{m, t}(U_j)\right)\right] \\
    &= \sum_{m_1, \dots, m_{\ell} \geq 0} \frac{c_{m_1}\dots c_{m_{\ell}}}{\sqrt{N_{m_1} \dots N_{m_{\ell}}}} \sum_{t_1 \in [N_{m_1}], \dots, t_{\ell} \in [N_{m_{\ell}}]} \E_{U_1, \dots, U_{\ell}}[\phi_{m_1, t_1}(U_1)\phi_{m_1, t_1}(U_2) \dots \phi_{m_{\ell}, t_{\ell}}(U_{\ell})\phi_{m_{\ell}, t_{\ell}}(U_1)]\,,
\end{align*}
where each summand satisfies (by the orthonormality of spherical harmonics)
\begin{align*}
    &\E_{U_1, \dots, U_{\ell}}[\phi_{m_1, t_1}(U_1)\phi_{m_1, t_1}(U_2) \dots \phi_{m_{\ell}, t_{\ell}}(U_{\ell})\phi_{m_{\ell}, t_{\ell}}(U_1)] \\
    &= \E_{U_1}[\phi_{m_{\ell}, t_{\ell}}(U_1)\phi_{m_1, t_1}(U_1)] \dots \E_{U_{\ell}}[\phi_{m_{\ell -1}, t_{\ell -1}}(U_{\ell})\phi_{m_{\ell}, t_{\ell}}(U_{\ell})] \\
    &= \bm{1}\{m_{\ell} = m_1, t_{\ell} = t_1\} \cdots \bm{1}\{m_{\ell -1} = m_{\ell}, t_{\ell -1} = t_{\ell}\} \\
    &= \bm{1}\{m_1 = \dots = m_{\ell}, t_1 = \dots = t_{\ell}\}\,.
\end{align*}
Thus, we obtain
\[\E_{G \sim \calG(n, p, d)}\left[\prod_{ij \in E(\mathrm{Cyc}_\ell)}G_{ij}\right] = \sum_{m=0}^\infty \frac{c_m^{\ell}}{N_m^{\ell /2}}N_m = \sum_{m=0}^\infty \frac{c_m^{\ell}}{N_m^{\ell /2-1}}\,.
\]
Since $c_0 = \iprod{\bm{1}\{\cdot \geq \tau(p, d)\}}{q_0}_{\calL^2(\mu)} = \Prob(\iprod{U_1}{U_2} \geq \tau(p, d)) = p$ and $N_0 = 1$, from \eqref{eq:signedcyclecountsimple} we have
\begin{equation}\label{eq:signedcyclecountinfinitesum}
    \E_{G \sim \calG(n, p, d)}\left[\prod_{ij \in E(\mathrm{Cyc}_\ell)}(G_{ij}-p)\right] = \sum_{m=1}^\infty \frac{c_m^{\ell}}{N_m^{\ell /2-1}}\,.
\end{equation}
Thus, it suffices to only bound the size of $c_m$ and $N_m$ appearing from Gegenbauer polynomials. This is done by decomposing the infinite sum depending on the value of $m$. In particular, we claim that for some constant $C > 0$,
\begin{align}
    \frac{1}{C^{\ell}} \frac{p^{\ell} \log^{\ell /2}(1/p)}{d^{\ell /2-1}} \leq \frac{c_1^{\ell}}{N_1^{\ell /2-1}} &\leq C^{\ell} \frac{p^{\ell} \log^{\ell /2}(1/p)}{d^{\ell /2-1}}\,,\label{eq:signedcyclecountleadingterm}\\
    \left|\sum_{m=2}^{\lfloor d^{1/4}\rfloor} \frac{c_m^{\ell}}{N_m^{\ell /2-1}}\right| &= o\left(\frac{p^{\ell}\log^{\ell /2}(1/p)}{C^{\ell}d^{\ell /2-1}}\right)\,, \label{eq:signedcyclecountsmallm} \\
    \left|\sum_{m=\lfloor d^{1/4} \rfloor + 1} ^ \infty \frac{c_m^{\ell}}{N_m^{\ell /2-1}}\right| &= o\left(\frac{p^{\ell}}{C^{\ell}d^{\ell /2-1}}\right)\,, \label{eq:signedcyclecountlargem}
\end{align}
where the asymptotics from now on are with respect to $d \to \infty$, uniformly over $3 \leq \ell \leq n$; clearly, Proposition~\ref{prop:signedcyclecount} follows from those. For \eqref{eq:signedcyclecountleadingterm} through \eqref{eq:signedcyclecountlargem}, we will use $C_0, C_1, \dots$ to denote absolute constants within each case.

\paragraph*{Leading term ($m = 1$).} For $m = 1$, we have $N_1 = d$ and
\[c_1 = \int_{\tau(p, d)}^1 \sqrt{d}x\frac{\Gamma(d/2)}{\Gamma((d-1)/2)\sqrt{\pi}}(1-x^2)^{(d-3)/2}dx =  \frac{\sqrt{d} \Gamma(d/2)}{\Gamma((d-1)/2)(d-1)\sqrt{\pi}}(1-\tau(p, d)^2)^{(d-1)/2}\,.
\]
We show that $c_1$ (equivalently, $(1-\tau(p, d)^2)^{(d-1)/2}$) is in the order of $p\sqrt{\log(1/p)}$. This is done by dividing into two cases, where $p$ is small (close to $0$) or large (bounded away from $0$). For small $p$, by \cite[Lemma 2]{bubeck2016rgg} there exists $C_0 > 0$ such that for all $p \in (0, 0.49]$,
\[\frac{1}{C_0}\sqrt{\frac{\log(1/p)}{d}} \leq \tau(p, d) \leq C_0\sqrt{\frac{\log(1/p)}{d}}\,.
\]
In particular, there exists $\eps_0 > 0$ such that for all $p \in (0, \eps_0]$, $\tau(p, d) \geq (1/C_0)\sqrt{\log(1/p)/d} \geq \sqrt{2/d}$. Then by \cite[Lemma 2.1(b)]{brieden2001cap},
\[\frac{2}{C_0}p\sqrt{\log(1/p)}\leq 2p\sqrt{d}\tau(p, d)\leq (1-\tau(p, d)^2)^{(d-1)/2} \leq 6p\sqrt{d}\tau(p, d) \leq 6C_0p\sqrt{\log(1/p)}\,,
\]
implying that $c_1$ is in the order of $p\sqrt{\log(1/p)}$ when $p \in (0, \eps_0]$. For $p \in (\eps_0, 0.5]$, as $p\sqrt{\log(1/p)}$ is of constant order, it suffices to show that $(1-\tau(p, d)^2)^{(d-1)/2}$ is of constant order for the same range of $p$ (uniformly over $d$). \cite[Lemma 3.6]{liu2022rgg} states 
\begin{equation}\label{eq:thresub}
    \tau(p, d) \leq \sqrt{\frac{3\log(1/p)}{d}} \quad \text{for all} \quad 0 < p \leq 0.5\,.
\end{equation}
Thus for all sufficiently large $d$,
\[p^2 \leq (1-\tau(p, d)^2)^{(d-1)/2} \leq 1\,.
\]
This shows that $(1-\tau(p, d)^2)^{(d-1)/2}$ is of constant order over $p \in (\eps_0, 0.5]$. Thus, $c_1$ is in the order of $p\sqrt{\log(1/p)}$ which proves \eqref{eq:signedcyclecountleadingterm}.

\paragraph*{Small $m$ ($2 \leq m \leq d^{1/4}$).} As a brief overview, for this regime of $m$ we will control the size of $c_m$ and $N_m$ in \eqref{eq:signedcyclecountinfinitesum} inductively using their recursive definitions. For this, we first consider an upper bound $\widetilde{c}_m \geq 0$ of $|c_m|$, defined as
\[
\widetilde{c}_m :=  \int_{\tau(d, p)}^1 |q_m(x)| \mu(dx)\,.
\]
Our first claim is that there exists a constant $C_0 > 0$ such that
\begin{equation}\label{eq:signedcyclecountsmallm1}
    \widetilde{c}_m \leq C_0p\sqrt{\log(1/p)}(0.9d)^{(m-1)/8}\,,
\end{equation}
for all $m \geq 1$. This can be proved by induction; for $m = 1$, this follows from $\widetilde{c}_1 = c_1$ and the corresponding analysis for $c_1$ (in the proof for \eqref{eq:signedcyclecountleadingterm}), and for $m = 2$ this follows from
\begin{align*}
    \widetilde{c}_2 &\leq \int_{\tau(p, d)}^1(dx^2 + 1)\frac{\Gamma(d/2)}{\Gamma((d-1)/2)\sqrt{\pi}}(1-x^2)^{(d-3)/2}dx \\
    &= \frac{d}{d-1}\tau(p, d) \times \frac{\Gamma(d/2)}{\Gamma((d-1)/2)\sqrt{\pi}}(1-\tau(p, d)^2)^{(d-1)/2} \\
    &\quad+ \frac{d}{d-1} \times \int_{\tau(p, d)}^1 \frac{\Gamma(d/2)}{\Gamma((d-1)/2)\sqrt{\pi}}(1-x^2)^{(d-1)/2}dx  +  p \\
    &\leq \sqrt{d}\tau(p, d)c_1 + 3p \leq C_0p\sqrt{\log(1/p)}\left(\sqrt{d}\tau(p, d) + \frac{3}{C_0 \sqrt{\log(1/p)}}\right)\,,
\end{align*}
where the first equality is from integration by parts, second inequality is from $d/(d-1) \leq 2$, $(1-x^2 )^{(d-1)/2} \leq (1-x^2)^{(d-3)/2}$ and the definition of $c_1$, and the final inequality is from \eqref{eq:signedcyclecountsmallm1} for $m = 1$. From \eqref{eq:thresub} and $d \geq (5\log(1/p))^4$, the factor $\sqrt{d}\tau(p, d) + 3/(C_0 \sqrt{\log(1/p)})$ is smaller than $(0.9d)^{1/8}$ for all sufficiently large $d$. 

Now we use the recursive definition of the Gegenbauer polynomials to construct a recursive inequality for $\widetilde{c}_m$. It can be shown that the factors in \eqref{eq:gegenbauerrecursive} satisfy
\[\sqrt{\frac{(2m+d)(2m+d-2)}{(m+1)(m+d-2)}} \leq \sqrt{d} \quad \text{and} \quad \sqrt{\frac{m(m+d-3)(m+d/2)}{(m+1)(m+d-2)(m+d/2 - 2)}} \leq 1 \,,
\]
for all $d \geq 4$, $m \geq 1$. This implies

\[|q_{m+1}(x)| \leq \sqrt{d}x|q_m(x)| + |q_{m-1}(x)|\,,
\]
for all $x \in [\tau(p, d), 1]$. Thus,
\begin{align*}
    \widetilde{c}_{m+1} &\leq \sqrt{d}\int_{\tau(p, d)}^1 x|q_m(x)| \mu(dx) +\int_{\tau(p, d)}^1 |q_{m-1}(x)| \mu(dx) \\
    &= \int_{\tau(p, d)}^{0.9/d^{3/8}} \sqrt{d}x|q_m(x)| \mu(dx) + \int_{0.9/d^{3/8}}^1 \sqrt{d}x|q_m(x)| \mu(dx) + \tilde{c}_{m-1} \\
    &\leq 0.9d^{1/8} \int_{\tau(p, d)}^1 |q_m(x)|\mu(dx) + \left(\int_{0.9/d^{3/8}}^1 dx^2 \mu(dx)\right)^{1/2} \left(\int_{0.9/d^{3/8}}^1 (q_m(x))^2 \mu(dx)\right)^{1/2} + \widetilde{c}_{m-1}\\
    &\leq 0.9d^{1/8} \widetilde{c}_m +  \left(\int_{0.9/d^{3/8}}^1 d \mu(dx)\right)^{1/2}  + \widetilde{c}_{m-1} \leq 0.9d^{1/8} \widetilde{c}_m + 2\sqrt{d}\exp\left(-\frac{(0.9)^2}{4}d^{1/4} \right)+ \tilde{c}_{m-1}\,,
\end{align*}
where the second inequality is by Cauchy-Schwarz, third inequality is from $\norm{q_m}_{\calL^2(\mu)} = 1$, and the last inequality is from concentration of spherical cap \cite[Theorem 14.1.1]{matousek2013lectures}.

Now we apply the induction hypothesis; suppose that \eqref{eq:signedcyclecountsmallm1} holds for $m$ and $m-1$, where $m \geq 2$. Then 
\begin{align*}
    \widetilde{c}_{m+1} &\leq 0.9d^{1/8}\times C_0p\sqrt{\log(1/p)}(0.9d)^{(m-1)/8} + 2\sqrt{d}\exp\left(-\frac{(0.9)^2}{4}d^{1/4} \right) + C_0 p\sqrt{\log(1/p)}(0.9d)^{(m-2)/8} \\
    &\leq C_0p\sqrt{\log(1/p)}(0.9d)^{(m-2)/8}(0.9d^{1/4} + 2) \leq C_0p\sqrt{\log(1/p)}(0.9d)^{m/8}\,,
\end{align*}
where the second inequality is from $2\sqrt{d}\exp(-(0.9)^2d^{1/4}/4) \leq \exp(-d^{1/4}/5) \leq p$ which is smaller than $C_0p\sqrt{\log(1/p)}(0.9d)^{(m-2)/8}$, and the final inequality follows from $(0.9d)^{1/4} \geq 0.9d^{1/4} + 2$ for all sufficiently large $d$. This proves the first claim \eqref{eq:signedcyclecountsmallm1}.

Our second claim is that 
\begin{equation}\label{eq:signedcyclecountsmallm2}
    N_m \geq \frac{d^2}{2} \times (0.9d^3)^{(m-2)/4}\,,
\end{equation}
for all $2 \leq m \leq d^{1/4}$. This can also be shown by induction, as $N_2 = \frac{d+2}{2}\binom{d-1}{1} \geq \frac{d^2}{2}$ and
\[\frac{N_m}{N_{m-1}} = \frac{(d+2m-2)(d+m-3)}{m(d+2m-4)} \geq \frac{d+2m-2}{m} \times (0.9)^{1/4} \geq (0.9d^3)^{1/4}\,.
\]
Combining \eqref{eq:signedcyclecountsmallm1} and \eqref{eq:signedcyclecountsmallm2}, we have
\begin{align*}
    \left|\sum_{2 \leq m \leq d^{1/4}} \frac{c_m^{\ell}}{N_m^{\ell /2 - 1}}\right| &\leq \sum_{2 \leq m \leq d^{1/4}} \frac{(\widetilde{c}_m)^{\ell}}{N_m^{\ell /2 - 1}} \\
    &\leq \frac{2^{\ell/2-1}(C_0p)^{\ell}\log^{\ell /2}(1/p)}{d^{\ell -2}}(0.9d)^{\ell/8} \sum_{m \geq 0} \frac{(0.9d)^{\ell m/8}}{(0.9d^3)^{(\ell/2 - 1)(m/4)}} \\
    &= \frac{2^{\ell/2-1}(C_0p)^{\ell}\log^{\ell /2}(1/p)}{d^{\ell -2}}(0.9d)^{\ell/8} \sum_{m \geq 0} \left(\frac{1}{d^{1/2}}\right)^{(\ell/2-1)m} (0.9d)^{m/4} \\
    &\leq \frac{2^{\ell/2-1}(C_0p)^{\ell}\log^{\ell /2}(1/p)}{d^{\ell -2}}(0.9d)^{\ell/8} \sum_{m \geq 0} \left(\frac{1}{d^{1/2}}\right)^{m/2} (0.9d)^{m/4} \\
    &\leq \frac{2^{\ell/2-1}(C_0p)^{\ell}\log^{\ell /2}(1/p)}{d^{\ell -2}}(0.9d)^{\ell/8} \sum_{m \geq 0} (0.9)^{m/4} \\
    &\leq \frac{100(2C_0p)^{\ell}\log^{\ell /2}(1/p)}{d^{7\ell/8 - 2}}\,,
\end{align*}
where for the third and the last inequality we use $\ell \geq 3$. Let $C > 0$ be the constant in \eqref{eq:signedcyclecountleadingterm}. Then 
\[\frac{100(2C_0p)^{\ell}\log^{\ell /2}(1/p)}{d^{7\ell/8 - 2}} = \frac{p^{\ell} \log^{\ell /2}(1/p)}{C^{\ell} d^{\ell /2-1}} \times \frac{100(2C_0 C)^{\ell}}{d^{3\ell/8 - 1}}\,,
\]
where the factor $100(2C_0 C)^{\ell} / d^{3\ell/8 - 1}$ is uniformly of order $o(1)$ for all $\ell \geq 3$ as $d \to \infty$. This proves \eqref{eq:signedcyclecountsmallm}.

\paragraph*{Large $m$ ($m > d^{1/4}$).} %
For this regime of $m$, instead of recursively bounding the size of $c_m$ we directly invoke approximation results for Gegenbauer polynomials. This results in a rather complicated expression in Gamma functions. As $N_m$ can also be written in terms of Gamma functions, we invoke Stirling's approximation for bounding these and carefully bound the resulting summands in \eqref{eq:signedcyclecountinfinitesum} to show that the remaining infinite sum is small.

To begin, we have

\begin{align*}
    |c_m| &\leq \int_{\tau(p, d)}^1 |q_m(x)| \frac{\Gamma(d/2)}{\Gamma((d-1)/2)\sqrt{\pi}}(1-x^2)^{(d-3)/2}dx \\
    &= \frac{\Gamma(d/2)}{\Gamma((d-1)/2)\sqrt{\pi}} \int_{\tau(p, d)}^1 \sqrt{N_m} \frac{\Gamma(m+1)\Gamma((d-1)/2)}{\Gamma((d-1)/2 + m)}|P_m^{(d-3)/2, (d-3)/2}(x)|(1-x^2)^{(d-3)/2}dx \\
    &= \frac{\sqrt{N_m}\Gamma(d/2)\Gamma(m+1)}{\Gamma((d-1)/2 + m)\sqrt{\pi}} \int_{\tau(p, d)}^1 |P_m^{(d-3)/2, (d-3)/2}(x)|(1-x^2)^{(d-3)/2}dx \\
    &\leq \frac{\sqrt{N_m}\Gamma(d/2)\Gamma(m+1)}{\Gamma((d-1)/2 + m)\sqrt{\pi}} \int_{\tau(p, d)}^1 \sqrt{3}\left(\frac{d-3}{2}\right)^{1/6}\left(1+\frac{d-3}{2m}\right)^{1/12}(1-x^2)^{d/4 - 1}dx \\
    &\leq \frac{\sqrt{N_m}\Gamma(d/2)\Gamma(m+1)}{\Gamma((d-1)/2 + m)} d^{1/4}\int_{\tau(p, d)}^1 (1-x^2)^{d/4 - 1}dx \\
    &\leq \frac{\sqrt{N_m}\Gamma(d/2)\Gamma(m+1)}{\Gamma((d-1)/2 + m)} d^{1/4} \times \frac{C_0}{\sqrt{d}} \leq  \frac{C_0\sqrt{N_m}\Gamma(d/2)\Gamma(m+1)}{\Gamma((d-1)/2 + m)d^{1/4}} \times p\exp\left(\frac{1}{5}d^{1/4}\right) \,,
\end{align*}
for some constant $C_0 > 0$. 
The first equality is from \cite[Equation B.2.1]{dai2013approximation},\footnote{The polynomial $P_m^{(d-3)/2, (d-3)/2}$ is called Jacobi polynomial, defined as 
\[P_m^{(d-3)/2, (d-3)/2}(x) := \frac{(-1)^m}{2^m m!}(1-x^2)^{-(d-3)/2}\frac{d^m}{dx^m}(1-x^2)^{m+(d-3)/2}\,.
\]
See also \cite[Claim 4.3]{li2024spectralclusteringgaussianmixture}.} the second inequality is from \cite[Theorem 2]{krasikov2007jacobi}, and the third inequality is from $\sqrt{3/\pi}((d-3)/2)^{1/6}(1+(d-3)/(2m))^{1/12} \leq d^{1/4}$; the next two inequalities are respectively from $\int_{-1}^1 (1-x^2)^{d/4 - 1} dx \lesssim 1/\sqrt{d}$ and $d \geq (5\log(1/p))^4$. Thus,
\begin{equation}\label{eq:signedcyclecountdandmterm}
\begin{aligned}
    \frac{d^{\ell /2-1}}{(C_0p)^{\ell}} \times \frac{|c_m|^{\ell}}{N_m^{\ell /2-1}}  = \left(\frac{|c_m|}{C_0p\sqrt{N_m}}\right)^{\ell} d^{\ell /2-1} N_m\leq \frac{ \Gamma(d/2)^{\ell} \Gamma(m+1)^{\ell}}{\Gamma((d-1)/2 + m)^{\ell }}\exp\left(\frac{\ell}{5} d^{1/4} + \frac{\ell}{4}\log d\right) \times N_m\,.
\end{aligned}
\end{equation}
Note that the main summand of interest $|c_m|^{\ell}/N_m^{\ell /2-1}$ is already normalized here; for \eqref{eq:signedcyclecountlargem}, it suffices to show that \eqref{eq:signedcyclecountdandmterm} is $o(1/(C_0C)^{\ell})$ uniformly over $d$.

Since we need to bound the sum for all $m > d^{1/4}$, within the right hand side we collect the factors that depend on $m$, namely,
\begin{equation}\label{eq:signedcyclecountmterm1}
\begin{aligned}
    &\frac{\Gamma(m+1)^{\ell}}{\Gamma((d-1)/2+m)^{\ell}} \times N_m \\
    &= \frac{\Gamma(m+1)^{\ell}}{\Gamma((d-1)/2+m)^{\ell}} \times \frac{d+2m-2}{m} \times \frac{\Gamma(d+m-2)}{\Gamma(m)\Gamma(d-1)} \\
    &\leq C_1^{\ell} \frac{(m+1)^{(m+1/2)\ell}}{((d-1)/2 + m)^{(d/2 + m - 1)\ell}} \exp((d-3)\ell/2) \times \frac{d+2m-2}{m} \times \frac{(d+m-2)^{d+m-5/2}}{m^{m-1/2}(d-1)^{d-3/2}}\exp(1) \\
    &\leq C_2^{\ell} \frac{\exp((d-3)\ell/2)}{(d-1)^{d-3/2}} \times \frac{(m+1)^{(m+1/2)\ell}(d+m-2)^{d+m-5/2}}{((d-1)/2 + m)^{(d/2+m-1)\ell-1}m^{m+1/2}} \\
    &\leq C_2^{\ell} \frac{\exp((d-3)\ell/2)2^d}{(d-1)^{d-3/2}} \times \frac{(m+1)^{(m+1/2)\ell}(d+m-2)^{m-5/2}}{((d-1)/2 + m)^{(d/2+m-1)\ell-1 - d}m^{m+1/2}}\,,
\end{aligned}
\end{equation}
for some constants $C_1, C_2 > 0$ where the first inequality is from Stirling's approximation, second inequality is from $d+2m-2 \leq 2((d-1)/2 + m)$, and the third inequality is from $d+m-2\leq 2((d-1)/2 + m)$.
By combining
\begin{align*}
    \frac{(m+1)^{m-5/2}(m+d-2)^{m-5/2}}{(m + (d-1)/2)^{2m-5}} &\leq 1\,, \\
    \frac{(m+1)^{2m+1}}{m^{m+1/2}(m+(d-1)/2)^{m+1/2}} & \leq 1\,, \\
    \frac{(m+1)^{(m+1/2)\ell - (3m - 3/2)}}{(m + (d-1)/2)^{(m+1/2)\ell - (3m - 3/2)}} & \leq 1\,,
\end{align*}
(note that for the last inequality we use $(m+1/2)\ell - (3m - 3/2) \geq 0$, which holds whenever $\ell \geq 3$) the last term of \eqref{eq:signedcyclecountmterm1} is upper bounded by
\begin{align*}
    C_2^{\ell} \frac{\exp((d-3)\ell/2)2^d}{(d-1)^{d-3/2}} \times \frac{1}{(m + (d-1)/2)^{(\ell/2 - 1)d - 3\ell/2 + 2}}\,.
\end{align*}
Going back to \eqref{eq:signedcyclecountdandmterm} and summing over $m > d^{1/4}$, we obtain
\begin{align*}
    &\frac{d^{\ell /2-1}}{(C_0p)^{\ell}} \sum_{m > d^{1/4}} \frac{|c_m|^{\ell}}{N_m^{\ell /2-1}} \\
    &\leq \Gamma(d/2)^{\ell} \exp\left(\frac{\ell}{5}d^{1/4} + \frac{\ell}{4}\log d\right) \times \sum_{m > d^{1/4}}\frac{\Gamma(m+1)^{\ell}}{\Gamma((d-1)/2 + m)^{\ell}}N_m \\
    &\leq \Gamma(d/2)^{\ell} \exp\left(\frac{\ell}{5}d^{1/4} + \frac{\ell}{4}\log d\right)\times C_2^{\ell} \frac{\exp((d-3)\ell/2)2^d}{(d-1)^{d-3/2}}  \times \sum_{m > d^{1/4}}\frac{1}{(m + (d-1)/2)^{(\ell/2 - 1)d - 3\ell/2 + 2}} \\
    &\leq \Gamma(d/2)^{\ell} \exp\left(\frac{\ell}{5}d^{1/4} + \frac{\ell}{4}\log d\right)\times C_2^{\ell} \frac{\exp((d-3)\ell/2)2^d}{(d-1)^{d-3/2}}  \times \frac{1}{((d + 2d^{1/4} - 3)/2)^{(\ell/2-1)d - 3\ell/2 + 1}}\,,
\end{align*}
where the last inequality is obtained by comparing the sum with integral (here we use that $(\ell/2-1)d > 3\ell/2$, which holds for all sufficiently large $d$ regardless of $\ell \geq 3$).

Now that the last line only depends on $d$ and $\ell$, from $\Gamma(d/2) \lesssim (d/2)^{(d-1)/2}\exp(-d/2)$ the last line is further upper bounded by
\begin{align*}
    &C_3^{\ell} \left(\frac{d}{2}\right)^{(d-1)\ell/2}\exp(-d\ell/2)\exp\left(\frac{\ell}{5}d^{1/4}  + \frac{\ell}{4}\log d\right)\frac{\exp((d-3)\ell/2)2^d}{(d-1)^{d-3/2}}\frac{1}{((d + 2d^{1/4} - 3)/2)^{(\ell/2-1)d - 3\ell/2 + 1}} \\
    &\leq C_3^{\ell} \left(\frac{d}{d+2d^{1/4}-3}\right)^{(\ell/2-1)d} \left(\frac{d}{2}\right)^{d - \ell/2}d^{3\ell/2-1} \exp\left(\frac{\ell}{5}d^{1/4}  + \frac{\ell}{4}\log d\right) \exp(-3\ell/2)d^{3/2}\left(\frac{2}{d-1}\right)^d \\
    &\leq C_3^{\ell} \exp\left(-\frac{8}{5}\left(\frac{\ell}{2}-1\right)d^{1/4}\right)\left(\frac{d}{d-1}\right)^{d} (2d)^{\ell  + 1/2}\exp\left(\frac{\ell}{5}d^{1/4}  + \frac{\ell}{4}\log d\right) \\
    &\leq \exp\left(-\left(\frac{3}{5}\ell - \frac{8}{5}\right)d^{1/4} + 2l\log d\right)\,,
\end{align*}
for some constant $C_3 > 0$; the last inequality holds as long as $d$ is sufficiently large. In particular, for the constant $C > 0$ in \eqref{eq:signedcyclecountleadingterm},
\[\exp\left(-\left(\frac{3}{5}\ell - \frac{8}{5}\right)d^{1/4} + 2l\log d\right)(C_0C)^{\ell}
\]
is of order $o(1)$ uniformly over all $\ell \geq 3$. This implies
\[\frac{d^{\ell /2-1}}{(C_0p)^{\ell}} \sum_{m > d^{1/4}} \frac{|c_m|^{\ell}}{N_m^{\ell /2-1}} \times (C_0 C)^{\ell} = \frac{d^{\ell /2-1}C^{\ell}}{p^{\ell}} \sum_{m > d^{1/4}} \frac{|c_m|^{\ell}}{N_m^{\ell /2-1}} = o(1)\,,
\]
uniformly over $\ell \geq 3$. This proves \eqref{eq:signedcyclecountlargem}.

\end{document}